\documentclass[11pt,twoside]{amsart}

\makeatletter
\@namedef{subjclassname@2020}{%
  \textup{2020} Mathematics Subject Classification}
\makeatother

\DeclareRobustCommand{\SkipTocEntry}[4]{}
\makeatletter
\def\@tocline#1#2#3#4#5#6#7{\relax
  \ifnum #1>\c@tocdepth 
  \else
    \par \addpenalty\@secpenalty\addvspace{#2}%
    \begingroup \hyphenpenalty\@M
    \@ifempty{#4}{%
      \@tempdima\csname r@tocindent\number#1\endcsname\relax
    }{%
      \@tempdima#4\relax
    }%
    \parindent\z@ \leftskip#3\relax \advance\leftskip\@tempdima\relax
    \rightskip\@pnumwidth plus4em \parfillskip-\@pnumwidth
    #5\leavevmode\hskip-\@tempdima
      \ifcase #1
       \or\or \hskip 1em \or \hskip 2em \else \hskip 3em \fi%
      #6\nobreak\relax
    \dotfill\hbox to\@pnumwidth{\@tocpagenum{#7}}\par
    \nobreak
    \endgroup
  \fi}
\makeatother

\makeatletter
\def\subsubsection{\@startsection{subsubsection}{3}%
  \z@{.5\linespacing\@plus.7\linespacing}{-.5em}%
  {\normalfont\bfseries}}
 \makeatother

\usepackage[all]{xy}
\usepackage{amsfonts}
\usepackage{amsmath}
\usepackage{amssymb}
\usepackage{amsthm}
\usepackage{xcolor}
\usepackage{enumitem}
\usepackage{hyperref}
\usepackage{fullpage}
\usepackage{graphicx}
\usepackage{mathtools}
\usepackage{supertabular}
\usepackage{diagbox}
\usepackage{url}

\setlist[enumerate,1]{label=(\roman*)}
\setlist[enumerate,2]{label=(\alph*)}
\setlist[enumerate,3]{label=(\Roman*)}
\setlist[enumerate,4]{label=(\Alph*)}

\theoremstyle{definition}
\newtheorem{defn}{Definition}[section]

\newtheorem{rmk}[defn]{Remark}
\newtheorem{question}[defn]{Question}
\theoremstyle{plain}
\newtheorem{thm}[defn]{Theorem}

\newtheorem{lem}[defn]{Lemma}
\newtheorem{prop}[defn]{Proposition}
\newtheorem{cor}[defn]{Corollary}

\newtheorem{conj}[defn]{Conjecture}

\def\C{\ensuremath{\mathbb{C}}}
\def\D{\ensuremath{\mathbb{D}}}

\def\H{\ensuremath{\mathbb{H}}}

\def\P{\ensuremath{\mathbb{P}}}
\def\Q{\ensuremath{\mathbb{Q}}}
\def\R{\ensuremath{\mathbb{R}}}
\def\Z{\ensuremath{\mathbb{Z}}}

\def\CC{\ensuremath{\mathcal C}}

\def\FF{\ensuremath{\mathcal F}}

\def\HH{\ensuremath{\mathcal H}}
\def\II{\ensuremath{\mathcal I}}

\def\OO{\ensuremath{\mathcal O}}

\def\TT{\ensuremath{\mathcal T}}

\def\ch{\mathop{\mathrm{ch}}\nolimits}

\def\Coh{\mathop{\mathrm{Coh}}\nolimits}

\def\Db{\mathop{\mathrm{D}^{\mathrm{b}}}\nolimits}
\def\deg{\mathop{\mathrm{deg}}\nolimits}

\def\dim{\mathop{\mathrm{dim}}\nolimits}

\def\ext{\mathop{\mathrm{ext}}\nolimits}
\def\Ext{\mathop{\mathrm{Ext}}\nolimits}
\def\lExt{\mathop{\mathcal Ext}\nolimits}

\def\Gr{\mathop{\mathrm{Gr}}\nolimits}

\def\hom{\mathop{\mathrm{hom}}\nolimits}
\def\Hom{\mathop{\mathrm{Hom}}\nolimits}

\def\RlHom{\mathop{\mathbf{R}\mathcal Hom}\nolimits}
\def\RHom{\mathop{\mathbf{R}\mathrm{Hom}}\nolimits}

\def\mod{\mathop{\mathrm{mod}}\nolimits}
\def\min{\mathop{\mathrm{min}}\nolimits}

\def\into{\ensuremath{\hookrightarrow}}
\def\onto{\ensuremath{\twoheadrightarrow}}

\begin{document}

\title{Sheaves of low rank in three dimensional projective space}
\author{Benjamin Schmidt}
\address{Gottfried Wilhelm Leibniz Universit\"at Hannover, Institut f\"ur Algebraische Geometrie, Welfengarten 1, 30167 Hannover, Germany}
\email{bschmidt.contact@gmail.com}
\urladdr{https://benjaminschmidt.github.io}



\begin{abstract}
We classify Chern characters of semistable sheaves up to rank four in three dimensional projective space. As a corollary we show that moduli spaces of semistable sheaves between rank zero and four with maximal third Chern character are smooth and irreducible.
\end{abstract}

\maketitle

\tableofcontents


\section{Introduction}

In \cite{DP85:stable_p2} Dr\'ezet and Le Potier classified the numerical invariants of Gieseker-semistable sheaves in $\P^2$. The answer is surprisingly complex, but beautiful. It involves a fractal curve. On one hand, the problem has been successfully studied on some special surfaces, and the topic remains an active area of research. On the other hand, there have been virtually no results in this direction for threefolds. It turns out that semistable sheaves on threefolds are generally much more difficult to study. Even in the case of $\P^3$ results are scarce. In this article, we will solve the problem for slope-semistable sheaves up to rank four on $\P^3$. 

For $r \in \Z_{> 0}$ and $c \in \Z$ we define $D(r, c)$ to be the supremum over all $d \in \tfrac{1}{2} \Z$ such that there is a slope-semistable sheaf $E$ with rank $r$, $c_1(E) = c$, and $\ch_2(E) = d$. The Bogomolov inequality shows that this is a finite number. It is not difficult to show there are slope-semistable sheaves for not just $d = D(r, c)$, but for $d < D(r, c)$ as well (see Proposition \ref{prop:function_D} for details). Therefore, determining $D(r, c)$ is the same as classifying possible second Chern characters of slope-semistable sheaves.

\begin{thm}
\label{thm:main_ch_2}
The first values of $D(r, c)$ are given by
\begin{enumerate}
    \item $D(r, 0) = 0$ for all $r > 0$,
    \item $D(2, -1) = -\tfrac{1}{2}$,
    \item $D(3, -1) = -\tfrac{1}{2}$,
    \item $D(4, -1) = -\tfrac{3}{2}$, $D(4, -2) = -1$.
\end{enumerate}
\end{thm}

Sheaves for $d = D(r, c)$ for $r \leq 4$ can actually be chosen to be Gieseker-semistable instead of just slope-semistable. Using tensor products with line bundles and the dual, one can obtain all $D(r, c)$ for $r \leq 4$ from these values (see again Proposition \ref{prop:function_D} for details).

On $\P^2$ the bounds by Dr\'ezet and Le Potier behave well with respect to scaling, i.e., one can generally obtain a sharp bound of $\ch_2(E)/r(E)$ in terms of the slope $\mu(E) = c_1(E)/r(E)$. This begs the question whether it holds here as well. More precisely, for $\mu \in \Q$, let $\widetilde{D}(\mu)$ be defined as the supremum over all $D(r, c)/r$ such that $r/c = \mu$.

\begin{question}
\label{q:scaling}
Is there a slope-semistable sheaf $E$ of rank $r$, $c_1(E) = c$, and $\ch_2(E) = d$ if and only if $d/r \leq \widetilde{D}(\mu(E))$? Is $\widetilde{D}: \Q \to \Q$ continuous?
\end{question}

Let $r \in \Z_{> 0}$, $c \in \Z$, and $d \in \tfrac{1}{2} \Z$ with $d \leq D(r, c)$. Assume further that for $c$ even we have $d \in \Z$, and for $c$ odd we have $d \notin \Z$. Then $E(r, c, d)$ is defined to be the supremum over all $e \in \tfrac{1}{6} \Z$ such that there is a $2$-Gieseker-semistable sheaf \footnote{The notion of $2$-Gieseker stability is in between slope stability and Gieseker stability. We refer to Section \ref{sec:stability} for full details.} $E$ with $\ch(E) = (r, c, d, e)$. By using the generalized Bogomolov inequality from \cite{BMT14:stability_threefolds, Mac14:conjecture_p3} one can show that $E(r, c, d) \neq \infty$ (see Proposition \ref{prop:function_E} for details).

\begin{thm}
\label{thm:main_ch3}
The values of $E(r, c, d)$ for $r \leq 3$ are given as follows:
\begin{enumerate}
    \item $E(1, 0, d) = \tfrac{1}{2}d^2 - \tfrac{1}{2}d$ for $d \leq 0$,
    \item \begin{enumerate}
    \item $E(2, -1, d) = \tfrac{1}{2}d^2 - d + \tfrac{5}{24}$ for $d \leq -\tfrac{1}{2}$,
        \item $E(2, 0, 0) = E(2, 0, -1) = 0$, and $E(2, 0, -d) = \tfrac{1}{2}d^2 + \tfrac{1}{2}d + 1$ for $d \leq -2$,
    \end{enumerate}
    \item \begin{enumerate}
        \item $E(3, -2, d) = \tfrac{1}{2}d^2 - \tfrac{3}{2}d + \tfrac{2}{3}$, for $d \leq 0$,
        \item $E(3, -1, -\tfrac{1}{2}) = -\tfrac{1}{6}$, and $E(3, -1, d) = \tfrac{1}{2}d^2 + \tfrac{17}{24}$ for $d \leq -\tfrac{3}{2}$,
        \item $E(3, 0, 0) = 0$, $E(3, 0, -1) = -1$, and $E(3, 0, d) = \tfrac{1}{2}d^2 + \tfrac{1}{2}d$ for $d \leq -2$,
    \end{enumerate}
    \item \begin{enumerate}
        \item $E(4, -3, d) = \tfrac{1}{2}d^2 - 2d + \tfrac{11}{8}$ for $d \leq -\tfrac{1}{2}$,
        \item $E(4, -2, d) = \tfrac{1}{2}d^2 - \tfrac{1}{2}d + \tfrac{2}{3}$ for $d \leq -1$,
        \item $E(4, -1, d) = \tfrac{1}{2}d^2 - \tfrac{7}{24}d$ for $d \leq -\tfrac{3}{2}$,
        \item $E(4, 0, 0) = E(4, 0, -2) = 0$, $E(4, 0, -1) = -2$, and $E(4, 0, d) = \tfrac{1}{2}d^2 + \tfrac{3}{2}d + 2$ for $d \leq -3$.
    \end{enumerate}
\end{enumerate}
\end{thm}

The case $r = 4$ is what is completely new in this article. The case of $r = 1$ deals just with ideal sheaves of curves. This is a fancy way of saying that curves with maximal genus in $\P^3$ are plane curves, a fact known since the days of Castelnuovo. A proof with the techniques presented here and some slight generalization to other notions of stability in the derived category can be found in \cite[Proposition 3.2]{MS20:space_curves}). The case $r = 2$ was first proved in \cite{OS85:spectrum_torsion_free_sheavesII} and a proof with the present methods and some minor generalizations can be found in \cite{Sch20:rank_two_p3}. The rank three case was first solved in \cite{MR87:P3_rank_three_chern_classes}. The method we use leads again to a slight generalization for certain semistable objects in the derived category. Lastly, there is also an appropriate rank zero case (see Theorem \ref{thm:rank_zero}).

As for the second Chern character, we can obtain all values for rank $r \leq 4$ via tensor products with line bundles. Moreover, this theorem together with the first one classifies Chern characters of $2$-Gieseker-semistable sheaves with rank $r \leq 4$. The details for the last two claims can be found in Proposition \ref{prop:function_E}. For $e = E(r, d, c)$ with $r \leq 4$ we actually find Gieseker-semistable sheaves with Chern character $(r, c, d, e)$. Therefore, the definition of $E$ is not dependent on whether we use slope stability or Gieseker stability in this case. However, for $e < E(r, c, d)$ there can be gaps, i.e., values for which there is a $2$-Gieseker-semistable, but not Gieseker-semistable sheaf (see Remark \ref{rmk:ch3_gaps}). Moreover, the analog of Question \ref{q:scaling} for the third Chern character fails (see Remark \ref{rmk:no_scaling_ch3}).

\begin{question}
What are the Chern character gaps for which there are slope-semistable (or $2$-Gieseker-semistable), but not Gieseker-semistable sheaves on $\P^3$?
\end{question}

For rank one the answer to this question is a classical result first stated by Halphen in \cite{Hal82:genus_space_curves} with an incorrect proof that was fixed in \cite{GP78:genre_courbesI}. For rank two this question was answered by Mir\'o-Roig in \cite{MR85:P3_gaps_rank_two}. As far as we know this is unknown for $r \geq 3$.

As an application, we can prove the following Conjecture made in \cite{Sch20:rank_two_p3} in a few more cases.

\begin{conj}
\label{conj:irreducible_smooth}
Let $(r, c, d) \in \Z \oplus \Z \oplus \tfrac{1}{2} \Z$ such that there are Gieseker-semistable sheaves $E$ with $\ch_{\leq 2}(E) = (r, c, d)$. Then the moduli space of Gieseker-semistable sheaves $M(r, c, d, E(r, c, d))$ is irreducible and smooth along the open locus of stable sheaves.
\end{conj}

\begin{thm}
\label{thm:conjecture_rank_three}
Conjecture \ref{conj:irreducible_smooth} holds for $r \leq 4$.
\end{thm}

In fact, there are fairly precise descriptions of these moduli spaces. For rank zero, the moduli spaces with maximal third Chern character are projective bundles over Hilbert schemes of plane curves (see the proof of Corollary \ref{cor:conjecture_rank_zero}). For rank one we are simply dealing with Hilbert schemes of plane curves. The rank two case was dealt with in both \cite{OS85:spectrum_torsion_free_sheavesII} and \cite{Sch20:rank_two_p3}. The rank three case can be found in \cite{MR87:P3_rank_three_extremal_moduli} and Remark \ref{rmk:moduli_spaces_rank_three}. Finally, for rank four we refer to Remark \ref{rmk:moduli_spaces_rank_four}.

\subsubsection*{Ingredients in the proof}

We use the notion of tilt stability in the derived category introduced for K3 surfaces in \cite{Bri08:stability_k3}, generalized to all surfaces in \cite{AB13:k_trivial}, and finally to higher dimensions in \cite{BMT14:stability_threefolds}. It is a generalization of slope stability to objects in another abelian subcategory of the bounded derived category consisting of some two-term complexes. Everything depends on two real parameters $\alpha > 0$, $\beta \in \R$. We refer to Section \ref{sec:stability} for the technical details.

Tilt stability has a few key properties. Walls are nested semicircles, i.e., they are linearly ordered. Above a \emph{biggest wall} semistable objects are simply $2$-Gieseker-semistable sheaves (see Proposition \ref{prop:large_volume_limit}). Using tricks with Euler characteristics or the generalized Bogomolov inequality from \cite{Mac14:conjecture_p3}, we can determine a \emph{smallest wall}. Together, this means that any semistable sheaf $E$ will destabilize along a wall that is determined by a short exact sequence
\[
0 \to F \to E \to G \to 0.
\]
If one knows upper bounds for $\ch_3(F)$ and $\ch_3(G)$, then one gets a bound for $\ch_3(E)$ by linearity of the Chern character in short exact sequences. Each object has the Bogomolov discriminant $\Delta(E) \coloneqq \ch_1(E)^2 - 2 \ch_0(E) \ch_2(E)$. The \emph{Bogomolov inequality} says that any semistable $E$ satisfies $\Delta(E) \geq 0$. Every time we hit a wall, we have $\Delta(F) + \Delta(G) \leq \Delta(E)$ with equality only in very special cases. This sets the problem up as an induction.

Bounding the second Chern character is slightly less direct. Using the classical Bogomolov inequality, one already gets a bound of $\ch_2$ from above. To obtain the stronger bounds claimed in Theorem \ref{thm:ch_2} for rank four, we only need to exclude finitely many cases. Given a value $d$ to exclude, we assume a sheaf $E$ exists and use the above techniques to bound $\ch_3(E)$ from above. We then use the same techniques to bound $\ch_3(E^{\vee})$. This implies a bound of $\ch_3(E)$ from below. Comparing the two bounds leads to a contradiction.

To show that the bounds are sharp, we simply construct a series of examples. Finally, the above technique for bounding third Chern characters also leads to a classification of semistable sheaves with maximal third Chern character in terms of destabilizing short exact sequences. This classification allows to give explicit descriptions of the moduli spaces either directly or with some standard uses of deformation theory. We immediately obtain Theorem \ref{thm:conjecture_rank_three} as a corollary.

\subsubsection*{Plan of the paper}

In Section \ref{sec:stability} we recall fundamental notions of stability. In Section \ref{sec:preparation} we set up our objects of study, prove some of their basic properties, and discuss open questions. Section \ref{sec:rank_zero_and_one} is about a rank zero version of Theorem \ref{thm:main_ch3} and the well-known rank one case. Then we recall the rank two case of Theorem \ref{thm:main_ch3} from \cite{OS85:spectrum_torsion_free_sheavesII, Sch20:rank_two_p3} in Section \ref{sec:rank_two}. In Section \ref{sec:special_cases} we proof Theorem \ref{thm:main_ch_2} and a few special cases of Theorem \ref{thm:main_ch3}. A new proof of the rank three case of Theorem \ref{thm:main_ch3} (see also \cite{MR87:P3_rank_three_chern_classes}) can be found in Section \ref{sec:rank_three} and finally, we prove the remaining rank four case of Theorem \ref{thm:main_ch3} in Section \ref{sec:rank_four}.


\subsubsection*{Acknowledgements}

I would like to thank Arend Bayer, Jack Huizenga, and Emanuele Macr\`i for discussing various aspects of this article. Additionally, I thank the referees for suggesting many improvements.


\subsubsection*{Notation}

\begin{center}
   \begin{supertabular}{ r l }
     $\Db(\P^3)$ & bounded derived category of coherent sheaves on $\P^3$ \\
     $\HH^{i}(E)$ & the $i$-th cohomology sheaf of a complex $E \in \Db(X)$ \\
     $H^i(E)$ & the $i$-th sheaf cohomology group of a complex $E \in \Db(X)$ \\
     $\D(\cdot)$ & the derived dual $\RlHom(\cdot, \OO_X)[1]$ \\
     $\ch(E)$ & total Chern character of an object $E \in \Db(X)$ \\
     $\ch_{\leq l}(E)$ & $(\ch_0(E), \ldots, \ch_l(E))$ \\
   \end{supertabular}
\end{center}


\section{Stability conditions}
\label{sec:stability}

In this section, we give the necessary background on stability conditions. For those completely unfamiliar, we refer to \cite{MS17:lectures_notes} for a more gentle introduction.

The notion of \emph{slope stability} was introduced for vector bundles on curves in \cite{Mum63:quotients}, and later generalized to arbitrary dimensions in \cite{Tak72:Stability1}. In higher dimensions, issues arise in the construction of moduli spaces of slope-semistable sheaves. A modification called \emph{Gieseker stability} was introduced which allows the construction of moduli spaces via a GIT construction (see \cite{Gie77:vector_bundles, Mar77:stable_sheavesI, Sim94:moduli_representations}). We will mostly recall things to set the notation and refer to \cite{HL10:moduli_sheaves} for more details.

For simplicity, we will restrict ourselves to $\P^3$ even though most of the following facts are perfectly fine on other varieties. For a sheaf $E \in \Coh(\P^3)$, we define its \emph{slope} as
\[
\mu(E) \coloneqq \frac{\ch_1(E)}{\ch_0(E)}
\]
with $\mu(E) = \infty$ if $\ch_0(E) = 0$. The sheaf $E$ is \emph{slope-semistable} if for any non-trivial subsheaf $F \into E$ the slope decreases, i.e., $\mu(F) \leq \mu(E)$. It is called \emph{slope-stable} if for all non-trivial subsheaves $F \subset E$ the inequality $\mu(F) < \mu(E/F)$ holds. Classically, an equivalent definition is used for stable that says that $\mu(F) < \mu(E)$, but only when $F$ has smaller rank than $E$. However, our definition generalizes more easily to stability notions in the derived category.

\begin{lem}
\label{lem:dual_semistable_slope}
If $E \in \Coh(X)$ is slope-semistable, then so is the dual sheaf $E^{\vee}$.
\end{lem}

It seems that the notion of stability depends only on the choice of the slope function $\mu$. Bridgeland's brilliant insight in \cite{Bri07:stability_conditions} is that another crucial ingredient is the definition of what a subobject is. This can be redefined by changing coherent sheaves for another abelian category inside the bounded derived category, more precisely another heart of a bounded t-structure.

For any $\beta \in \R$ we define two full additive subcategories of $\Coh(\P^3)$ as follows. The category $\FF^{\beta} \subset \Coh(\P^3)$ is the smallest extension-closed full additive subcategory that contains all slope-semistable sheaf $E$ with $\mu(E) \leq \beta$. Similarly, $\TT^{\beta} \subset \Coh(\P^3)$ is the smallest extension-closed full additive subcategory of $\Coh(\P^3)$ that contains all slope-semistable sheaf $E$ with $\mu(E) > \beta$. Then one defines $\Coh^{\beta}(\P^3) \subset \Db(X)$ to be the full additive subcategory of complexes $E \in \Db(X)$ such that $\HH^0(E) \in \TT^{\beta}$, $\HH^{-1}(E) \in \FF^{\beta}$, and $\HH^i(E) = 0$ for $i \neq 0, -1$. It turns out that $\Coh^{\beta}(\P^3)$ is the heart of a bounded t-structure. In particular, it is abelian.

To simplify notation, we define the twisted Chern character $\ch^{\beta} \coloneqq \ch \cdot e^{-\beta H}$, where $H$ is the hyperplane class. If $\beta \in \Z$ and $E \in \Db(\P^3)$, then we simply have $\ch^{\beta}(E) = \ch(E(-\beta H))$.

The slope function depends on another real parameter $\alpha > 0$ and is defined as
\[
\nu_{\alpha, \beta} \coloneqq \frac{\ch_2^{\beta} - \frac{\alpha^2}{2} 
\ch_0^{\beta}}{\ch_1^{\beta}}
\]
with $\nu_{\alpha, \beta}(E) = \infty$ if $\ch_1^{\beta}(E) = 0$. As previously, $E \in \Coh^{\beta}(\P^3)$ is called \emph{$\nu_{\alpha, \beta}$-semistable} (or \emph{tilt-semistable}) if for any non-trivial subobject $F \into E$ in $\Coh^{\beta}(\P^3)$ the inequality $\nu_{\alpha, \beta}(F) \leq \nu_{\alpha, \beta}(E)$ holds. Moreover, we call $E$ \emph{$\nu_{\alpha, \beta}$-stable} (or \emph{tilt-stable}) if for any non-trivial subobject $F \into E$ in $\Coh^{\beta}(\P^3)$ the inequality $\nu_{\alpha, \beta}(F) < \nu_{\alpha, \beta}(E/F)$ holds.

In this new notion of stability the role of the rank is replaced by $\ch_1^{\beta}$, i.e., one could for example define stability by only looking at subobjects with strictly smaller $\ch_1^{\beta}$.

There are various known bounds of Chern characters for semistable objects. The most famous is the Bogomolov inequality. It was first proved in \cite{Bog78:inequality}. It also holds in tilt stability as shown in \cite[Corollary 7.3.2]{BMT14:stability_threefolds}.

\begin{thm}[Bogomolov inequality]
If $E$ is a slope-semistable sheaf or a tilt-semistable object, then
\[
\Delta(E) \coloneqq \ch_1^2(E) - 2\ch_0(E) \ch_2(E) \geq 0.
\]
\end{thm}

In \cite{BMT14:stability_threefolds} a conjecture was made that in tilt stability there is a similar inequality involving the third Chern character. In the case of $\P^3$ it was proved in \cite{Mac14:conjecture_p3}.

\begin{thm}[Generalized Bogomolov inequality]
Assume that $E \in \Coh^{\beta}(\P^3)$ is $\nu_{\alpha, \beta}$-semistable. Then
\[
Q_{\alpha, \beta}(E) \coloneqq \alpha^2 \Delta(E) + 4 \left( \ch_2^{\beta}(E) \right)^2 - 6 \ch_1^{\beta}(E) \ch_3^{\beta}(E) \geq 0.
\]
\end{thm}

Next, we need to connect tilt stability with stability for sheaves. In general, there is no chamber in which tilt stability is the same as either slope stability or Gieseker stability. Instead, we will have to work with the slightly unusual notion of \emph{$2$-Gieseker stability}. 

The polynomial ring $\R[m]$ has a pre-order defined as follows:
\begin{enumerate}
    \item If $f \in \R[m]$ is non-zero, then $f \prec 0$.
    \item If $f, g \in \R[m]$ are non-zero with $\deg(f) > \deg(g)$, then $f \prec g$.
    \item For non-zero $f, g \in \R[m]$ with $\deg(f) = \deg(g)$, let $a_f$ and $a_g$ be their leading coefficients. We define $f \preceq (\prec) g$ if $f(m)/a_f \leq (<) g(m)/a_g$ for $m \gg 0$.
    \item Lastly, if $f, g \in \R[m]$ satisfy both $f \preceq g$ and $g \preceq f$, then $f \asymp g$.
\end{enumerate}

If $E \in \Coh(\P^3)$, then its \emph{Hilbert polynomial} and the terms $\alpha_i(E)$ are defined by
\[
P(E, m) \coloneqq \chi(E(mH)) = \sum_{i=0}^3 \alpha_i(E) m^i.
\]
We also write $P_2(E, m) = \sum_{i = 1}^3 \alpha_i(E) m^i$.

\begin{defn}
\begin{enumerate}
    \item A sheaf $E \in \Coh(\P^3)$ is \emph{Gieseker-semistable} if for any non-trivial subsheaf $F \into E$ the inequality $P(F, m) \preceq P(E, m)$ holds. Moreover, $E$ is called \emph{Gieseker-stable} if instead $P(F, m) \prec P(E, m)$.
    \item A sheaf $E \in \Coh(\P^3)$ is \emph{$2$-Gieseker-semistable} if for any non-trivial subsheaf $F \into E$ the inequality $P_2(F, m) \preceq P_2(E, m)$ holds. Moreover, $E$ is called \emph{$2$-Gieseker-stable} if instead $P_2(F, m) \prec P_2(E/F, m)$.
\end{enumerate}

\end{defn}

Note that as in the definition of slope stability we need to compare to the quotient to define $2$-Gieseker-stable. The reason is that otherwise no sheaf would be slope-stable or $2$-Gieseker-semistable, because there is always a quotient onto a skyscraper sheaf $E \onto \OO_P$ that would ruin the day. The same problem does not occur for Gieseker stability, but the definition would work equivalently by comparing with the quotient.

The following statement was proved in \cite[Proposition 14.2]{Bri08:stability_k3} for K3 surfaces, but the proof is essentially the same in general, see for example \cite[Theorem 5.2]{JM22:walls_bridgeland_3fold}.

\begin{prop}[Large volume limit]
\label{prop:large_volume_limit}
Let $E \in \Db(X)$ and let $\beta < \mu(E)$. Then $E$ is in the category $\Coh^{\beta}(X)$ and $E$ is $\nu_{\alpha, \beta}$-(semi)stable for $\alpha \gg 0$ if and only if $E$ is a $2$-Gieseker-(semi)stable sheaf.
\end{prop}

We will not really use Gieseker stability outside the following chain of implications and refer to \cite{HL10:moduli_sheaves} for more background on stability of sheaves.

\centerline{
\xymatrix{
\text{slope-stable} \ar@{=>}[r] & \text{$2$-Gieseker-stable} \ar@{=>}[r] & \text{Gieseker-stable} \ar@{=>}[d] \\
\text{slope-semistable} & \text{$2$-Gieseker-semistable} \ar@{=>}[l] &\text{Gieseker-semistable} \ar@{=>}[l]
}}

For a fixed class $(r, c, d) \in \Z \oplus \Z \oplus \tfrac{1}{2} \Z$, there is a locally finite wall and chamber structure in the upper half-plane $\H := \{ (\alpha, \beta) \in \R^2 : \ \alpha > 0\}$ such that the set of stable objects with $\ch_{\leq 2} = (r, c, d)$ only changes along walls. A \emph{numerical wall} is a non-trivial proper subset of $(\alpha, \beta)$ for which $\nu_{\alpha, \beta}(r, c, d) = \nu_{\alpha, \beta}(s, x, y)$ for some fixed $s, x, y \in \R$. For further clarity we will occasionally call walls \emph{actual walls}. The structure of these walls is well understood and was first formally established in \cite{Mac14:nested_wall_theorem}. 

\begin{prop}[Numerical properties of walls in tilt stability]
\label{prop:structure_walls}
Fix $v = (r, c, d) \in \Z \oplus \Z \oplus \tfrac{1}{2} \Z$ with $\Delta_H(v) \geq 0$.
\begin{enumerate}
    \item Numerical walls are either semicircles with center along the $\beta$-axis or vertical rays. They do not intersect, not even for $\alpha = 0$ unless $\Delta(v) = 0$.
    \item All semicircular numerical walls intersect the curve of $(\alpha, \beta)$ with $\nu_{\alpha, \beta}(v) = 0$ in their apex, i.e., their point with maximal $\beta$.
    \item If $r \neq 0$, then the equation $\nu_{\alpha, \beta}(v) = 0$ defines a hyperbola that is symmetric with respect to the vertical ray $\beta = r/c$. In particular, the semicircular numerical walls are nested along either of the two branches of the hyperbola. Moreover, there is at most one vertical wall given by $\beta = r/c$.
    \item If $r = 0$ and $c \neq 0$, then there is no numerical vertical wall and $\nu_{\alpha, \beta}(v) = 0$ is equivalent to $\beta = d/c$. In particular, all numerical walls are nested semicircles with apex along this vertical ray.
    \item There is a largest actual wall for a fixed $v$, i.e., the radii of these semicircles are bounded from above.
    \item If $\Delta(v) > 0$, then the set of $(\alpha, \beta)$ for which $Q_{\alpha, \beta}(r, c, d, e) = 0$ is the semicircular numerical wall $W((r, c, d), (c, 2d, 3e))$. The region in which $Q_{\alpha, \beta}(r, c, d, e) < 0$ is always given by the semidisk enclosed by $Q_{\alpha, \beta}(r, c, d, e) = 0$.
\end{enumerate}
\end{prop}

A proof of part (i) - (v) can be found in \cite[Theorem 3.3 and Lemma 7.3]{Sch20:stability_threefolds}. Part (vi) is a straightforward calculation. For objects $E, F \in \Db(X)$ we denote the numerical wall where $\nu_{\alpha, \beta}(E) = \nu_{\alpha, \beta}(F)$ by $W(E, F)$. If this wall is semicircular, than we denote the radius as $\rho(E, F)$ and the center along the $\beta$-axis by $s(E, F)$. We denote the numerical wall defined through $Q_{\alpha, \beta}(E) = 0$ as $W_Q(E)$. If the wall is semicircular, then the center and radius of $W_Q(E)$ are denoted by $s_Q(E)$ and $\rho_Q(E)$.

For any tilt-semistable $E$ with $\ch_0(E) \neq 0$, we define $\beta_{-}(E) \leq \beta_{+}(E)$ to be the solutions to the quadratic equation $\nu_{0, \beta}(E) = \ch_2^{\beta}(E) = 0$. We have
\begin{align*}
    \beta_{-}(E) &= \mu(E) - \sqrt{\frac{\Delta(E)}{(H^2 \cdot \ch_1(E))^2}}, \\
    \beta_{+}(E) &= \mu(E) + \sqrt{\frac{\Delta(E)}{(H^2 \cdot \ch_1(E))^2}}.
\end{align*}
Whenever an object $E$ is detabilized along a wall $W$, this is due to a short exact sequence 
\[
0 \to F \to E \to G \to 0
\]
in $\Coh^{\beta}(\P^3)$ for $(\alpha, \beta) \in W$ such that $W = W(E, F) = W(E, G) = W(F, G)$. 

\begin{prop}[Properties of destabilizing sequences]
\label{prop:properties_dest_sequences}
Let $0 \to F \to E \to G \to 0$ be a destabilizing sequence for $E$ along a wall $W$. Then the following properties hold.
\begin{enumerate}
    \item If $W$ is semicircular and $\ch_0(F) > \ch_0(E) \geq 0$ or $\ch_0(F) < \ch_0(E) \leq 0$, then
    \[
    \rho(E, F)^2 \leq \frac{\Delta(E)}{4 \ch_0(F)(\ch_0(F) - \ch_0(E))}.
    \]
    \item It is $\Delta(F) + \Delta(F) \leq \Delta(E)$ with equality if and only if either $\ch_{\leq 2}(F) = 0$ or $\ch_{\leq 2}(G) = 0$.
    \item For any $(\alpha, \beta) \in W$, we have $0 \leq \ch_1^{\beta}(F) \leq \ch_1^{\beta}(E)$. Moreover, in case of equality on either end, the wall is vertical. In particular, if $\ch_1^{\beta_0}(F) > 0$ is minimal for this value of $\beta_0$ then there is no wall that intersects the ray $\beta = \beta_0$.
    \item If $E$ has rank zero, then $F$ and $G$ must have non-zero rank.
    \item Assume that $E$ has strictly positive rank. Then either $\ch_0(F) > 0$ and $\mu(F) \leq \mu(E)$, or $\ch_0(G) > 0$ and $\mu(G) \leq \mu(E)$.
    \item Assume that $W$ is semicircular, $E$ has strictly positive rank, $\ch_0(F) > 0$, and $\mu(F) \leq \mu(E)$. Then $\beta_{-}(E) < \beta_{-}(F) \leq \mu(F) < \mu(E)$.
\end{enumerate}
\end{prop}

\begin{proof}
A statement close to (i) first appeared in \cite{CH16:ample_cone_plane}. This precise version can be found in \cite[Lemma 2.4]{MS20:space_curves}. Part (ii) is a special case of the results established in \cite[Appendix A]{BMS16:abelian_threefolds}.

To show (iii) note that $\ch_1^{\beta}(F) \geq 0$ is an immediate consequence of the definition of $\Coh^{\beta}(\P^3)$. The inequality $\ch_1^{\beta}(F) \leq \ch_1^{\beta}(E)$ is equivalent to $\ch_1^{\beta}(G) \geq 0$. Finally, in case of either equality, a straightforward computations shows that $W$ becomes the vertical wall.

Part (iv) follows because the equation $\nu_{\alpha, \beta}(E) = \nu_{\alpha, \beta}(F)$ is independent of $(\alpha, \beta)$ whenver both $E$ and $F$ have rank zero. Obviously, this cannot define a wall.

For part (v) there are two cases to consider. If $\ch_0(F) \geq \ch_0(E)$, then the statement follows as a consequence of (iii). The same argument works if $\ch_0(G) \geq \ch_0(E)$. Assume $0 < \ch_0(F)$, and $\ch_0(G) < \ch_0(F)$. Then the fact that $\ch_{\leq 1}(F) + \ch_{\leq 1}(G) = \ch_{\leq 1}(E)$ implies that either $\mu(F) \leq \mu(E) \leq \mu(G)$ or $\mu(G) \leq \mu(E) \leq \mu(F)$.

The proof for (vi) can be found in the proof of \cite[Lemma 3.6]{BMSZ17:stability_fano}.
\end{proof}

Next, we need to talk about the derived dual. For any $E \in \Db(X)$ we define the \emph{derived dual} as $\D(E) \coloneqq \RHom(E, \OO_X)[1]$. As shown in \cite[Proposition 5.1.3]{BMT14:stability_threefolds} tilt stability behaves reasonably well with respect to this operation:

\begin{prop}
\label{prop:derived_dual}
Let $E \in \Coh^{\beta}(\P^3)$ be $\nu_{\alpha, \beta}$-semistable with $\nu_{\alpha, \beta}(E) \neq 0$. Then there is a $\nu_{\alpha, -\beta}$-semistable object $\tilde{E} \in \Coh^{-\beta}(\P^3)$ together with a sheaf $T$ supported in dimension zero and a distinguished triangle
\[
T[-2] \to \tilde{E} \to \D(E) \to T[-1].
\]
\end{prop}

We need the following fact that roughly says if it looks like a line bundle, swims like a line bundle, and quacks like a line bundle, then it is a line bundle.

\begin{prop}[{\cite[Proposition 4.1 and 4.6]{Sch20:stability_threefolds}}]
\label{prop:line_bundles}
Assume that $E$ is slope-semistable or $\nu_{\alpha, \beta}$-semistable for $\alpha > 0$, $\beta < m$ and $\ch(E) = \ch(\OO(m)^{\oplus r})$ for some $r, m \in \Z$ with $r \geq 0$. Then $E \cong \OO(m)^{\oplus r}$. Moreover, if $E$ is $\nu_{\alpha, \beta}$-semistable for $\beta > m$ and $\ch_{\leq 2}(E) = - \ch_{\leq 2}(\OO(m)^{\oplus r})$ for $r, m \in \Z$ with $r \geq 0$, then $E \cong \OO(m)^{\oplus r}[1]$.
\end{prop}

Finally, we use the following bound on the second Chern character established in \cite{Li19:conjecture_fano_threefold}.

\begin{thm}[{\cite[Proposition 3.2]{Li19:conjecture_fano_threefold}}]
\label{thm:li_bound}
Let $E$ be a tilt-stable object with $\ch_0(E) \neq 0$ such that there is $n \in \Z$ with either $\beta_{-}(E), \beta_{+}(E) \in [n, n+1)$ or $\beta_{-}(E), \beta_{+}(E) \in (n, n+1]$. Then
\[
\Delta(E) \geq \frac{3}{8} \ch_0(E)^2.
\]
\end{thm}

We refer to \cite[Theorem 3.1]{BMSZ17:stability_fano} for a proof in the precise terms above. It is possible to circumvent the use of this Theorem and give completely self-contained proofs. However, its use shortens several arguments quite substantially.

We finish the section with some remarks on moduli spaces and deformation theory. The situation for Gieseker-semistable sheaves is well understood and can be found for example in \cite{HL10:moduli_sheaves}. The key property is that moduli spaces are projective and for a Gieseker-stable sheaf $E$ the tangent space is given by $\Ext^1(E, E)$. Moreover, the obstruction lies in $\Ext^2(E, E)$, i.e., if $\Ext^2(E, E) = 0$, then the moduli space is smooth at $E$. In \cite{Ina02:moduli_complexes} and \cite{Lie06:moduli_complexes} it was shown that the deformation theory for complexes is governed by the same $\Ext$-groups. The situation for moduli spaces is more complicated. However, we are only dealing with examples, where the moduli spaces can be directly read off the destabilizing sequences. Therefore, we will not have to dig deeper into this issue.

We will denote the moduli space of Gieseker-semistable sheaves $E$ with $\ch(E) = (r, c, d, e)$ by $M(r, c, d, e)$ and the moduli space of $\nu_{\alpha, \beta}$-semistable objects $E$ with $\ch(E) = (r, c, d, e)$ by $M^{\alpha, \beta}(r, c, d, e)$.


\section{Preparation and open questions}
\label{sec:preparation}

Let 
\begin{align*}
\operatorname{\CC\HH}(\P^3) &\coloneqq \{ \ch(E) : E \in \Coh(\P^3) \}, \\
\operatorname{\CC\HH}_{\leq 2}(\P^3) &\coloneqq \{ \ch_{\leq 2}(E) : E \in \Coh(\P^3) \}, \\
\operatorname{\CC\HH}_{\leq 1}(\P^3) &\coloneqq \{ \ch_{\leq 1}(E) : E \in \Coh(\P^3) \}.
\end{align*}

\begin{lem}
\label{lem:ch_classification}
Let $v = (r, c, d, e) \in \Z \oplus \Z \oplus \tfrac{1}{2} \Z \oplus \tfrac{1}{6} \Z$.
\begin{enumerate}
    \item We have $(r, c) \in \operatorname{\CC\HH}_{\leq 1}(\P^3)$ if and only if 
    \begin{enumerate}
        \item $r \geq 0$, and
        \item if $r = 0$, then $c \geq 0$.
    \end{enumerate}
    \item We have $(r, c, d) \in \operatorname{\CC\HH}_{\leq 2}(\P^3)$ if and only if
    \begin{enumerate}
        \item $(r, c) \in \operatorname{\CC\HH}_{\leq 1}(\P^3)$,
        \item if $r = c = 0$, then $d \geq 0$,
        \item if $c$ is even, then $d \in \Z$, and
        \item if $c$ is odd, then $d \notin \Z$.
    \end{enumerate}
    \item We have $v \in \operatorname{\CC\HH}(\P^3)$ if and only if
    \begin{enumerate}
        \item $(r, c, d) \in \operatorname{\CC\HH}_{\leq 2}(\P^3)$,
        \item if $r = c = d = 0$, then $e \geq 0$,
        \item $\chi(v) \in \Z$.
    \end{enumerate}
\end{enumerate}
\end{lem}

\begin{proof}
\begin{enumerate}
\item The highest non-zero Chern character of a sheaf is the degree of its support. Therefore, it has to be positive. This immediately implies that the condition is necessary. The construction of sheaves will be the same in part (ii), and we will therefore delay it for the moment.
\item  For any $E \in \Coh(\P^3)$ we know $c_2(E) \in \Z$ and $\ch_2(E) = \tfrac{1}{2} \ch_1^2(E) - c_2(E)$. This implies the last two conditions. Condition (b) follows again from the fact that the highest non-zero Chern character is the class of the support of the sheaf, i.e., has to be positive.

Next, we have to construct sheaves with given Chern character. Let $(r, c, d)$ fulfill all four conditions from the statement. We can compute $\ch_{\leq 2}(\OO) = (1, 0, 0)$. If $V \subset \P^3$ is a plane, then $\ch_{\leq 2}(\OO_V(f)) = (0, 1, - \tfrac{1}{2})$, and if $L \subset \P^3$ is a line, then $\ch_{\leq 2}(\OO_L) = (0, 0, 1)$. Note that
\[
(r, c, d) = r \ch_{\leq 2}(\OO) + c \ch_{\leq 2}(\OO_V) + (d + \tfrac{c}{2}) \ch_{\leq 2}(\OO_L).
\]
If $c \geq 0$ and $d + \tfrac{c}{2} \geq 0$, then
\[
E = \OO^{\oplus r} \oplus \OO_V^{\oplus c} \oplus \OO_L^{\oplus (d + \tfrac{c}{2})}
\]
satisfies $\ch(E) = (r, c, d)$. If $c < 0$ or $d + \tfrac{c}{2} < 0$, then for any $n \in \Z$ we define \[
(r', c', d') = (r, c, d) \otimes \ch_{\leq 2}(\OO(n)) = (r, c + nr, d + nc + \tfrac{1}{2} n^2 r)
\]
In particular, we can choose $n$ large enough such that $c' > 0$ and $d' + \tfrac{c'}{2} \geq 0$. We have already constructed $E' \in \Coh(\P^3)$ such that $\ch(E') = (r', c', d')$, and $E = E' \otimes \OO(-n)$ satisfies $\ch(E) = (r, c, d)$. 
\item If $r = c = d = 0$, then any $E$ with $\ch(E) = (0, 0, 0, e)$ is supported on a zero-dimensional subscheme of length $e$, i.e., $e \geq 0$. By definition $\chi(E) \in \Z$ for any $E \in \Coh(\P^3)$.

Vice-versa, we need to construct a sheaf $E$ with $\ch(E) = (r, c, d, e)$ subject to the conditions in the statement. If $r = c = d = 0$, then we can choose an arbitrary $P \in \P^3$, and set $E = \OO_P^{\oplus e}$. Assume that $(r, c, d) \neq 0$. We already know that there is an $E' \in \Db(X)$ with $\ch(E') = (r, c, d, e')$ for some $e' \in \tfrac{1}{6} \Z$. From here we get
$e - e' = \chi(v) - \chi(E') \in \Z$. Since $\ch(\OO_P) = (0, 0, 0, 1)$ for any $P \in \P^3$, we can take $E = E' \oplus \OO_P^{\oplus (e' - e)}$ if $e' > e$. We can always find $P \in \P^3$ together with a surjective map $E' \onto \OO_P$ to obtain a sheaf with Chern character $(r, c, d, e' - 1)$. By induction on $e - e'$ starting at $e = e'$, we can finish the remaining case $e' \leq e$. \qedhere
\end{enumerate}
\end{proof}

Less trivially, the point of this article is to understand Chern characters of semistable sheaves instead of arbitrary sheaves. We start by introducing the relevant functions.

\begin{defn}
\begin{enumerate}
    \item We define a function $D: \operatorname{\CC\HH}_{\leq 1}(\P^3) \to \tfrac{1}{2} \Z \cup \{\infty, -\infty\}$ as follows. Let $(r, c) \in \operatorname{\CC\HH}_{\leq 1}(\P^3)$. If $r \neq 0$, then $D(r, c)$ is the supremum over all $d$ such that there is a slope-semistable sheaf $E$ with $\ch_{\leq 2}(E) = (r, c, d)$. If $r = 0$, then $D(r, c)$ is the supremum over all $d$ such that there is a $2$-Gieseker-semistable sheaf $E$ with $\ch_{\leq 2}(E) = (r, c, d)$. In both cases, we define $D(r, c) = -\infty$ if no such sheaf exists.
    \item Similarly, we define a function $E: \operatorname{\CC\HH}_{\leq 2}(\P^3) \to \tfrac{1}{6} \Z \cup \{\infty, -\infty\}$ as follows. Assume that $(r, c, d) \in \operatorname{\CC\HH}_{\leq 2}(\P^3)$. If $(r, c) \neq (0, 0)$, then $E(r, c, d)$ is the supremum over all $e$ such that there is $2$-Gieseker-semistable sheaf $E$ with $\ch(E) = (r, c, d, e)$. If $(r, c) = (0, 0)$, then $E(r, c, d)$ is the supremum over all $e$ such that there is Gieseker-semistable sheaves $E$ with $\ch(E) = (0, 0, d, e)$. Again, we write $E(r, c, d) = -\infty$ if no such sheaf exists.
\end{enumerate}
\end{defn}

We varied which precise notion of semistability to use in this definition. This is certainly confusing at first sight and some explanation is in order.

\begin{enumerate}
    \item Slope stability has the easiest definition. It behaves well under dualizing. This allows us to prove some general results below. The downside is two-fold. Firstly, moduli spaces are classically constructed for Gieseker stability. Secondly, slope stability does not behave well with respect to tilt stability.
    \item It not known whether $2$-Gieseker stability admits reasonable moduli spaces. The big advantage of this notion is that it occurs as a limit for tilt stability. Since tilt stability is our main method of proving non-trivial results, we need this notion largely for technical reasons.
    \item Gieseker stability has somewhat reasonable moduli spaces. Therefore, we ultimately care about this notion the most. However, it is the most cumbersome notion to prove anything about.
\end{enumerate}

The above definition of $D$ and $E$ is using the least strong, but reasonable notion of stability in each case. That makes it possible to prove general results, while obtaining meaningful and non-trivial bounds. All rank zero objects are slope-semistable by our definition. Therefore, we cannot expect good bounds with this notion. Similarly, sheaves supported in dimension one are all $2$-Gieseker-semistable. We expect that the bounds $d \leq D(r, c)$ and $e \leq D(r, c, d)$ will be sharp for Gieseker-semistable sheaves as well. It would be convenient to define $E(r, c, d)$ for $r > 0$ with slope stability as well, but this turns out to cause problems. For example, the sheaf $\OO \oplus \II_L$ for a line $L \subset \P^3$ is slope-semistable with Chern character $(2, 0, -1, 1)$, but according to Theorem \ref{thm:rank_two} there is no Gieseker-semistable sheaf with these invariants.

Next, we prove a few properties of these functions.

\begin{prop}
\label{prop:function_D}
We have $D(0, c) = \infty$ for any $c \geq 0$. Assume that $(r, c) \in \operatorname{\CC\HH}_{\leq 1}(\P^3)$ with $r \neq 0$. Then the following properties hold.
\begin{enumerate}
    \item $D(r, c) < \infty$.
    \item $D(r, c + nr) =  D(r, c) + nc + \tfrac{1}{2}n^2r$ for all $n \in \Z$.
    \item $D(r, -c) = D(r, c)$. 
    \item $D(r, 0) = 0$.
    \item For any $d \leq D(r, c)$ with $(r, c, d) \in \operatorname{\CC\HH}_{\leq 2}(\P^3)$, there is a slope-semistable $E \in \Coh(X)$ with $\ch_{\leq 2}(E) = (r, c, d)$.
\end{enumerate}
\end{prop}

\begin{proof}
If $C \subset \P^3$ is a curve of degree $d$, Then $\ch_{\leq 2}(\OO_C) = (0, 0, d)$. Therefore, $D(0,0) = \infty$. Assume that $c > 0$. If $X \subset \P^3$ is a surface of degree $c$, then $\lim_{n \to \infty} \ch_2(\OO_X(n)) = \infty$, i.e., $D(0, c) = \infty$. From now on assume that 
$(r, c) \in \operatorname{\CC\HH}_{\leq 1}(\P^3)$ with $r \neq 0$.

\begin{enumerate}
    \item The Bogomolov inequality $\Delta(E) \geq 0$ is equivalent to $d \leq \tfrac{c^2}{2r}$, i.e, $D(r, c) < \infty$.
    \item The formula $D(r, c + nr) =  D(r, c) + nc + \tfrac{1}{2}n^2r$ is an immediate consequences from the fact that taking the tensor product with $\OO(n)$ preserves stability.
    \item Next, we will prove $D(r, -c) \geq D(r, c)$. By duality, this will then imply $D(r, -c) = D(r, c)$. Assume that there is a slope-semistable sheaf $E$ with $\ch_{\leq 2}(E) = (r, c, D(r, c))$. By Lemma \ref{lem:dual_semistable_slope} the dual sheaf $E^{\vee}$ is slope-semistable as well. Since $E$ is torsion-free, we know due to \cite{HL10:moduli_sheaves}[Proposition 1.1.10] that $\lExt^3(E, \OO_X) = 0$, $\ch_{\leq 2}(\lExt^2(E, \OO_X)) = (0, 0, 0)$, and $\ch_{\leq 2}(\lExt^1(E, \OO_X)) = (0, 0, y)$ for some $y \geq 0$. Therefore, $\ch_{\leq 2}(E^{\vee}) = (r, -c, D(r, c) + y)$, and this implies $D(r, -c) \geq D(r, c) + y \geq D(r, c)$. 
    \item The fact that $D(r, 0) = 0$ for $r \neq 0$ is obvious, because the Bogomolov inequality implies $D(r, 0) \leq 0$ and $\ch_{\leq 2}(\OO^{\oplus r}) = (r, 0, 0)$.
    \item Lastly, we need to show the existence part of the statement. If $D(r, c) = -\infty$, there is nothing to show. Assume $D(r, c) \neq -\infty$. By definition there is a slope-semistable sheaf $E$ with $\ch_{\leq 2}(E) = (r, c, D(r, c))$. If $Z \subset \P^3$ is the closed subset of $Z$ at which $E$ is not locally-free, then $\dim Z \leq 1$, and we can choose a line $L \subset \P^3$ such that $Z \cap L = \emptyset$. Choose $a$ large enough such that $\Hom(E, \OO_L(a)) > 0$. Let $E_1$ be the kernel of a non-zero morphism $E \to \OO_L(a)$. Then $E_1$ is slope-semistable as well, and $\ch_{\leq 2}(E_1) = (r, c, D(r, c) - 1)$. We can continue this process to create slope-semistable sheaves $E_n$ with $\ch_{\leq 2}(E_n)(r, c, D(r, c) - n)$ for any $n \geq 0$. \qedhere
\end{enumerate}
\end{proof}

We will show in Theorem \ref{thm:rank_zero} that an analog of the last property holds for rank zero objects as well, i.e., there is a $2$-Gieseker-semistable sheaf $E \in \Coh(X)$ with $\ch_{\leq 2}(E) = (0, c, d)$ for any $(0, c, d) \in \operatorname{\CC\HH}_{\leq 2}(\P^3)$.

A similar statement can be proved for the third Chern character.

\begin{prop}
\label{prop:function_E}
We have $E(0, 0, d) = \infty$ for any $d \geq 0$. Assume that $(r, c) \in \operatorname{\CC\HH}_{\leq 1}(\P^3)$ with $(r, c) \neq (0, 0)$. Then the following properties hold.
\begin{enumerate}
    \item $E(r, c, d) < \infty$.
    \item $E(r, c + nr, d + nc + \tfrac{1}{2} n^2 r) = E(r, c, d, e) + nd + \tfrac{1}{2} n^2 c + \tfrac{1}{6} n^3 r$ for all $n \in \Z$.
    \item $E(r, 0, 0) = 0$ for $r > 0$.
    \item There is a $2$-Gieseker-semistable sheaf $E \in \Coh(X)$ with $\ch(E) = (r, c, d, e)$ if and only if $e \leq E(r, c, d)$ and $(r, c, d, e) \in \operatorname{\CC\HH}(\P^3)$.
    \item There is a Gieseker-semistable sheaf $E$ with $\ch(E) = (0, 0, d, e)$ for any $(0, 0, d, e) \in \operatorname{\CC\HH}(\P^3)$.
\end{enumerate}
\end{prop}

\begin{proof}
If $P \in \P^3$, then $\lim_{e \to \infty} \ch_3(\OO_P^{\oplus e}) = \infty$, i.e., $E(0, 0, 0) = \infty$. If $C \subset \P^3$ is a curve of degree $d$, then $\lim_{n \to \infty} \ch_3(\OO_C(n)) = \infty$. Therefore, $E(0, 0, d) = \infty$. From now on assume that $(r, c) \in \operatorname{\CC\HH}_{\leq 1}(\P^3)$ with $(r, c) \neq (0, 0)$.
\begin{enumerate}
    \item By Proposition \ref{prop:large_volume_limit} about the large volume limit and Proposition \ref{prop:structure_walls} there is a numerical wall $W$ for objects $E$ with $\ch_{\leq 2}(E) = (r, c, d)$ such that $E$ is tilt-semistable above $W$ if and only if $E$ is $2$-Gieseker-semistable. In particular, such an $E$ will be tilt-semistable for $0 < \alpha \ll 1$ and $\beta \ll 0$. 

    If $\Delta(r, c, d) = c^2 - 2rd = 0$, then the inequality $Q_{0, \beta}(r, c, d, e) \geq 0$ for $\beta \ll 0$ is equivalent to 
    \[
    e \leq \frac{c^3}{6r^2}.
    \]
    Assume $\Delta(r, c, d) > 0$. Then the inequality $Q_{0, \beta}(r, c, d, e) \geq 0$ is equivalent to
    \[
    \Delta(r, c, d) \beta^2 + (6er - 2cd) \beta + 4d^2 - 6ce \geq 0.
    \]
    Therefore, a $2$-Gieseker-semistable $E$ has bounded $\ch_3$ depending on its other Chern characters and the choice of a $\beta \ll 0$ such that $(0, \beta)$ lies outside the wall $W$.
    
    \item The statement $E(r, c + nr, d + nc + \tfrac{1}{2} n^2 r) = E(r, c, d, e) + nd + \tfrac{1}{2} n^2 c + \tfrac{1}{6} n^3 r$ for all $n \in \Z$ is simply a consequence of the fact that tensoring with the line bundle $\OO(n)$ preserves stability.
    \item We have already shown that $E(r, 0, 0) \leq 0$ in the proof of (i). Equality follows from the existence of $\OO^{\oplus r}$.
    \item Next, we need to construct sheaves. By definition, if $E(r, c, d) \neq \pm \infty$, there is a $2$-Gieseker-semistable sheaf $E$ with $\ch(E) = (r, c, d, E(r, c, d))$. We can find a point $P \in \P^3$ such that there is a surjective homomorphism $E \onto \OO_P$. The kernel $E_1$ is still $2$-Gieseker-semistable and sastifies $\ch(E_1) = (r, c, d, E(r, c, d) - 1)$. We can continue this process to construct $2$-Gieseker-semistable sheaves $E_n$ with $\ch(E_n) = (r, c, d, E(r, c, d) - n)$ for any $n \in \Z_{\geq 0}$.
    \item If $d = 0$, then $E = \OO_P^{\oplus e}$ is Gieseker-semistable and satisfies $\ch(E) = (0, 0, 0, e)$. If $d > 0$, let $C \subset \P^3$ be a curve of degree $d$ and genus $0$. Then $\ch(\OO_C) = (0, 0, d, 1 - 2d)$. By using the tensor product with lines bundles, we can reduce to the case $e = 1 - 2d - e'$ for $e' \in \Z_{\geq 0}$. By taking a length $e'$ closed subscheme $Z \subset C$ of dimension zero, we get $\ch(\OO_C(-Z)) = (0, 0, d, 1 - 2d -e') = (0, 0, d, e)$. \qedhere
\end{enumerate}
\end{proof}

In contrast to the $\ch_2$ case, we do not get nice behavior for $\ch_3$ under dualizing. This is because the sign of $\ch_3$ changes under the derived dual, while it is constant for $\ch_2$.

For objects with maximal Chern characters, we get the following result regarding the vertical wall.

\begin{lem}
\label{lem:vertical_wall_crossing}
Let $E \in \Coh(\P^3)$ be a torsion-free sheaf with $\ch(E) = (r, c, d, e)$ with $d = D(r, c)$ and $e = E(r, c, d)$. Then $E[1]$ is tilt-semistable in an open neighborhood around the vertical wall $\beta = \mu(E)$ if and only if $E$ is slope-semistable and all the stable factors $E$ have parallel $\ch_{\leq 2}$.
\end{lem}

\begin{proof}
Observe that $E[1]$ for a torsion-free sheaf $E \in \Coh(\P^3)$ is tilt-semistable along $\beta = \mu(E)$ if and only if $E$ is slope-semistable. However, while this holds for semistability, it does not for stability.
\begin{enumerate}
    \item Assume that $E[1]$ resp. $E$ is tilt-semistable in an open neighborhood around $\beta = \mu(E)$. By Proposition \ref{prop:large_volume_limit} about the large volume limit we know that $E$ is $2$-Gieseker-semistable. In particular, $E$ is slope-semistable. If $E$ is not slope-stable, then there is a non-trivial short exact sequence
    \[
    0 \to F \to E \to G \to 0
    \]
    in $\Coh(\P^3)$ such that $\mu(F) = \mu(E) = \mu(G)$ and these objects are all slope-semistable. Therefore, the appropriate shift of this sequence is also exact in $\Coh^{\mu(E)}(\P^3)$ making $E$ strictly tilt-semistable. Since $E$ is tilt-stable in an open neighborhood of the vertical wall, this sequence cannot destabilize $E$ for $\beta < \mu(E)$ or destabilize $E[1]$ for $\beta > \mu$. This is implies that $\ch_{\leq 2}(F)$, $\ch_{\leq 2}(E)$, and $\ch_{\leq 2}(G)$ are all parallel.

    \item Assume that $E$ is slope-semistable and all the stable factors of $E$ have parallel $\ch_{\leq 2}$. Then it is least tilt-semistable along $\beta = \mu(E)$. Assume that $E$ is tilt-unstable on either side of the vertical wall. Then there is a non-trivial short exact sequence
    \[
    0 \to F \to E[1] \to G[1] \to 0
    \]
    in $\Coh^{\mu(E)}(\P^3)$ making $E$ strictly-semistable. Without loss of generality, we may assume that $F$ is tilt-stable. In particular, this means that either $F$ or $F[-1]$ is tilt-stable for $\beta < \mu(E)$ and $\alpha \gg 0$. Therefore, we either get that $F$ is the shift of a torsion-free slope-stable sheaf or $F$ is a torsion sheaf. Taking the long exact sequence in cohomology shows that $G$ is a torsion-free sheaf, and we get
    \[
    0 \to \HH^{-1}(F) \to E \to G \to \HH^0(F) \to 0.
    \]
    Note that $\nu_{\alpha, \mu(E)}(E) = \infty$, and therefore, the same has to be true for $F$ and $G$. This implies $\ch_1^{\mu(E)}(F) = \ch_1^{\mu(E)}(G) = 0$. Now by Proposition \ref{prop:large_volume_limit} about the large volume limit we have two possibilities for $F$. Firstly, it could be that $F$ is the shift of a torsion-free $2$-Gieseker-stable sheaf with $\mu(F) = \mu(E)$. Since $F$ destabilizes $E$ to some side of the wall, it cannot be that $\ch_{\leq 2}(F)$ is parallel to $\ch_{\leq 2}(E)$, a contradiction. Secondly, $F$ could be a torsion sheaf supported in dimension less than or equal to one. However, in this case $G$ is an object that contradicts our maximality assumption on the Chern characters of $E$. \qedhere
\end{enumerate}
\end{proof}

We would like to raise a few questions about this. Someone more brave might have made them conjectures.

\begin{question}[Arbitrary rank and $c_1$?]
Is $D(r, c) \neq -\infty$ for any $(r, c) \in \operatorname{\CC\HH}_{\leq 1}(\P^3)$? What if we replace slope stability by Gieseker stability or $2$-Gieseker stability in the definition of $D(r, c)$?
\end{question}

\begin{question}[Difference slope stability and Gieseker stability]
Is there a Gieseker-semistable sheaf $E$ with $\ch_{\leq 2}(E) = (r, c, d)$ for all $(r, c, d) \in \operatorname{\CC\HH}_{\leq 2}(\P^3)$ with $d \leq D(r, c)$?
\end{question}

Both questions have an affirmative answer up to rank $4$ by the results in the upcoming sections of this article.

\begin{question}[Difference slope stability and Gieseker stability II]
For which $(r, c, d, e) \in \operatorname{\CC\HH}(\P^3)$ with $e \leq E(r, c, d, e)$ is there no Gieseker-semistable sheaf $E$ with $\ch(E) = (r, c, d, e)$?
\end{question}

\begin{rmk}
\label{rmk:ch3_gaps}
The last question is reasonable, because there are certainly gaps where there is a $2$-Gieseker-semistable, but not Gieseker-semistable sheaf. For example, for any $P \in \P^3$ the object $\OO \oplus \II_P$ is $2$-Gieseker-semistable with Chern character $(2, 0, 0, -1)$. Assume that $E$ is a Gieseker-semistable sheaf $E$ with $\ch(E) = (2, 0, 0, -1)$.

Assume further that $E$ is slope-stable. Then Theorem \ref{thm:li_bound} shows $\Delta(E) > 0$, a contradiction. Thus, $E$ has to be strictly slope-semistable. This means there is a short exact sequence
\[
0 \to F \to E \to G \to 0
\]
where $F$ and $G$ are slope stable sheaves of rank one with slope zero. Since $\ch_2(F), \ch_2(G) \leq 0$ and $\ch_2(F) + \ch_2(G) = 0$, we must have $\ch_2(F) = \ch_2(G) = 0$. Similarly, $\ch_3(F), \ch_3(G) \leq 0$ and $\ch_3(F) + \ch_3(G) = -1$ implies that one of $\ch_3(F), \ch_3(G)$ has to be $0$ while the other one is $-1$. Since $E$ is Gieseker-semistable we must have $\ch(F) = (1, 0, 0, -1)$ and $\ch_3(G) = (1, 0, 0, 0)$. Therefore, $F = \II_Q$ for some $Q \in \P^3$ and $G = \OO$. Since $H^1(\II_Q) = 0$ this implies $E = \OO \oplus \II_Q$ and $E$ is clearly not Gieseker-semistable.
\end{rmk}

The bounds that Dr\'ezet and Le Potier proved in the case of $\P^2$ (see \cite{DP85:stable_p2}) scale well. More precisely, this means that for objects of positive rank on $\P^2$ one can obtain sharp bounds of $\ch_2(E)/r(E)$ from the slope $\mu(E) = c_1(E)/r(E)$. For $\mu \in \Q$, let $\widetilde{D}(\mu)$ be defined as the supremum over all $D(r, c)/r$ such that $c/r = \mu$. It is an obvious question whether the same holds for $\P^3$.

\begin{question}
Is there a slope-semistable sheaf $E$ with $\ch_{\leq 2}(E) = (r, c, d) \in \operatorname{\CC\HH}_{\leq 2}(\P^3)$ and $r > 0$ if and only if $d/r \leq \widetilde{D}(\mu(E))$? Is $\widetilde{D}: \Q \to \Q$ continuous?
\end{question}

\begin{rmk}
\label{rmk:no_scaling_ch3}
The analog of this question for the third Chern character is wrong. For example due to Theorem \ref{thm:rank_two} there is a slope-semistable sheaf $E$ with $\ch(E) = (2, 0, -6, 16)$. On the other hand, if $\ch(F) = (1, 0, -3, e)$, then $e \leq 6$ due to Theorem \ref{thm:rank_one}.
\end{rmk}


\section{Rank one and rank zero bounds}
\label{sec:rank_zero_and_one}

The following fact was proved in \cite[Proposition 3.2]{MS20:space_curves}. Note that the statement there contains less information than here, but the proof allows to draw these additional conclusions.

\begin{thm}
\label{thm:rank_one}
Let $E \in \Db(\P^3)$ with $\ch(E) = (1, 0, d, e)$ for some $d \in \Z_{\leq 0}$ and $e \in \Z$. Assume that $E$ is either slope-semistable or $\nu_{\alpha, \beta}$-semistable for some $(\alpha, \beta) \in \R_{> 0} \times \R$. Then $e \leq \tfrac{1}{2}d^2 - \tfrac{1}{2}d$. Moreover, $E(1, 0, d) = \tfrac{1}{2}d^2 - \tfrac{1}{2}d$ and if $e = E(1, 0, d)$, then $E$ fits into a short exact sequence
\[
0 \to \OO(-1) \to E \to \OO_V(d) \to 0
\]
for a plane $V \subset \P^3$. Moreover, any such non-trivial extension is given by $\II_C$ for a plane curve $C \subset \P^3$ of degree $-d$ (resp. $E = \OO$ for $d = 0$), and is both slope-stable and $\nu_{\alpha, \beta}$-stable for $(\alpha, \beta)$ strictly above $W(E, \OO(-1))$ (or for all $(\alpha, \beta) \in \R_{> 0} \times \R$ if $d = 0$).
\end{thm}

Note that for sheaves $E$ with $\ch_{\leq 1}(E) = (1, 0)$ there is no difference between slope-stable, slope-semistable and all the other notions in between. Moreover, the fact that we assumed $\ch_1(E) = 0$ is no restriction at all, since we can always use tensor products with line bundles to reduce to this special case.

\begin{cor}
\label{cor:rank_minus_one}
Let $E$ be an object in $\Coh^{\beta}(\P^3)$ for $\beta > 0$ with $\ch(E) = (-1, 0, d, e)$ for some $d \in \Z_{\geq 0}$ and $e \in \Z$. If $E$ is $\nu_{\alpha, \beta}$-semistable for some $\alpha > 0$, then $e \leq \tfrac{1}{2}d^2 + \tfrac{1}{2}d$. Moreover, if $e = \tfrac{1}{2}d^2 + \tfrac{1}{2}d$, then $E$ is $\nu_{\alpha, \beta}$-stable if and only if $(\alpha, \beta)$ lies above the wall $W(E, \OO(1))$ and $E = \D(\II_C)$ for a plane curve $C \subset \P^3$ of degree $d$ (resp. $E = \OO[1]$ for $d = 0$).
\end{cor}

\begin{proof}
By Proposition \ref{prop:derived_dual} there is a $\nu_{\alpha, -\beta}$-object $\tilde{E}$ and a torsion sheaf $T$ supported in dimension zero together with a distinguished triangle
\[
T[-2] \to \tilde{E} \to \D(E) \to T[-1].
\]
Let $t$ be the length of the support of $T$. Then $\ch(\tilde{E}) = (1, 0, -d, e + t)$ and by Theorem \ref{thm:rank_one} we get
\[
e \leq e + t \leq \tfrac{1}{2}d^2 + \tfrac{1}{2}d.
\]
In case of equality, we must have $t = 0$, and thus, Theorem \ref{thm:rank_one} implies $E = \D(\II_C)$ for a plane curve $C$ of degree $d$ (except for the special case $d = 0$, where $E = \D(\OO) = \OO[1]$).
\end{proof}

\begin{cor}
\label{cor:conjecture_rank_one}
Conjecture \ref{conj:irreducible_smooth} holds for rank one objects.
\end{cor}

\begin{proof}
By Theorem \ref{thm:rank_one} the moduli space of Gieseker-semistable sheaves with Chern character $\ch(E) = (1, 0, d, E(1, 0, d))$ is simply the Hilbert scheme of plane curves of degree $-d$. This moduli space is smooth and irreducible. This can for example be seen by the above exact sequences
\[
0 \to \OO(-1) \to \II_C \to \OO_V(-d) \to 0.
\]
They imply that the space is a projective bundle with base the space of planes $V$, i.e., $\P^3$, and fiber $\P(\Ext^1(\OO_V(-d), \OO(-1)))$.
\end{proof}

In Proposition \ref{prop:function_E} we already established that there are Gieseker-semistable sheaves $E$ with $\ch(E) = (0, 0, d, e)$ for any $(0, 0, d, e) \in \operatorname{\CC\HH}(\P^3)$. Therefore, in this section section, we will deal with sheaves $E$ with $\ch(E) = (0, c, d, e)$ with $c \geq 1$.

For any $c \in \Z_{>0}$ and $f \in [0, c) \cap \Z$ we define
\[
\varepsilon(c, f) \coloneqq \frac{f}{2} \left(c - f - 1 + \frac{f}{c} \right).
\]

\begin{thm}
\label{thm:rank_zero}
For any $(0, c, d) \in \operatorname{\CC\HH}_{\leq 2}(\P^3)$, let $d + c^2/2 \equiv -f (\mod c)$. Let $E \in \Db(X)$ with $\ch(E) = (0, c, d, e)$. Assume that $E$ is either $2$-Gieseker-semistable or $\nu_{\alpha, \beta}$-semistable for some $\alpha \in \R_{> 0}$ and $\beta \in \R$. Then 
\[
e \leq \frac{c^3}{24} + \frac{d^2}{2c} - \varepsilon(c, f) = E(0, c, d).
\]
If $e = E(0, c, d)$, then there is a unique wall in tilt stability for objects with Chern character $(0, c, d, e)$. Below this wall there are no semistable objects. Along this wall any $E$ has unique stable Jordan-H\"older factors. What follows is a detailed description of these walls.
\begin{enumerate}
    \item Assume that $c = 1$ and $e = E(0, 1, d)$. Then the unique wall for $E$ in tilt stability is induced by a short exact sequence
    \[
    0 \to \OO\left(d + \frac{1}{2}\right) \to E \to \OO\left(d - \frac{1}{2}\right)[1] \to 0.
    \]
    \item Assume that $c = 2$, $f = 1$, and $e = E(0, 2, d)$. Then the unique wall for $E$ in tilt stability is induced by a short exact sequence
    \[
    0 \to \OO\left( \frac{d + 1}{2} \right)^{\oplus 2} \to E \to \OO\left( \frac{d - 1}{2} \right)^{\oplus 2}[1] \to 0.
    \]
    \item Assume that either $c \geq 3$ and $0 \leq f \leq \tfrac{c}{2}$, or $c = 2$ and $f = 0$. Then the unique wall for $E$ in tilt stability is induced by a short exact sequence
    \[
    0 \to \II_C \left( \frac{d}{c} + \frac{c}{2} + \frac{f}{c} \right) \to E \to \OO \left( \frac{d}{c} - \frac{c}{2} + \frac{f}{c} \right) [1] \to 0
    \]
    where $C$ is a plane curve of degree $f$. 
    \item Assume $c \geq 3$ and $f \geq \tfrac{c}{2}$. Then the unique wall for $E$ in tilt stability is induced by a short exact sequence
    \[
    0 \to \OO \left( \frac{d}{c} + \frac{c}{2} + \frac{f}{c} - 1\right) \to E \to \D(\II_C) \otimes \OO \left( \frac{d}{c} - \frac{c}{2} + \frac{f}{c} - 1\right) \to 0
    \]
    where $C$ is a plane curve of degree $c - f$. 
\end{enumerate}
\end{thm}

We would like to point out that the case $f = \tfrac{c}{2}$ is slightly more complicated than visible at first sight. This is because $\II_C(d/c + c/2 + f/c)$ for a conic $C$ and $\D(\II_{C'}) \otimes \OO( d/c - c/2 + f/c - 1)$ for a plane cubic $C'$ are both strictly semistable along the unique wall. Note that this theorem was almost established in \cite[Theorem 3.4]{MS20:space_curves}. In the following, we will only point out what has to be modified to obtain a complete proof.

\begin{proof}
The case $c = 1$ can be found in \cite[Lemma 5.4]{Sch20:stability_threefolds}. All other bounds were already proved in \cite[Theorem 3.4]{MS20:space_curves}. What remains is the classification of the semistable objects for $c \geq 2$.


Let $c = 2$, $f = 1$, and $e = E(0, 2, d)$. Up to tensoring with line bundles, we can assume that $d = -1$, and therefore, $e = \tfrac{1}{3}$. By Proposition \ref{prop:properties_dest_sequences} on destabilizing sequences there is no wall for $E$ of radius larger than or equal to $1$. The unique numerical wall with radius one for $E$ is given by $W(E, \OO)$. Thus, $E$ is tilt-stable along $W(E, \OO(-4))$ and we can deduce $\Ext^2(\OO, E) = \Ext^1(E, \OO(-4))^{\vee} = 0$. Thus, $\hom(\OO, E) \geq \chi(E) = 2$. By Proposition \ref{prop:line_bundles}, which characterizes line bundles among semistable objects, the cone of the induced morphism $\OO^{\oplus 2} \to E$ is indeed $\OO(-1)^{\oplus 2}[1]$. 

If either $(c, f) = (2, 0)$ or $c \geq 3$, then we can use the proof of \cite[Theorem 3.4]{MS20:space_curves}. That proof was done under the assumption $e > E(0, c, d)$, but (outside of the special cases we have already excluded) it works perfectly fine with $e \geq E(0, c, d)$. There it was shown that $E$ has to be destabilized by either
\begin{equation}
\label{eq:line_bundle_quotient}
0 \to \II_C \left( \frac{d}{c} + \frac{c}{2} + \frac{f}{c} \right) \to E \to \OO \left( \frac{d}{c} - \frac{c}{2} + \frac{f}{c} \right) [1] \to 0
\end{equation}
where $C$ is a plane curve of degree $f$, or by
\begin{equation}
\label{eq:line_bundle_sub}
0 \to \OO \left( \frac{d}{c} + \frac{c}{2} + \frac{f}{c} - 1\right) \to E \to \D(\II_C) \otimes \OO \left( \frac{d}{c} - \frac{c}{2} + \frac{f}{c} - 1\right) \to 0
\end{equation}
where $C$ is a plane curve of degree $c - f$.

A straighforward computation shows that 
\[
\rho\left(E, \OO \left( \frac{d}{c} - \frac{c}{2} + \frac{f}{c} \right)[1]\right) \geq \rho \left(E, \OO \left( \frac{d}{c} + \frac{c}{2} + \frac{f}{c} - 1\right)\right)
\]
if and only if $f \leq \tfrac{c}{2}$. Equality happens if and only if $f = \tfrac{c}{2}$.

Assume that $f < \tfrac{c}{2}$. Then the object
\[
\D(\II_C) \otimes \OO \left( \frac{d}{c} - \frac{c}{2} + \frac{f}{c} - 1\right)
\]
is destabilized by a map onto $\OO(d/c - c/2 + f/c)[1]$ along $W(E, \OO(d/c + c/2 + f/c - 1))$. In this case, we can rule out that $E$ is destabilized by \eqref{eq:line_bundle_sub}.

Assume that $f > \tfrac{c}{2}$. The the object
\[
\II_C \left( \frac{d}{c} + \frac{c}{2} + \frac{f}{c} \right)
\]
is destabilized by a map from $\OO(d/c + c/2 + f/c)$ along $W(E, \OO(d/c - c/2 + f/c)[1])$. In this case, we can rule out that $E$ is destabilized by \eqref{eq:line_bundle_quotient}.

If $f = c/2$, either of the two quotients work, since both
\[
\II_C \left( \frac{d}{c} + \frac{c}{2} + \frac{f}{c} \right) \text{ and } \D(\II_C) \otimes \OO \left( \frac{d}{c} - \frac{c}{2} + \frac{f}{c} - 1\right)
\]
are strictly semistable along the unique wall. In case \eqref{eq:line_bundle_quotient}, we can switch to \eqref{eq:line_bundle_sub} by taking the quotient of the composition
\[
\OO \left( \frac{d}{c} + \frac{c}{2} + \frac{f}{c} - 1\right) \into \II_C \left( \frac{d}{c} + \frac{c}{2} + \frac{f}{c} \right) \into E.
\]
Vice versa, starting with  \eqref{eq:line_bundle_sub}, we can switch to \eqref{eq:line_bundle_quotient} by taking the kernel of the composition
\[
E \onto \D(\II_C) \otimes \OO \left( \frac{d}{c} - \frac{c}{2} + \frac{f}{c} - 1\right) \onto \OO \left( \frac{d}{c} - \frac{c}{2} + \frac{f}{c} \right) [1]. \qedhere
\]
\end{proof}

\begin{cor}
\label{cor:conjecture_rank_zero}
Conjecture \ref{conj:irreducible_smooth} holds for rank zero objects.
\end{cor}

\begin{proof}
For $c = 1$, we are dealing with the moduli space of planes $V \subset \P^3$ which is isomorphic to $\P^3$. Clearly, this is smooth and irreducible.

For $c = 2$ and $f = 1$, we are dealing with the moduli space of quiver representation of the Kronecker quiver (two vertices with four arrows in the same direction) with dimension vector $(2, 2)$. Moduli spaces of representations of quivers without relations are always irreducible and smooth along the stable loci (see \cite{Kin94:moduli_quiver_reps}).

If $c \geq 2$ and $f = 0$, then up to tensoring with line bundles all stable objects $E$ fit into a short exact sequence in $\Coh(\P^3)$
\[
0 \to \OO (-c) \to \OO \to E \to 0.
\]
This means $E = \OO_X$ for a surface $X \subset \P^3$ of degree $c$, and the moduli space is again a projective space.

Let $0 < f < \tfrac{c}{2}$ and $c \geq 3$. Then up to tensoring with line bundles all stable objects $E$ fit into a short exact sequence in $\Coh(\P^3)$
\begin{equation*}
0 \to \OO(-c) \to \II_C \to E \to 0,
\end{equation*}
where $C$ is a plane curve of degree $f$. Note that $H^0(\II_C(c))$ is independent of the plane curve $C$ chosen. Therefore, the space of such non-trivial extensions is a projective bundle over the moduli space of plane curves of degree $f$. We are dealing again with a smooth and irreducible moduli space.

Similarly, if $\tfrac{c}{2} < f < c$ and $c \geq 3$, then up to tensoring with line bundles all stable objects $E$ fit into a short exact sequence in $\Coh^{\beta}(\P^3)$
\begin{equation*}
0 \to \OO \to E \to \D(\II_C) \otimes \OO(-c) \to 0,
\end{equation*}
where $C$ is a plane curve of degree $c - f$. We have $\Ext^1(\D(\II_C) \otimes \OO(-c), \OO) = H^0(\II_C(c))$ which is again independent of $C$. As before, the moduli space is a projective bundle over the space of plane curves of degree $c - f$. This is a smooth and irreducible moduli space.

Lastly, assume that $f = \tfrac{c}{2}$ and $c \geq 3$. We still get that the above short exact sequence
\begin{equation*}
0 \to \OO(-c) \to \II_C \to E \to 0,
\end{equation*}
where $C$ is a plane curve of degree $f$. However, since $\II_C$ is strictly-semistable along the wall, this sequence might not be canonical. It is possible that a single such $E$ has multiple of these sequences. We have
\[
\Ext^1(\II_C, \OO(-c)) = H^2(\II_C(c - 4))^{\vee} = 0.
\]
If there were two such sequences for $E$ than this fact plus an application of the Snake Lemma shows that the sequence is canonical in this case as well.
\end{proof}

The bounds up to this point are sharp for semistable sheaves or equivalently for tilt-semistable objects above the largest semicircular wall. However, once walls are crossed one should expect stronger bounds to hold. Generally, this becomes a much more difficult problem.

\begin{prop}
\label{prop:rank_one_between_-1_-2}
Let $E$ with $\ch(E) = (1, 0, d, e)$  with $d \leq -3$ be $\nu_{\alpha, \beta}$-semistable for $(\alpha, \beta)$ on a numerical wall with center $s$.
Moreover, let $f \equiv d (\mod 2)$ with $f \in \{0, 1\}$. Then
\[
e \leq \begin{cases}
\frac{1}{2} d^2 - \frac{1}{2}d &\text{ if $s \leq d - \frac{1}{2}$,} \\
\tfrac{1}{2}d^2 - ds + s^2 - 2d + 2s + \frac{3}{4} &\text{ if $d - \frac{1}{2} \leq s \leq \tfrac{1}{2}d + \frac{f - 3}{2}$,} \\
\frac{1}{4} d^2 - d - \frac{f}{4} &\text{ if $\tfrac{1}{2}d + \frac{f - 3}{2} \leq s$.}
\end{cases}
\]
\end{prop}

The proof will imply that these bounds are sharp if the numerical wall is an actual wall and $s \leq \tfrac{1}{2}d - 1$, but we will not need this. Note that $s((1, 0, d), \OO(-1)) = d - \tfrac{1}{2}$ and $s((1, 0, d), \OO(-2)) = \tfrac{1}{2}d - 1$. Thus, the bounds in this statement show how $\ch_3$ of rank one sheaves decreases when an object stays stable between the two largest line bundle walls. It is certainly reasonable to ask what happens for smaller walls, but the question becomes more difficult.

\begin{proof}
The numerical wall with center $d - \frac{1}{2}$ is the largest actual wall for such a rank one object and the bound for objects above is simply the one from Theorem \ref{thm:rank_one}. We will prove the remainder by induction on $-d$.

If $d = -3$, then this result is simply a fancy formulation of the wall computation in \cite{Sch20:stability_threefolds} for the Hilbert scheme of twisted cubics. Similarly, the case $d = -4$ is a reformulation of the wall computations for the Hilbert scheme of elliptic quartics in \cite{GHS16:elliptic_quartics}.

From now on let $d \leq -5$. Note that for $s = d - \frac{1}{2}$ we have
\[
\tfrac{1}{2}d^2 - ds + s^2 - 2d + 2s + \frac{3}{4} = \frac{1}{2} d^2 - \frac{1}{2}d.
\]
For $s = \tfrac{1}{2}d + \frac{f - 3}{2}$ we have
\[
\tfrac{1}{2}d^2 - ds + s^2 - 2d + 2s + \frac{3}{4} = \frac{1}{4} d^2 - d + \frac{f(f - 2)}{4} = \frac{1}{4} d^2 - d - \frac{f}{4}.
\]
It is not difficult to see that the claimed bounds are strictly decreasing in $s$. Therefore, it will be enough to show that these bounds hold for objects destabilized at a wall with radius $s$ instead of all objects that are stable along that wall.

Beyond this, the terms involving $f$ can be interpreted as rounding terms making sure that $\chi(E) \in \Z$. Therefore, if $e \leq \tfrac{1}{4} d^2 - d$ we would be done even if $f = 1$. We can assume that $e > \tfrac{1}{4} d^2 - d$. 

Assume that $E$ is destabilized along a semicircular wall induced by a short exact sequence
\[
0 \to F \to E \to G \to 0.
\]
By Proposition \ref{prop:properties_dest_sequences} on destabilizing sequences we may assume $\ch_0(F) > 0$ and $\mu(F) < \mu(E) = 0$. We can compute
\[
\rho^2_Q(E) - \frac{\Delta(E)}{8} = \frac{3e}{2d} + \frac{d}{4} > \frac{9(d+4)^2}{64} > 0.
\]
Again by Proposition \ref{prop:properties_dest_sequences} this implies that $\ch_0(F) = 1$. 
Next, we compute
\[
Q_{0, -3}(E) = 4d^2 - 18d - 18e < -\frac{d^2}{2} < 0.
\]
Therefore, $\ch_1^{-3}(F) > 0$ and together with $\mu(F) < 0$ we get $\ch_1(F) \in \{-1, -2\}$.
\begin{enumerate} 
    \item Assume that $\ch(F) = (1, 0, y, z) \cdot \ch(\OO(-1))$ and $s = s(E, F) = d - y - \tfrac{1}{2} \in [d - \frac{1}{2}, \tfrac{1}{2}d + \frac{f - 3}{2}]$. Then $\tfrac{1}{2}d + \tfrac{1}{2} \leq y \leq 0$, and by induction $z \leq \tfrac{1}{2}y^2 - \tfrac{1}{2}y$. We can use Theorem \ref{thm:rank_zero} on the quotient $G$ to obtain
    \[
    e \leq \frac{1}{2}d^2 - dy + y^2 - \frac{1}{2}d - y = \frac{1}{2}d^2 - ds + s^2 - 2d + 2s + \frac{3}{4}.
    \]
    \item Assume that $\ch(F) = (1, 0, y, z) \cdot \ch(\OO(-1))$ and $s = s(E, F) = d - y - \tfrac{1}{2} < \tfrac{1}{2}d + \frac{f - 3}{2}$. There is no wall with $\tfrac{1}{2}d + \frac{f - 3}{2} < s < \tfrac{1}{2}d - 1$, and thus, we may assume that $s \geq \tfrac{1}{2}d - 1$. This means $y \leq \tfrac{1}{2}d + \tfrac{1}{2}$. Since we established that $Q_{0, -3}(E) < 0$, the wall must be larger than $W(E, \OO(-3))$. This is equivalent to $y \geq \tfrac{2}{3}d + 1$. This means we can obtain a bound for $z$ by induction using the second case in the statement, i.e., 
    \[
    z \leq d^2 - 3dy + \tfrac{5}{2}y^2 + 3d - \tfrac{11}{2}y + 2.
    \]
    We can use Theorem \ref{thm:rank_zero} on the quotient $G$ to obtain
    \[
    e \leq \tfrac{3}{2}d^2 - 4dy + 3y^2 + \tfrac{5}{2}d - 6y + 2.
    \]
    This bounds is increasing in $y$ and thus, we get a bound for all cases of $y$ by setting $y = \tfrac{1}{2}d + \tfrac{1}{2}$, i.e.,
    \[
    e \leq \frac{1}{4} d^2 - d - \frac{1}{4} \leq \frac{1}{4} d^2 - d - \frac{f}{4}.
    \]

    \item Assume $\ch(F) = (1, 0, y, z) \cdot \ch(\OO(-2))$. The fact that the wall is larger than $W(E, \OO(-3))$ is equivalent to $\tfrac{1}{3}d + 1 \leq y$. By induction we have $z \leq \tfrac{1}{2}y^2 - \tfrac{1}{2}y$ and another application of Theorem \ref{thm:rank_zero} we get
    \[
    e \leq \frac{1}{4}d^2 - \frac{1}{2}dy + \frac{3}{4}y^2 - d - \frac{3}{2}y.
    \]
    This bound is again increasing in $y$ and for $y = 0$, we get
    \[
    e \leq \tfrac{1}{4}d^2 - d.
    \]
    As previously, the term $\tfrac{f}{4}$ is just a rounding term that comes for free. \qedhere
\end{enumerate}
\end{proof}

We need a few more details on the cases $d = -3$ and $d = -4$.

\begin{prop}[\cite{Sch20:stability_threefolds}]
\label{prop:final_model_(1, 0, -3, 5)}
If $E$ with $\ch(E) = (1, 0, -3, e)$ is tilt-stable in between $W(E, \OO(-1))$ and $W(E, \OO(-2))$, then $e \leq 5$. If $e = 5$, then $E$ fits into a short exact sequence
\[
0 \to \OO(-2)^{\oplus 3} \to E \to \OO(-3)^{\oplus 2}[1] \to 0.
\]
In particular, a generic such $E$ is the ideal sheaf of a twisted cubic curve. Moreover, the moduli space of them is a smooth and irreducible moduli spaces of representation of the generalized Kronecker quiver with two vertices, four arrows, and dimension vector $(2, 3)$.
\end{prop}

\begin{prop}[\cite{GHS16:elliptic_quartics}]
\label{prop:final_model_(1, 0, -4, 8)}
If $E$ with $\ch(E) = (1, 0, -3, 5)$ is tilt-stable in between $W(E, \II_L(-1))$ for a line $L \subset \P^3$ and $W(E, \OO(-2))$, then $E$ fits into a short exact sequence
\[
0 \to \OO(-2)^{\oplus 2} \to E \to \OO(-4)[1] \to 0.
\]
In particular, a generic such $E$ is the ideal sheaf of an elliptic quartic curve. Moreover, the moduli space of them is $\Gr(2, 10)$.

If $E$ is strictly-tilt-semistable along $W(E, \II_L(-1))$, then it is an extension between $\II_L(-1)$ and $\OO_V(-3)$ for a line $L \subset \P^3$ and a plane $V \subset \P^3$. The moduli space of tilt-stable objects above this wall is the blow-up of $\Gr(2, 10)$ in a sublocus isomorphic to $\Gr(2, 4) \times \P^3$.
\end{prop}


\section{Rank two bounds}
\label{sec:rank_two}

In this section, we summarize the statements for rank two. This was already established in \cite{Sch20:rank_two_p3}. In particular, in that article Conjecture \ref{conj:irreducible_smooth} was proved for rank two objects.

\begin{thm}[{\cite[Theorem 1.1, Theorem 3.1]{Sch20:rank_two_p3}}]
\label{thm:rank_two}
We have $D(2, -1) = -\tfrac{1}{2}$ and $D(2, 0) = 0$. Moreover, $E(2, -1, d) = \tfrac{1}{2}d^2 - d + \tfrac{5}{24}$ for $d \leq -\tfrac{1}{2}$, $E(2, 0, 0) = E(2, 0, -1) = 0$, and finally $E(2, 0, -d) = \tfrac{1}{2}d^2 + \tfrac{1}{2}d + 1$ for $d \leq -2$. The same bounds hold for $\nu_{\alpha, \beta}$-semistable objects.

If $E$ is is slope-semistable or $\nu_{\alpha, \beta}$-semistable for some $(\alpha, \beta) \in \R_{>0} \times \R$ with Chern character $\ch(E) = (2, -1, -\tfrac{1}{2}, \tfrac{5}{6})$, then $E$ is a sheaf that fits into a short exact sequence in $\Coh(\P^3)$ of the form
\[
0 \to \OO(-2) \to \OO(-1)^{\oplus 3} \to E \to 0.
\]
\end{thm}

In \cite[Theorem 3.1]{Sch20:rank_two_p3} there is a precise classification of all the semistable objects with maximal third Chern character in the case of rank two. We will only need this classification for the special case $E \in M(2, -1, -\tfrac{1}{2}, \tfrac{5}{6})$, and therefore will not repeat the complete result here.



\section{Special cases and the second Chern character}
\label{sec:special_cases}

The goal of this section is to prove the following theorem and deal with a few special cases regarding bounds of $\ch_3$.

\begin{thm}
\label{thm:ch_2}
The first values of $D(r, c)$ are given by
\begin{enumerate}
    \item $D(r, 0) = 0$ for all $r > 0$,
    \item $D(2, -1) = -\tfrac{1}{2}$,
    \item $D(3, -1) = -\tfrac{1}{2}$,
    \item $D(4, -1) = -\tfrac{3}{2}$, $D(4, -2) = -1$.
\end{enumerate}
\end{thm}

\subsection{Preparatory lemmas}

\begin{lem}
\label{lem:walls_(4, -1, -1/2)}
There are no walls in tilt stability left of the vertical wall for objects $E$ with Chern character $\ch_{\leq 2}(E) = (4, -1, -\tfrac{1}{2})$.
\end{lem}

\begin{proof}
Assume that there is such a wall induced by a short exact sequence
\[
0 \to F \to E \to G \to 0
\]
where $\ch_{\leq 2}(E) = (4, -1, -\tfrac{1}{2})$ and $\ch_{\leq 2}(F) = (s, x, y)$. By Proposition \ref{prop:properties_dest_sequences} on destabilizing sequences we know that up to exchanging the roles of $F$ and $G$ we can assume $s > 0$ and $\mu(F) < \mu(E)$.

We have $\beta_{-}(E) = - \tfrac{\sqrt{5} + 1}{4}$, and all walls to the left of the vertical wall have to intersect the ray $\beta = \beta_{-}(E)$. By the definition of $\Coh^{\beta}(\P^3)$ we get
\[
0 < \ch_1^{\beta_{-}(E)}(F) = x - \beta_{-}(E) s < \ch_1^{\beta_{-}(E)}(E) = \sqrt{5}.
\]
This implies that
\begin{equation}
\label{eq:bound_mu(F)_(4, -1, -1/2)}
\mu(F) \in \left(- \frac{\sqrt{5} + 1}{4}, \frac{\sqrt{5}}{s} - \frac{\sqrt{5} + 1}{4} \right).
\end{equation}
By Proposition \ref{prop:properties_dest_sequences} on destabilizing sequences we know that 
\[
\beta_{-}(E) \leq \beta_{-}(F) \leq \mu(F) < \mu(E).
\]
If $s \geq 6$, then
\[
-1 < \beta_{+}(F) = 2\mu(F) - \beta_{-}(F) \leq \frac{2\sqrt{5}}{s} - \frac{\sqrt{5} + 1}{2} + \frac{\sqrt{5} + 1}{4} \leq \frac{\sqrt{5} - 3}{12} < 0.
\]
This means that Theorem \ref{thm:li_bound} applies, and we get $\Delta(F) \geq \tfrac{3}{8} s^2$. On the other hand, using Proposition \ref{prop:properties_dest_sequences} on destabilizing sequences and $s \geq 6$ we get
\[
\Delta(F) < \Delta(E) = 5 < \frac{3}{8} s^2,
\]
a contradiction.

Next, if $s = 5$, then \eqref{eq:bound_mu(F)_(4, -1, -1/2)} implies that $x \in \{-2, -3, -4\}$. If $x = -2$, then Bogomolov's inequality implies $y \leq 0$. For $y = 0$, we get that Theorem \ref{thm:li_bound} applies. But that is a contradiction since $\Delta(5, -2, 0) = 4 < \tfrac{3}{8} 5^2$. If $x = -2$ and $y \leq -1$, then numerically $W(F, E)$ is on the right side of the vertical wall. For $x = -3$, we use Bogomolov's inequality again to get $y \leq \tfrac{1}{2}$. Depending on $y$ the numerical wall $W(E, F)$ is either empty or on the right side of the vertical wall. If $x = -4$, then $y \leq 1$. Again, depending on $y$ the numerical wall $W(E, F)$ is either empty or on the right side of the vertical wall.

If $s = 4$, then \eqref{eq:bound_mu(F)_(4, -1, -1/2)} implies that $x \in \{-2, -3\}$. For $x = -2$, we have $y \leq 0$ and for $x = -3$, we have $y \leq \tfrac{1}{2}$. In either case, the numerical wall $W(E, F)$ is either empty or on the wrong side of the vertical wall.

If $s = 3$, then \eqref{eq:bound_mu(F)_(4, -1, -1/2)} implies that $x \in \{-1, -2\}$. For $x = -1$, we have $y \leq -\tfrac{1}{2}$ and for $x = -2$, we have $y \leq 0$. In either case, the numerical wall $W(E, F)$ is either empty or on the wrong side of the vertical wall.

If $s = 2$, then \eqref{eq:bound_mu(F)_(4, -1, -1/2)} together with $\mu(F) < \mu(E)$ implies that $x = -1$. Bogomolov's inequality implies $y \leq -\tfrac{1}{2}$, and again, the numerical wall $W(E, F)$ is either empty or on the wrong side of the vertical wall.

Finally, if $s = 1$, then \eqref{eq:bound_mu(F)_(4, -1, -1/2)} together with $\mu(F) < \mu(E)$ says $-1 < x < 0$, a contradiction.
\end{proof}

\begin{lem}
\label{lem:no_(4, -1, -1/2)}
There is no tilt-stable object $E$ with $\ch_{\leq 2}(E) = (4, -1, -\tfrac{1}{2})$.
\end{lem}

\begin{proof}
By Lemma \ref{lem:walls_(4, -1, -1/2)} we know that $E$ is a $\nu_{\alpha, \beta}$-stable for any $\alpha > 0$ and $\beta < -\tfrac{1}{4}$. In particular, $E$ is tilt-stable along $W(E, \OO(-4)[1])$. Therefore, tilt stability implies both $H^0(E) = \Hom(\OO, E) = 0$ and $H^2(E) = \Hom(E, \OO(-4)[1])^{\vee} = 0$. This allows us to conclude that $0 \geq \chi(E) = e + \tfrac{7}{6}$, i.e., $e \leq -\tfrac{7}{6}$.

Assume that $\ch_3(E) = E(4, -1, -\tfrac{1}{2})$. By Lemma \ref{lem:vertical_wall_crossing} the object $E[1]$ is also $\nu_{\alpha, \beta}$-stable for $\beta > -\tfrac{1}{4}$ as long as $\alpha \gg 0$. We have $\ch(E[1]) = (-4, 1, \tfrac{1}{2}, -e)$, and there is no wall for $E[1]$ along $\beta = 0$. By stability we get $\Hom(E, \OO(-1)) = 0$. Moreover, we can observe the inequality $\nu_{0, 0}(\OO(3)) = \tfrac{3}{2} > \tfrac{1}{2} = \nu_{0, 0}(E[1])$ to conclude $\Ext^2(E, \OO(-1)) = \Hom(\OO(3), E[1])^{\vee} = 0$. This leads to the contradiction $0 \geq \chi(E, \OO(-1)) = -e - \tfrac{1}{6} \geq 1$.
\end{proof}

\begin{lem}
\label{lem:walls_(5, -1, -1/2)}
There are no walls in tilt stability left of the vertical wall for objects $E$ for which $\ch_{\leq 2}(E) = (5, -1, -\tfrac{1}{2})$.
\end{lem}

\begin{proof}
Assume that there is such a wall induced by a short exact sequence
\[
0 \to F \to E \to G \to 0
\]
where $\ch_{\leq 2}(E) = (5, -1, -\tfrac{1}{2})$ and $\ch_{\leq 2}(F) = (s, x, y)$. By Proposition \ref{prop:properties_dest_sequences} on destabilizing sequences we know that up to exchanging the roles of $F$ and $G$ we can assume $s > 0$ and $\mu(F) < \mu(E)$.

We have $\beta_{-}(E) = - \tfrac{\sqrt{6} + 1}{5}$, and all walls to the left of the vertical wall have to intersect the ray $\beta = \beta_{-}(E)$. By the definition of $\Coh^{\beta}(\P^3)$ we get
\[
0 < \ch_1^{\beta_{-}(E)}(F) = x - \beta_{-}(E) s < \ch_1^{\beta_{-}(E)}(E) = \sqrt{6}.
\]
This implies that
\begin{equation}
\label{eq:bound_mu(F)_(5, -1, -1/2)}
\mu(F) \in \left(- \frac{\sqrt{6} + 1}{5}, \frac{\sqrt{6}}{s} - \frac{\sqrt{6} + 1}{5} \right).
\end{equation}
By Proposition \ref{prop:properties_dest_sequences} on destabilizing sequences we know that 
\[
\beta_{-}(E) \leq \beta_{-}(F) \leq \mu(F) < \mu(E).
\]
If $s \geq 7$, then
\[
-1 < \beta_{+}(F) = \mu(F) + \frac{\sqrt{\Delta(F)}}{s} \leq \frac{\sqrt{6}}{s} - \frac{\sqrt{6} + 1}{5} + \frac{\sqrt{5}}{s} \leq  \frac{\sqrt{6}}{7} - \frac{\sqrt{6} + 1}{5} + \frac{\sqrt{5}}{7} < 0.
\]
This means that Theorem \ref{thm:li_bound} applies, and we get $\Delta(F) \geq \tfrac{3}{8} s^2$. On the other hand, we have
\[
\Delta(F) \leq \Delta(E) - 1 = 5 < \frac{3}{8} s^2,
\]
a contradiction.

Next, if $s = 6$, then \eqref{eq:bound_mu(F)_(5, -1, -1/2)} implies that $x \in \{-2, -3, -4\}$. For $x = -2$ the Bogomolov inequality says $y \leq 0$, but the case $y = 0$ is ruled out by Theorem \ref{thm:li_bound}. Thus, $y \leq -1$. For $x = -3$, we have $y \leq \tfrac{1}{2}$ and for $x = -4$ we have $y \leq 1$. In all these cases, depending on $y$ the numerical wall $W(E, F)$ is either empty or on the wrong side of the vertical wall. 

If $s = 5$, then \eqref{eq:bound_mu(F)_(5, -1, -1/2)} implies that $x \in \{-2, -3\}$. If $x = -2$, then Bogomolov's inequality implies $y \leq 0$. For $x = -3$, we use Bogomolov's inequality again to get $y \leq \tfrac{1}{2}$. In both cases, depending on $y$, the numerical wall $W(E, F)$ is either empty or on the right side of the vertical wall.

If $s = 4$, then \eqref{eq:bound_mu(F)_(5, -1, -1/2)} implies that $x \in \{-1, -2\}$. For $x = -1$, we have $y \leq -\tfrac{1}{2}$ and for $x = -2$, we have $y \leq 0$. In either case, the numerical wall $W(E, F)$ is either empty or on the wrong side of the vertical wall.

If $s = 3$, then \eqref{eq:bound_mu(F)_(5, -1, -1/2)} implies that $x \in \{-1, -2\}$. For $x = -1$, we have $y \leq -\tfrac{1}{2}$ and for $x = -2$, we have $y \leq 0$. In either case, the numerical wall $W(E, F)$ is either empty or on the wrong side of the vertical wall.

If $s = 2$, then \eqref{eq:bound_mu(F)_(5, -1, -1/2)} together with $\mu(F) < \mu(E)$ implies that $x = -1$. Bogomolov's inequality implies $y \leq -\tfrac{1}{2}$, and again, the numerical wall $W(E, F)$ is either empty or on the wrong side of the vertical wall.

Finally, if $s = 1$, then \eqref{eq:bound_mu(F)_(5, -1, -1/2)} together with $\mu(F) < \mu(E)$ says $-1 < x < 0$, a contradiction.
\end{proof}

\begin{lem}
\label{lem:no_(5, -1, -1/2)}
There is no tilt-stable object $E$ with $\ch_{\leq 2}(E) = (5, -1, -\tfrac{1}{2})$.
\end{lem}

\begin{proof}
Assume there is such an object $E$. Then by Lemma \ref{lem:walls_(5, -1, -1/2)} the object $E$ is stable along $W(E, \OO(-4)[1])$ and thus, $H^2(E) = \Hom(E, \OO(-4)[1])^{\vee} = 0$. Stability also implies the cohomology vanishing $H^0(E) = \Hom(\OO, E) = 0$. Together this implies $e + \tfrac{13}{6} = \chi(E) \leq 0$, i.e., $e \leq -\tfrac{13}{6}$.

Assume that $\ch_3(E) = E(5, -1, -\tfrac{1}{2})$. By Lemma \ref{lem:vertical_wall_crossing} the object $E[1]$ is also $\nu_{\alpha, \beta}$-stable for $\beta > -\tfrac{1}{5}$ as long as $\alpha \gg 0$. We have $\ch(E[1]) = (-5, 1, \tfrac{1}{2}, -e)$, and there is no wall for $E[1]$ along $\beta = 0$. By stability we get $\Hom(E, \OO(-1)) = 0$. Moreover, we observe the inequality $\nu_{0, 0}(\OO(3)) = \tfrac{3}{2} > \tfrac{1}{2} = \nu_{0, 0}(E[1])$ to conclude $\Ext^2(E, \OO(-1)) = \Hom(\OO(3), E[1])^{\vee} = 0$. This leads to the contradiction $0 \geq \chi(E, \OO(-1)) = -e - \tfrac{1}{6} \geq 2$.
\end{proof}

\begin{lem}
\label{lem:walls_(3, 0, -1)}
There are no walls in tilt stability left of the vertical wall for objects $E$ for which $\ch_{\leq 2}(E) = (3, 0, -1)$.
\end{lem}

\begin{proof}
Assume there is such a wall induced by a short exact sequence
\[
0 \to F \to E \to G \to 0
\]
where $\ch_{\leq 2}(E) = (3, 0, -1)$ and $\ch_{\leq 2}(F) = (s, x, y)$. By Proposition \ref{prop:properties_dest_sequences} on destabilizing sequences we can assume that $s \geq 0$ and $\mu(F) < \mu(E)$. We have $\beta_{-}(F) = -\sqrt{\tfrac{2}{3}}$, and the fact $0 < \ch^{\beta_-}(F) < \ch^{\beta_-}(E) = \sqrt{6}$ implies
\begin{equation}
\label{eq:bound_mu(F)_(3, 0, -1)}
\mu(F) \in \left(-\sqrt{\frac{2}{3}}, \frac{\sqrt{6}}{s} - \sqrt{\frac{2}{3}}\right).
\end{equation}
By Proposition \ref{prop:properties_dest_sequences} on destabilizing sequences we know $\beta_{-}(E) \leq \beta_{-}(F) \leq \mu(F)$. If $s \geq 6$, then
\[
-1 < \beta_{+}(F) = 2\mu(F) - \beta_{-}(F) < \frac{2\sqrt{6}}{s} - 2\sqrt{\frac{2}{3}} + \sqrt{\frac{2}{3}} = \frac{2\sqrt{6}}{s} - \sqrt{\frac{2}{3}} \leq 0.
\]
This means that Theorem \ref{thm:li_bound} applies for $s \geq 6$, and we get $\Delta(F) \geq \tfrac{3}{8} s^2$. On the other hand, by Proposition \ref{prop:properties_dest_sequences} on destabilizing sequences we have $\Delta(F) \leq 5$, a contradiction to $s \geq 6$.

If $s = 5$, then \eqref{eq:bound_mu(F)_(3, 0, -1)} implies that $x \in \{-2, -3, -4\}$. For $x = -2$, we can use Theorem \ref{thm:li_bound} to get $y \leq -1$. For $x = -3$ Bogomolov's inequality implies $y \leq \tfrac{1}{2}$ and for $x = -4$ we have $y \leq 1$. In all three cases, the numerical wall $W(E, F)$ is either empty or on the wrong side of the vertical wall.

If $s = 4$, then \eqref{eq:bound_mu(F)_(3, 0, -1)} implies that $x \in \{-1, -2, -3\}$. For $x = -1$, we can use Bogomolov's inequality together with Lemma \ref{lem:no_(4, -1, -1/2)} to obtain $y \leq -\tfrac{3}{2}$. For $x = -2$, we have $y \leq 0$ and for $x = -3$, we have $y \leq \tfrac{1}{2}$. In either case, the numerical wall $W(E, F)$ is either empty or on the wrong side of the vertical wall.

If $s = 3$, then \eqref{eq:bound_mu(F)_(3, 0, -1)} implies that $x \in \{-1, -2\}$. For $x = -1$, we have $y \leq -\tfrac{1}{2}$ and for $x = -2$, we have $y \leq 0$. In either case, the numerical wall $W(E, F)$ is either empty or on the wrong side of the vertical wall.

If $s = 2$, then \eqref{eq:bound_mu(F)_(3, 0, -1)} together with $\mu(F) < \mu(E)$ implies that $x = -1$. Bogomolov's inequality implies $y \leq -\tfrac{1}{2}$, and again, the numerical wall $W(E, F)$ is either empty or on the wrong side of the vertical wall.

Finally, if $s = 1$, then \eqref{eq:bound_mu(F)_(3, 0, -1)} together with $\mu(F) < \mu(E)$ implies $-1 < x < 0$, a contradiction.
\end{proof}

\begin{lem}
\label{lem:walls_(4, 0, -1)}
There are no walls in tilt stability left of the vertical wall for objects $E$ for which $\ch_{\leq 2}(E) = (4, 0, -1)$.
\end{lem}

\begin{proof}
Assume there is such a wall induced by a short exact sequence
\[
0 \to F \to E \to G \to 0
\]
where $\ch_{\leq 2}(E) = (4, 0, -1)$ and $\ch_{\leq 2}(F) = (s, x, y)$. By Proposition \ref{prop:properties_dest_sequences} on destabilizing sequences we can assume that $s \geq 0$ and $\mu(F) < \mu(E)$. We have $\beta_{-}(F) = -\tfrac{\sqrt{2}}{2}$, and the inequalities $0 < \ch^{\beta_-}(F) < \ch^{\beta_-}(E) = \sqrt{8}$ imply
\begin{equation}
\label{eq:bound_mu(F)_(4, 0, -1)}
\mu(F) \in \left(-\frac{\sqrt{2}}{2}, \frac{\sqrt{8}}{s} - \frac{\sqrt{2}}{2}\right).
\end{equation}
By Proposition \ref{prop:properties_dest_sequences} on destabilizing sequences we know $\beta_{-}(E) \leq \beta_{-}(F) \leq \mu(F)$. If $s \geq 8$, then
\[
-1 < \beta_{+}(F) = \mu(F) + \frac{\sqrt{\Delta(F)}}{s} < \frac{\sqrt{8} + \sqrt{7}}{s} - \frac{\sqrt{2}}{2} \leq \frac{\sqrt{8} + \sqrt{7}}{8} - \frac{\sqrt{2}}{2} < 0.
\]
If $s = 7$, then \eqref{eq:bound_mu(F)_(4, 0, -1)} implies $\mu(F) \leq -\tfrac{3}{7}$, and thus $\beta_{+}(F) \leq -\tfrac{3}{7} + \frac{\sqrt{7}}{7} < 0$. This means that Theorem \ref{thm:li_bound} applies for $s \geq 7$, and we get $\Delta(F) \geq \tfrac{3}{8} s^2$. On the other hand, we have $\Delta(F) \leq 7$, a contradiction to $s \geq 7$.

Next, if $s = 6$, then \eqref{eq:bound_mu(F)_(4, 0, -1)} implies that $x \in \{-2, -3, -4\}$. For $x = -2$ we have $y \leq 0$, but $y = 0$ is ruled out by Theorem \ref{thm:li_bound}, i.e., $y \leq -1$. If $x = -3$, then $y \leq \tfrac{1}{2}$ and if $x = -4$, then $y \leq 1$. In all three cases, the numerical wall $W(E, F)$ is either empty or on the wrong side of the vertical wall.

If $s = 5$, then \eqref{eq:bound_mu(F)_(4, 0, -1)} implies that $x \in \{-1, -2, -3\}$. For $x = -1$, we can use Bogomolov's inequality together with Lemma \ref{lem:no_(5, -1, -1/2)} to get $y \leq -\tfrac{3}{2}$. For $x = -2$, we have $y \leq 0$ and for $x = -3$, we have $y \leq \tfrac{1}{2}$. In all three cases, the numerical wall $W(E, F)$ is either empty or on the wrong side of the vertical wall.

If $s = 4$, then \eqref{eq:bound_mu(F)_(4, 0, -1)} implies that $x \in \{-1, -2\}$. For $x = -1$, we can use Bogomolov's inequality to obtain $y \leq -\tfrac{1}{2}$. For $x = -2$, we have $y \leq 0$. In either case, the numerical wall $W(E, F)$ is either empty or on the wrong side of the vertical wall.

If $s = 3$, then \eqref{eq:bound_mu(F)_(4, 0, -1)} implies that $x \in \{-1, -2\}$. For $x = -1$, we have $y \leq -\tfrac{1}{2}$ and for $x = -2$, we have $y \leq 0$. In either case, the numerical wall $W(E, F)$ is either empty or on the wrong side of the vertical wall.

If $s = 2$, then \eqref{eq:bound_mu(F)_(4, 0, -1)} together with $\mu(F) < \mu(E)$ implies that $x = -1$. Bogomolov's inequality implies $y \leq -\tfrac{1}{2}$, and again, the numerical wall $W(E, F)$ is either empty or on the wrong side of the vertical wall.

Finally, if $s = 1$, then \eqref{eq:bound_mu(F)_(4, 0, -1)} together with $\mu(F) < \mu(E)$ implies $-1 < x < 0$, a contradiction.
\end{proof}

\begin{lem}
\label{lem:no_(4, -2, 0)}
There is no tilt-stable object $E$ with $\ch_{\leq 2}(E) = (4, -2, 0)$.
\end{lem}

\begin{proof}
Assume there is such an object $E$ and assume that $\ch_3(E)$ is maximal. We start by proving that there is no wall left of the numerical vertical wall for such an $e$. Since $\nu_{0, -1}(E) = 0$, any such wall would have to intersect the ray $\beta = -1$. Assume that we have a wall intersecting $\beta = -1$ induced by a short exact sequence
\[
0 \to F \to E \to G \to 0.
\]
By Proposition \ref{prop:properties_dest_sequences} on destabilizing sequences we may assume that $\ch_0(F) > 0$ and $\mu(F) < \mu(E)$. We have $\ch^{-1}_{\leq 2}(E) = (4, 2, 0)$. This implies that $\ch^{-1}_1(F) = 1$. Let $\ch^{-1}_{\leq 2}(F) = (s, 1, y)$. The inequality $\mu(F) < \mu(E)$ implies $s \geq 3$. The equation $\nu_{\alpha, -1}(F) = \nu_{\alpha, -1}(E)$ is equivalent to $\alpha^2 (s - 2) = 2y$. This is only possible if $y > 0$. However, the Bogomolov inequality implies $y \leq \tfrac{1}{2s}$, and there is no value for $y$ in between.

We have shown that there is no wall along $\beta = -1$ and by stability we have both $H^0(E) = 0$ and $H^2(E) = \Hom(E, \OO(-4)[1]) = 0$. Thus, $0 \geq \chi(E) = e + \tfrac{1}{3}$, i.e., $e \leq -\tfrac{1}{3}$.

By Proposition \ref{prop:large_volume_limit} about the large volume limit we know that $E$ is $2$-Gieseker-semistable. Assume that $E$ is strictly slope-semistable. Then it would have two Jordan-H\"older factors both with Chern characters $\ch_{\leq 2} = (2, -1, d)$ for some $d \leq -\tfrac{1}{2}$. But then $\ch_2(E) \leq -1$, a contradiction. Therefore, $E$ must be slope-stable. Since additionally $\ch_3(E)$ and $\ch_2(E)$ are maximal, we can use Lemma \ref{lem:vertical_wall_crossing} to cross the vertical wall and show that $E[1]$ is $\nu_{\alpha, \beta}$-stable for $\beta > -\tfrac{1}{2}$ and $\alpha 
\gg 0$. By Proposition \ref{prop:derived_dual} we get a tilt-semistable object $\tilde{E}$ with $\ch(\tilde{E} \otimes \OO(-1)) = (4, -2, 0, -e + \tfrac{1}{3} + t)$ for some integer $t \geq 0$. The same bound that we proved for $E$ implies $-e + \tfrac{1}{3} + t \leq -\tfrac{1}{3}$, i.e., $e \geq \tfrac{2}{3}$, a contradiction.
\end{proof}

\subsection{Results}

\begin{prop}
\label{prop:rank_three_special_cases}
We have $D(3, -2) = 0$, $D(3, -1) = -\tfrac{1}{2}$, and $D(3, 0) = 0$. Let $E$ be $\nu_{\alpha, \beta}$-semistable for some $(\alpha, \beta) \in \R_{>0} \times \R$ with $\ch(E) = (3, c, d, e)$.
\begin{enumerate}
    \item Assume that $c = -2$ and $d = 0$. Then $e \leq E(3, -2, 0) = \tfrac{2}{3}$. If $e = E(3, -2, 0)$ and $E$ is tilt-stable, then $E$ is the twisted tangent bundle $T(-2)$.
    \item Assume that $c = -1$ and $d = -\tfrac{1}{2}$. Then $e \leq E(3, -1, -\tfrac{1}{2}) = -\tfrac{1}{6}$. If $e = E(3, -1, -\tfrac{1}{2})$  and $E$ is tilt-stable, then $E$ is the twisted cotangent bundle $\Omega(1)$.
    \item Assume that $c = -1$ and $d = -\tfrac{3}{2}$. Then $e \leq E(3, -1, -\tfrac{3}{2}) = \tfrac{11}{6}$. If $e = E(3, -1, -\tfrac{3}{2})$ and $E$ is tilt-stable, then $E$ fits into a short exact sequence of sheaves
    \[
    0 \to \OO(-2)^{\oplus 2} \to \OO(-1)^{\oplus 5} \to E \to 0.
    \]
    \item Assume that $c = 0$ and $d = 0$. Then $e \leq E(3, 0, 0) = 0$. If $e = E(3, 0, 0)$, then $E = \OO^{\oplus 3}$.
    \item Assume that $c = 0$ and $d = -1$. Then $e \leq E(3, 0, -1) = -1$. If $e = E(3, 0, -1)$, then $E$ fits into a short exact sequence
    \[
    0 \to E \to \OO^{\oplus 3} \to \OO_L(2) \to 0
    \]
    for a line $L \subset \P^3$.
    \item Assume that $c = 0$ and $d = -2$. Then $e \leq E(3, 0, -2) = 1$. If $e = E(3, 0, -2)$, then $E$ fits into a short exact sequence
    \[
    0 \to \Omega(1) \to E \to \OO_V(-1) \to 0
    \]
    for a plane $V \subset \P^3$.
\end{enumerate}
\end{prop}

\begin{proof}
The Bogomolov inequality implies $D(3, -2) \leq 0$, $D(3, -1) \leq -\tfrac{1}{2}$, and $D(3, 0) \leq 0$.
\begin{enumerate}
    \item Let $c = -2$ and $d = 0$. Assume that $e \geq \tfrac{2}{3}$. Then $\ch^{-1}_1(E) = 1$ and there is no wall for $E$ along $\beta = -1$. A straightforward computation shows that the wall $W(E, \OO(-1))$ is smaller than $W(E, \OO(-5)[1])$. Therefore, $E$ is tilt-stable along $W(E, \OO(-5)[1])$ and $\Ext^2(\OO(-1), E) = \Hom(E, \OO(-5)[1])^{\vee} = 0$. We have 
    \[
    \hom(\OO(-1), E) \geq \chi(\OO(-1), E) = e + \tfrac{10}{3} \geq 4.
   	\]
    Therefore, $E$ is destabilized along $W(E, \OO(-1))$ and fits into a short exact sequence
    \[
    0 \to \OO(-1)^{\oplus 4} \to E \to G \to 0,
    \]
    where $G$ is tilt-semistable with $\ch(G) = (-1, 2, -2, e + \tfrac{2}{3})$. By     if $E$ is stable along, then Proposition \ref{prop:line_bundles}, which characterizes line bundles among semistable objects, shows that $G = \OO(-2)[1]$ and $e = \tfrac{2}{3}$. We have shown that $E$ fits into the Euler sequence (tensored with a line bundle) and therefore, $E = T(-2)$. Vice-versa, $T(-2)$ is stable, and therefore, we get $D(3, -2) = 0$ and $E(3, -2, 0) = \tfrac{2}{3}$. 
    
    \item Let $c = -1$ and $d = -\tfrac{1}{2}$. Then $\nu_{0, -1}(E) = 0$. This means all walls to the left of the vertical wall intersect the ray $\beta = -1$. Assume that such a wall is induced by a short exact sequence
    \[
    0 \to F \to E \to G \to 0.
    \]
    for tilt-semistable objects $F, G \in \Coh^{-1}(\P^3)$. We can compute $\ch^{-1}_{\leq 2}(E) = (3, 2, 0, e - \tfrac{1}{2})$, and thus, $\ch^{-1}(F) = \ch^{-1}(G) = 1$. Let $\ch^{-1}_{\leq 2}(F) = (s, 1, y)$. Up to exchanging the roles of $F$ and $G$ we can assume that $s \geq 2$. The Bogomolov inequality implies $y \leq \tfrac{1}{2s} \leq \tfrac{1}{4}$, i.e., $y \leq -\tfrac{1}{2}$. On the other hand, the equation $\nu_{\alpha, -1}(E) = \nu_{\alpha, -1}(F)$ is equivalent to
    \[
    \alpha^2 = \frac{4y}{2s - 3}.
    \]
    Since $2s - 3 > 0$, $y < 0$, and $\alpha^2 > 0$, this is a contradiction. We have shown that there is no wall for $E$ to the left of the vertical wall. In particular, $E$ is stable along $W(E, \OO(-4)[1])$. This implies $H^2(E) = \Hom(E, \OO(-4)[1])^{\vee} = 0$. By $2$-Gieseker stability of $E$ we have $H^0(E) = \Hom(\OO, E) = 0$. Thus, $0 \geq \chi(E) = e + \tfrac{1}{6}$, i.e., $e \leq -\tfrac{1}{6}$.
    
    Assume that $e = -\tfrac{1}{6}$. We have a short exact sequence
    \[
    0 \to E \to E^{\vee \vee} \to T \to 0
    \]
    where $T$ is supported in dimension less than or equal to one. Since $E^{\vee \vee}$ is slope-stable, and $\ch_2(E)$, $\ch_3(E)$ are maximal for slope-stable sheaves with $\ch_{\leq 1}(E) = (3, -1)$, we must have $T = 0$. This means $E$ is reflexive. By \cite[Proposition 1.1.10]{HL10:moduli_sheaves} this implies that $\lExt^i(E, \OO_X) = 0$ for $i \neq 0, 1$, and $\lExt^1(E, \OO_X)$ is supported in dimension $0$. If $\ch(\lExt^1(E, \OO_X)) = (0, 0, 0, t)$, then $t \geq 0$ and $\ch(E^{\vee} \otimes \OO(-1)) = (3, -2, 0, \tfrac{2}{3} + t)$. By part (i) this means $t = 0$, and $E^{\vee} \otimes \OO(-1) = T(-2)$, i.e., $E = \Omega(1)$.
    
    \item Let $c = -1$ and $d = -\tfrac{3}{2}$, and assume that $e \geq \tfrac{11}{6}$. We start by showing that there is no wall along $\beta = -1$. Assume that there is such a wall induced by a short exact sequence
    \[
    0 \to F \to E \to G \to 0.
    \]
    By Proposition \ref{prop:properties_dest_sequences} on destabilizing sequences we can assume $\ch_0(F) > 0$ and $\mu(F) < \mu(E)$. One computes $\ch^{-1}_{\leq 2}(E) = (3, 2, -1)$ and thus, $\ch^{-1}(F) = (s, 1, y)$ for $s > 0$. The inequality $\mu(F) < \mu(E)$ implies $s \geq 2$. The Bogomolov inequality is equivalent to $y \leq \tfrac{1}{2s} \leq \tfrac{1}{4}$. The equation $\nu_{\alpha, -1}(F) = \nu_{\alpha, -1}(E)$ is equivalent to $(2s - 3) \alpha^2 = 4y + 2$ which implies $y > -\tfrac{1}{2}$. However, there is no possible $y$ in this small intervall. 
    
    In particular, $E$ is stable along the numerical wall $W(E, \OO(-5)[1])$ that intersects $\beta = -1$. By stability we get $\Ext^2(\OO(-1), E) = \Hom(E, \OO(-5)[1])^{\vee} = 0$. This yields the inequality $\hom(\OO(-1), E) \geq \chi(\OO(-1), E) \geq 5$, and we get a short exact sequence
    \[
    0 \to \OO(-1)^{\oplus 5} \to E \to G \to 0
    \]
    where $\ch(G) = (-2, 4, -4, e + \tfrac{5}{6})$. By Proposition \ref{prop:line_bundles}, which characterizes line bundles among semistable objects, we get $G = \OO(-2)^{\oplus 2}[1]$ and $e = \tfrac{11}{6}$.
    
    \item Let $c = 0$ and $d = 0$. We have shown in Proposition \ref{prop:function_E} that $E(3, 0, 0) = 0$. If $E$ is slope-semistable with $\ch(E) = (3, 0, 0, 0)$, then by Proposition \ref{prop:line_bundles}, which characterizes line bundles among semistable objects, we get $E = \OO^{\oplus 3}$.

    \item Let $c = 0$ and $d = -1$. By stability we must have $H^0(E) = \Hom(\OO, E) = 0$. By Lemma \ref{lem:walls_(3, 0, -1)} there are no walls to the left of the vertical wall for $E$. This means that that $E$ is $\nu_{\alpha, \beta}$-semistable for $(\alpha, \beta)$ strictly below the wall $W(E, \OO(-4)[1])$. Therefore, we get $H^2(E) = \Hom(E, \OO(-4)[1])^{\vee} = 0$, and hence, $0 \geq \chi(E) = e + 1$, i.e., $e = -1$.
    
    Again by Lemma \ref{lem:walls_(3, 0, -1)} we know that if $E$ is tilt-semistable, then it has to be a $2$-Gieseker-stable sheaf as well. We have a short exact sequence
    \[
    0 \to E \to E^{\vee \vee} \to T \to 0
    \]
    where $T$ is supported in dimension less than or equal to one. If $T$ is supported in dimension zero, then $\ch_{\leq 2}(E^{\vee \vee}) = (3, 0, -1)$ and $\ch_3(E^{\vee \vee}) > -1$, a contradiction to what we have already shown. Therefore, $T$ must be supported in dimension one, and $\ch_{\leq 2}(E^{\vee \vee}) = (3, 0, 0)$. Therefore, $E^{\vee \vee}$ is a slope-semistable reflexive sheaf with $\ch(E^{\vee \vee}) = (3, 0, 0, z)$ for $z \leq 0$. This means $\ch(E^{\vee}) = \ch(E^{\vee \vee \vee}) = (3, 0, 0, z')$ with $0 \leq -z \leq z' \leq 0$, i.e., $z = 0$ and by Proposition \ref{prop:line_bundles}, which characterizes line bundles among semistable objects, we have $\ch(E) = \OO^{\oplus 3}$. This means $\ch(T) = (0, 0, 1, 1)$ and we get $T = \OO_L(2)$ for some line $L \subset \P^3$.
    
    \item Let $c = 0$, and $d = -2$. The numerical wall $W(E, \Omega(1))$ is given by 
    \[
    \alpha^2 + \left(\beta + \frac{3}{2}\right)^2 = \frac{11}{12}.
    \]
    We will show first that there is no larger wall. Assume there is a wall induced by a short exact sequence
    \[
    0 \to F \to E \to G \to 0
    \]
    such that $W(E, F)$ is larger than or equal to $W(E, \Omega(1))$. Let $\ch_{\leq 2}(F) = (s, x, y)$. By Proposition \ref{prop:properties_dest_sequences} on destabilizing sequences we may assume $\ch_0(F) > 0$ and $\mu(F) < \mu(E)$. We have
    \[
    \frac{\Delta(E)}{16} = \frac{3}{4} < \frac{11}{12}.
    \]
    Proposition \ref{prop:properties_dest_sequences} on destabilizing sequences also implies that $s \in \{1, 2, 3\}$. This wall has to intersect the ray $\beta = -1$. Since $\ch_1^{-1}(E) = 3$, we get $x + s \in \{1, 2\}$. Together with $\mu(F) < \mu(E)$ this immediately rules out $s = 1$. If $s = 2$, then $x = -1$ and $y \leq -\tfrac{1}{2}$. A straightforward computation shows that any such wall is smaller than $W(E, \Omega(1))$. The same holds for the case $s = 3$ and $x = -2$, where $y \leq 0$ and all walls are smaller. If $s = 3$ and $x = -1$, then the wall $W(E, \Omega(1))$ is the largest such wall that happens precisely for $y = -\tfrac{1}{2}$.
    
    Assume that $e = E(3, 0, -2) \geq 1$. Then $\ch_3(F)$ and $\ch_3(G)$ have to be maximal as well, i.e., $\ch(F) = (3, -1, -\tfrac{1}{2}, -\tfrac{1}{6})$ and $\ch(G) = (0, 1, -\tfrac{3}{2}, \tfrac{7}{6})$. This shows that $F = \Omega(1)$ and $G = \OO_V(-1)$ for a plane $V \subset \P^3$. For objects destabilized at this wall, we get $e = 1$. We have to show that there are no stable objects below this wall. Note that 
\[
s(\Omega(1), E) = -\tfrac{3}{2} > -\tfrac{59}{26} = s(E, \Omega(-3)[1]).
\]
Therefore, if $E$ is stable below $W(\Omega(1), E)$, then stability implies $\Hom(\Omega(1), E) = 0$ and $\Ext^2(\Omega(1), E) = \Hom(E, \Omega(-3)[1])^{\vee} = 0$. We get the contradiction $1 \leq \chi(\Omega(1), E) \leq 0$. \qedhere
\end{enumerate}
\end{proof}

\begin{prop}
\label{prop:classification_(3,-1,-5/2)}
Let $E$ be $\nu_{\alpha, \beta}$-semistable with Chern character $\ch(E) = (3, -1, -\tfrac{5}{2}, e)$. Then $e \leq E(3, -1, -\tfrac{5}{2}) = \tfrac{23}{6}$. Assume that $e = E(3, -1, -\tfrac{5}{2})$. Then there are three walls in tilt stability for such objects $E$.
\begin{enumerate}
    \item There are no semistable objects below the smallest wall $\alpha^2 + (\beta + 2)^2 = 1$ which is induced by short exact sequences of the form
    \[
    0 \to \OO(-1)^{\oplus 4} \to E \to \OO(-3)[1] \to 0.
    \]
    \item The second wall is given by $\alpha^2 + (\beta + \tfrac{5}{2})^2 = \tfrac{35}{12}$. Let $E$ be strictly $\nu_{\alpha, \beta}$-semistable along this wall. If $E$ is stable below the wall, then it fits into a non-trivial short exact sequence
    \[
    0 \to \OO_V(-2) \to E \to T(-2) \to 0
    \]
    for a plane $V \subset \P^3$. If $E$ is stable above the wall, then it fits into a non-trivial short-exact sequence
    \[
    0 \to T(-2) \to E \to \OO_V(-2) \to 0
    \]
    for a plane $V \subset \P^3$.
    \item The third and largest wall is given by $\alpha^2 + (\beta + \tfrac{7}{2})^2 = \tfrac{33}{4}$. Let $E$ be strictly $\nu_{\alpha, \beta}$-semistable along this wall. If $E$ is stable below the wall, then it fits into a non-trivial short exact sequence
    \[
    0 \to \II_C \to E \to F \to 0
    \]
    where $C \subset \P^3$ is a conic and $F \in M(2, -1, -\tfrac{1}{2}, \tfrac{5}{6})$. If $E$ is stable above the wall, then it fits into a non-trivial short-exact sequence
    \[
    0 \to F \to E \to \II_C \to 0
    \]
    where $C \subset \P^3$ is a conic and $F \in M(2, -1, -\tfrac{1}{2}, \tfrac{5}{6})$.
\end{enumerate}
\end{prop}

\begin{proof}
Assume that $e \geq \tfrac{23}{6}$. We start by classifying walls above $W(E, \OO(-1))$. Assume that we have such a wall induced by a short exact sequence
\[
0 \to F \to E \to G \to 0.
\]
Let $\ch(F) = (s, x, y, z)$ and by Proposition \ref{prop:properties_dest_sequences} on destabilizing sequences we may assume (up to exchanging the roles of $F$ and $G$) $s > 0$ and $\mu(F) < \mu(E)$. We can compute
\[
\frac{\Delta(v)}{16} = 1 = \rho^2(E, \OO(-1)).
\]
By Proposition \ref{prop:properties_dest_sequences} on destabilizing sequences we get $s \in \{1, 2, 3\}$. Since $\ch^{-1}_{\leq 2}(E) = (3, 2, -2)$, we get $x + s = 1$. Together with $\mu(F) < \mu(E)$, this immediately rules out $s = 1$.

Let $s = 2$. Then $x = -1$ and $y \leq -\tfrac{1}{2}$. The only way that the wall is larger than $W(E, \OO(-1))$ is if $y = -\tfrac{1}{2}$. By Theorem \ref{thm:rank_two} we get $z \leq \tfrac{5}{6}$. The quotient satisfies $\ch(G) = (1, 0, -2, e - z)$, and by Theorem \ref{thm:rank_one} we have $e - z \leq 3$. Since we assumed $e \geq \tfrac{23}{6}$ this is only possible if $z = \tfrac{5}{6}$ and $e = \tfrac{23}{6}$. In that case, we end up with the wall described in (iii).

Let $s = 3$. Then $x = -2$ and $y \leq 0$. The only way that the wall is larger than $W(E, \OO(-1))$ is if $y = 0$. In this case, Proposition \ref{prop:rank_three_special_cases} implies $z \leq \tfrac{2}{3}$. The quotient satisfies $\ch(G) = (0, 1, -\tfrac{5}{2}, e - z)$, and by Theorem \ref{thm:rank_zero} we have $e - z \leq \tfrac{19}{6}$. Since we assumed $e \geq \tfrac{23}{6}$ this is only possible if $z = \tfrac{2}{3}$ and $e = \tfrac{23}{6}$. In that case, we end up with the wall described in (ii).

Assume that that $E$ is stable below $W(E, T(-2))$. Then $E$ is also stable below $W(E, \OO(-5)[1])$ and stability implies $\Ext^2(\OO(-1), E) = \Hom(E, \OO(-5)[1])^{\vee} = 0$. Therefore, we obtain the inequality $\hom(\OO(-1), E) \geq \chi(\OO(-1), E) = e + \tfrac{1}{6} \geq 4$, and we get the final wall (i).
\end{proof}

\begin{prop}
\label{prop:classification_(3,0,-3)}
Let $E$ be $\nu_{\alpha, \beta}$-semistable with $\ch(E) = (3, 0, -3, e)$. Then $e \leq E(3, 0, -3) = 3$. Assume that $e = E(3, 0, -3)$. Then there are two walls in tilt stability for such objects $E$.
\begin{enumerate}
    \item There are no semistable objects below the smallest wall $\alpha^2 + (\beta + \tfrac{3}{2})^2 = \tfrac{1}{4}$ which is induced by short exact sequences of the form
    \[
    0 \to \OO(-1)^{\oplus 6} \to E \to \OO(-2)^{\oplus 3}[1] \to 0.
    \]
    \item The second and largest wall is given by $\alpha^2 + (\beta + \tfrac{5}{2})^2 = \tfrac{17}{4}$. Let $E$ be strictly $\nu_{\alpha, \beta}$-semistable along this wall. If $E$ is stable below the wall, then it fits into a non-trivial short exact sequence
    \[
    0 \to \OO_V(-2) \to E \to \Omega(1) \to 0
    \]
    for a plane $V \subset \P^3$. If $E$ is stable above the wall, then it fits into a non-trivial short-exact sequence
    \[
    0 \to \Omega(1) \to E \to \OO_V(-2) \to 0
    \]
    for a plane $V \subset \P^3$.
\end{enumerate}
\end{prop}

\begin{proof}
Assume that $e \geq 3$. Let $W$ be a wall that intersects the ray $\beta = -1$ induced by a short exact sequence
\[
0 \to F \to E \to G \to 0
\]
with $\ch(F) = (s, x, y, z)$. By Proposition \ref{prop:properties_dest_sequences} on destabilizing sequences we can assume that $\ch_0(F) > 0$ and $\mu(F) < \mu(E)$. The wall $W(E, \OO(-1))$ has center $s(E, \OO(-1)) = -\tfrac{3}{2}$ and radius $\rho(E, \OO(-1)) = \tfrac{1}{2}$. Therefore, we get $s(E, F) < \tfrac{3}{2}$. Moreover, we can compute
\[
\frac{\Delta(E)}{4 \cdot 6 \cdot 3} = \frac{1}{4} = \rho(E, \OO(-1))^2.
\]
Therefore, Proposition \ref{prop:properties_dest_sequences} implies $s \leq 5$. The inequalities $0 < \ch_1^{-1}(F) = x + s < \ch_1^{-1}(E) = 3$, $\mu(F) < \mu(E)$, and $s \leq 5$ lead to the finite list
\[
(s, x) \in \{ (2, -1), (3, -1), (3, -2), (4, -2), (4, -3), (5, -3), (5, -4) \}.
\]
\begin{enumerate}
    \item Let $(s, x) = (2, -1)$. Then Bogomolov's inequality implies $y \leq -\tfrac{1}{2}$, but $s(E, F) < -\tfrac{3}{2}$ implies $y > -\tfrac{1}{2}$.
    \item Let $(s, x) = (3, -1)$. Then Bogomolov's inequality implies $y \leq -\tfrac{1}{2}$ and $s(E, F) < -\tfrac{3}{2}$ implies $s > -\tfrac{3}{2}$. Overall, this means $y = -\tfrac{1}{2}$. By Proposition \ref{prop:rank_three_special_cases} we have $z \leq -\tfrac{1}{6}$. Together with $e \geq 3$ and Theorem \ref{thm:rank_zero} applied to $G$, we get that $e = 3$ and $z = -\tfrac{1}{6}$. We are dealing with the wall described in $(ii)$.
    \item Let $(s, x) = (3, -2)$. Then Bogomolov's inequality implies $y \leq 0$, but$s(E, F) < -\tfrac{3}{2}$ implies $y > 0$.
    \item Let $(s, x) = (4, -2)$. Then Lemma \ref{lem:no_(4, -2, 0)} together with Bogomolov's inequality says $y \leq -1$. On the other hand, $s(E, F) < -\tfrac{3}{2}$ implies $y > -1$.
    \item Let $(s, x) = (4, -3)$. Then Bogomolov's inequality implies $y \leq \tfrac{1}{2}$, but $s(E, F) < -\tfrac{3}{2}$ implies $y > \tfrac{1}{2}$.
    \item Let $(s, x) = (5, -3)$. Then Theorem \ref{thm:li_bound} implies $y \leq -\tfrac{1}{2}$, but $s(E, F) < -\tfrac{3}{2}$ implies $y > -\tfrac{1}{2}$.
    \item Let $(s, x) = (5, -4)$. Then Bogomolov's inequality implies $y \leq 1$, but $s(E, F) < -\tfrac{3}{2}$ implies $y > 1$.
\end{enumerate}
If $E$ is stable below the wall described in (ii), then it is stable below $W(E, \OO(-5)[1])$. Stability implies $\Ext^2(\OO(-1), E) = \Hom(E, \OO(-5)[1])^{\vee}$, and thus, 
\[
\hom(\OO(-1), E) \geq \chi(\OO(-1), E) = e + 3 \geq 6.
\]
The sequence induced by $\OO(-1)^{\oplus 6} \to E$ leads to wall $(i)$.
\end{proof}

\begin{prop}
\label{prop:rank_four_special_cases}
We have $D(4, -3) = -\tfrac{1}{2}$, $D(4, -2) = -1$, $D(4, -1) = -\tfrac{3}{2}$, and $D(4, 0) = 0$. Let $E$ be $\nu_{\alpha, \beta}$-semistable for some $(\alpha, \beta) \in \R_{>0} \times \R$ with $\ch(E) = (4, c, d, e)$.
\begin{enumerate}
    \item Assume that $c = -2$ and $d = -1$. Then $e \leq E(4, -2, -1) = \tfrac{5}{3}$. If $e = E(4, -2, -1)$ and $E$ is tilt-semistable, then $E$ fits into a short exact sequence
    \[
    0 \to \OO(-1)^{\oplus 6} \to E \to \OO(-2)^{\oplus 2}[1] \to 0.
    \]
    
    \item Assume that $c = -1$ and $d = -\tfrac{3}{2}$. Then $e \leq E(4, -1, -\tfrac{3}{2}) = \tfrac{5}{6}$. If $e = E(4, -1, -\tfrac{3}{2})$ and $E$ is tilt-stable, then $E$ fits into a short exact sequence
    \[
    0 \to \Omega(1) \to E \to \II_L \to 0.
    \]
    for a line $L \subset \P^3$.
    
    \item Assume that $c = 0$ and $d = 0$. Then $e \leq E(4, 0, 0) = 0$. If $e = E(4, 0, 0)$, then $E = \OO^{\oplus 4}$.
    
    \item Assume that $c = 0$ and $d = -1$. Then $e \leq E(4, 0, -1) = -2$. If $e = E(4, 0, -1)$, then $E$ fits into a short exact sequence
    \[
    0 \to E \to \OO^{\oplus 4} \to \OO_L(3) \to 0
    \]
    for a line $L \subset \P^3$.
    
    \item Assume that $c = 0$ and $d = -2$. Then $e \leq E(4, 0, -2) = 0$. If $e = E(4, 0, -2)$, then
    $E$ fits into a short exact sequence
    \[
    0 \to \OO(-1)^{\oplus 2} \to \Omega(1)^{\oplus 2} \to E \to \to 0.
    \]
    
    \item Assume that $c = 0$ and $d = -3$. Then $e \leq E(4, 0, -3) = 2$. If $e = E(4, 0, -3)$, then $E$ fits into a short exact sequence
    \[
    0 \to \Omega(1) \to E \to G(1) \to 0
    \]
    where $G \in M^{\alpha, \beta}(1, 0, -3, 5)$ for $(\alpha, \beta) \in W(E, \Omega(1))$. The generic such $G$ is given by $\II_C$ where $C \subset \P^3$ is a twisted cubic.
\end{enumerate}
\end{prop}

\begin{proof}
The Bogomolov inequality implies $D(4, -2) \leq 0$, $D(4, -1) \leq -\tfrac{1}{2}$, and $D(4, 0) \leq 0$. By Lemma \ref{lem:no_(4, -1, -1/2)} we get $D(4, -1) \leq -\tfrac{3}{2}$ and by Proposition \ref{prop:function_D} this means $D(4, -3) \leq -\tfrac{1}{2}$. Moreover, Lemma \ref{lem:no_(4, -2, 0)} shows $D(4, -2) \leq -1$.

\begin{enumerate}
    \item Let $c = -2$ and $d = -1$. Assume that $e \geq \tfrac{5}{3}$ and that $E$ is destabilized above the wall $W(E, \OO(-1))$ by a short exact sequence
    \[
    0 \to F \to E \to G \to 0
    \]
    with $\ch^{-1}(F) = (s, x, y, z)$. Since $\ch^{-1}_1(E) = 2$ we must have $x = 1$. By Proposition \ref{prop:properties_dest_sequences} on destabilizing sequences we may assume that $s > 0$ and $\mu(F) < \mu(E) = -\tfrac{1}{2}$. Since $\mu(F) = \frac{1}{s} - 1$, this immediately implies $s \geq 3$. The equation $\nu_{\alpha, -1}(F) = \nu_{\alpha, -1}(E)$ is equivalent to $(s - 2) \alpha^2 = 2y + 1$, i.e., $y > -\tfrac{1}{2}$. Bogomolov's inequality implies $y \leq \tfrac{x^2}{2s} < \tfrac{x^2}{6}$ and there is no $y$ fitting in this small interval.
    
    This means that $E$ has to be semistable below the wall $W(E, \OO(-5)[1])$ that is larger than $W(E, \OO(-1))$. By stability, we get $\Ext^2(\OO(-1), E) = \Hom(E, \OO(-5)[1])^{\vee} = 0$. Therefore, $6 \leq e + \tfrac{13}{3} = \chi(\OO(-1), E) \leq \hom(\OO(-1), E)$. By Proposition \ref{prop:line_bundles}, which characterizes line bundles among semistable objects, the quotient of the map $\OO(-1)^{\oplus 6} \to E$ is $\OO(-2)^{\oplus 2}$.
    
    \item Let $c = -1$ and $d = -\tfrac{3}{2}$ and assume that $e \geq \tfrac{5}{6}$. We compute $s(E, \Omega(1)) = -\tfrac{5}{2}$, $s(E, \Omega(-3)[1]) = - \tfrac{229}{98}$, and $\rho^2(E, \Omega(-3)[1]) = \tfrac{34017}{9604}$. In particular, $W(E, \Omega(1))$ is slightly larger than $W(E, \Omega(-3)[1])$. If $E$ is tilt-stable below $W(E, \Omega(-3)[1])$, then we get the vanishing $\Ext^2(\Omega(1), E) = \Hom(E, \Omega(-3)[1])) = 0$, and the inequality
	\[
    \hom(\Omega(1), E) \geq \chi(\Omega(1), E) = 3e - \tfrac{3}{2} \geq 1.
    \]
    But this means that $E$ was already destabilized at $W(\Omega(1), E)$. We get that $E$ is destabilized above or at the wall $W(E, \Omega(-3)[1])$ by a short exact sequence
    \[
    0 \to F \to E \to G \to 0
    \]
    with $\ch(F) = (s, x, y, z)$. By Proposition \ref{prop:properties_dest_sequences} on destabilizing sequences we may assume $s > 0$ and $\mu(F) < \mu(E) = -\tfrac{1}{4}$. The wall $W(E, \Omega(-3)[1])$ intersects the point
    \[
    (\alpha, \beta) = \left(0, -\frac{229}{98} + \sqrt{\frac{34017}{9604}}\right).
    \]
    This implies
    \[
    \mu(E) \in \left(-\frac{229}{98} + \sqrt{\frac{34017}{9604}}, -\frac{1}{4} \right).
    \]
    Moreover, $\Delta(E) = 13$ and Proposition \ref{prop:properties_dest_sequences} on destabilizing sequences imply $s \leq 4$. This is only possible if $s = 3$ and $x = -1$. In that case $y \leq -\tfrac{1}{2}$, but only the wall for $y = -\tfrac{1}{2}$ is large enough and equal to $W(E, \Omega(1))$.
    
    We get $z \leq -\tfrac{1}{6}$ and Theorem \ref{thm:rank_one} applied to $G$ shows
    \[
    e \leq z + 1 \leq \tfrac{5}{6}.
    \]
    In case of equality, we must $F = \Omega(1)$ and $G = \II_L$ for a line $L \subset \P^3$.
    
    \item Let $c = 0$ and $d = 0$. We have shown in Proposition \ref{prop:function_E} that $E(4, 0, 0) = 0$. If $E$ is slope-semistable with $\ch(E) = (4, 0, 0, 0)$, then by Proposition \ref{prop:line_bundles}, which characterizes line bundles among semistable objects, we get $E = \OO^{\oplus 4}$.
    
    \item Let $c = 0$ and $d = -1$. Assume that $e \geq -2$. By stability we must have the vanishing $H^0(E) = \Hom(\OO, E) = 0$. By Lemma \ref{lem:walls_(4, 0, -1)} there are no walls to the left of the vertical wall for $E$. This means that that $E$ is $\nu_{\alpha, \beta}$-semistable for $(\alpha, \beta)$ strictly below the wall $W(E, \OO(-4)[1])$. Then $H^2(E) = \Hom(E, \OO(-4)[1])^{\vee} = 0$, and therefore, $0 \geq \chi(E) = e + 2$, i.e., $e = -2$.
    
    Again by Lemma \ref{lem:walls_(4, 0, -1)} we know that if $E$ is tilt-semistable, then it has to be a $2$-Gieseker-stable sheaf as well. We have a short exact sequence of sheaves
    \[
    0 \to E \to E^{\vee \vee} \to T \to 0
    \]
    where $T$ is supported in dimension less than or equal to one. If $T$ is supported in dimension zero, then $\ch_{\leq 2}(E^{\vee \vee}) = (4, 0, -1)$ and $\ch_3(E^{\vee \vee}) > -2$, a contradiction to what we have just shown. Therefore, $T$ must be supported in dimension one, and $\ch_{\leq 2}(E^{\vee \vee}) = (4, 0, 0)$. Moreover, $E^{\vee \vee}$ is a slope-semistable reflexive sheaf with $\ch(E^{\vee \vee}) = (4, 0, 0, z)$ for $z \leq 0$. This means $\ch(E^{\vee}) = \ch(E^{\vee \vee \vee}) = (4, 0, 0, z')$ with $0 \leq -z \leq z' \leq 0$, i.e., $z = 0$ and by Proposition \ref{prop:line_bundles}, which characterizes line bundles among semistable objects, we have $\ch(E) = \OO^{\oplus 4}$. This means $\ch(T) = (0, 0, 1, 2)$ and therefore, $T = \OO_L(3)$ for some line $L \subset \P^3$.
    
    \item Let $c = 0$ and $d = -2$. Assume that $e \geq 0$. The equation $\nu_{0, \beta}(E) = 0$ is equivalent to $\beta = \pm 1$. This means all semicircular walls for $\beta < 0$ intersect the line $\beta = -1$. Assume that such a wall is induced by a short exact sequence
    \[
    0 \to F \to E \to G \to 0
    \]
    with $\ch^{-1}_{\leq 2}(F) = (s, x, y)$. Since $\ch^{-1}_1(E) = 4$, we have $x \in \{1, 2, 3\}$. By Proposition \ref{prop:properties_dest_sequences} on destabilizing sequences we may assume that $s > 0$ and $\mu(F) = \tfrac{x}{s} - 1 < \mu(E) = 0$, i.e., $x < s$. This immediately rules out $s = 1$. Bogomolov's inequality says $y \leq \tfrac{x^2}{2s}$. The equation $\nu_{\alpha, -1}(F) = \nu_{\alpha, -1}(E)$ is equivalent to $\alpha^2 (s-x) = y$, i.e., $y > 0$. The only numbers that satisfy all these inequalities are $4 \leq s \leq 9$, $x = 3$, and $y = \tfrac{1}{2}$. In that case, we get $0 \leq \Delta(G) = 5 - s$, i.e., $s \in \{4, 5\}$. For $s = 4$, we have $\ch_{\leq 2}(F) = (4, -1, -\tfrac{1}{2})$, a case ruled out by Lemma \ref{lem:no_(4, -1, -1/2)}. For $s = 5$, we have $\ch_{\leq 2}(F) = (5, -2, 0)$ and such an $F$ does not exist by Theorem \ref{thm:li_bound}. So far we have shown that there is no wall to the left of the vertical wall.
    
    In particular, we know that $E$ is stable along $W(E, \OO(-4)[1])$ and therefore, the vanishing $H^2(E) = \Hom(E, \OO(-4)[1])^{\vee} = 0$ holds. Moreover, stability also implies the fact that $H^0(E) = \Hom(\OO, E) = 0$ and we get $e = \chi(E) \leq 0$, i.e., $e = 0$.
    
    Additionally, we know that $E$ is stable along $W(E, \Omega(-3)[1])$. Therefore, the equalities $\Ext^2(\Omega(1), E) = \Hom(E, \Omega(-3)[1])$ and $\hom(\Omega(1), E) \geq \chi(\Omega(1), E) = 2$ hold. Those maps from $\Omega(1)$ induce a short exact sequence
    \[
    0 \to \Omega(1)^{\oplus 2} \to E \to \OO(-1)[1] \to 0
    \]
    in $\Coh^{-1}(\P^3)$.
    
    \item Let $c = 0$ and $d = -3$. Assume that $e \geq 2$. Assume that there is a strictly bigger wall $W$ for $E$ than $W(E, \Omega(1)$ induced by a short exact sequence
    \[
    0 \to F \to E \to G' \to 0
    \]
    where $\ch(F) = (s, x, y, z)$. By Proposition \ref{prop:properties_dest_sequences} on destabilizing sequences we may assume that $\mu(F) < \mu(E)$ and $s > 0$. Note that the wall $W(E, \Omega(1)$ intersects $(\alpha, \beta) = (0, -\tfrac{1}{2})$. This means that $W$ must intersect the ray $\beta = -\tfrac{1}{2}$. We get $\mu(F) \in (-\tfrac{1}{2}, 0)$. The computation
    \[
    \rho^2(E, \Omega(1)) - \frac{\Delta(F)}{20} = \frac{25}{16} - \frac{6}{5} > 0 
    \]
    and Proposition \ref{prop:properties_dest_sequences} on destabilizing sequences imply that $s \leq 4$. The only two possibilities are $(s, x) = (3, -1)$ and $(s, x) = (4, -1)$
    
    If $(s, x) = (4, -1)$, then Bogomolov's inequality together with Lemma \ref{lem:no_(4, -1, -1/2)} imply $y \leq -\tfrac{3}{2}$. But such a wall is either smaller than $W(E, \Omega(1))$, empty, or on the wrong side of the vertical wall. Therefore, we must have $(s, x) = (3, -1)$, where $y \leq -\tfrac{1}{2}$. The largest such wall occurs for $y = -\tfrac{1}{2}$ and is $W(E, \Omega(1))$. 
    
    Next, we can compute that $W(E, \Omega(-3)[1])$ is larger than $W(E, \Omega(1))$ and stability implies $\Ext^2(\Omega(1), E)) = \Hom(E, \Omega(-3)[1]) = 0$. This yields
    \[
    \hom(\Omega(1), E) \geq \chi(\Omega(1), E) \geq 1.
    \]
    Therefore, we get a destabilizing sequence
    \[
    0 \to \Omega(1) \to E \to G' \to 0.
    \]
    The quotient $G'$ satisfies $\ch(G' \otimes \OO(-1)) = (1, 0, -3, e - 3)$. By Proposition \ref{prop:final_model_(1, 0, -3, 5)}, we get $e - 3 \leq 5$, i.e., $e = 2$. Moreover, the same Proposition also implies the statement about the generic such $G'$. \qedhere
\end{enumerate}
\end{proof}

\begin{prop}
\label{prop:classification_(4,-2,-2)}
Let $E$ be $\nu_{\alpha, \beta}$-semistable with $\ch(E) = (4, -2, -2, e)$. Then the inequality $e \leq E(4, -2, -2) = \tfrac{11}{3}$ holds. Assume that $e = E(4, -2, -2)$. Then there are two walls in tilt stability for such objects $E$.
\begin{enumerate}
    \item There are no semistable objects below the smallest wall $\alpha^2 + (\beta + 2)^2 = 1$ which is induced by short exact sequences of the form
    \[
    0 \to \OO(-1)^{\oplus 5} \to E \to \OO(-3)[1] \to 0.
    \]
    \item The second and largest wall is given by $\alpha^2 + (\beta + 3)^2 = 5$. Let $E$ be strictly $\nu_{\alpha, \beta}$-semistable along this wall. If $E$ is stable below the wall, then it fits into a non-trivial short exact sequence
    \[
    0 \to \II_C \to E \to T(-2) \to 0
    \]
    for a conic $C \subset \P^3$. If $E$ is stable above the wall, then it fits into a non-trivial short-exact sequence
    \[
    0 \to T(-2) \to E \to \II_C \to 0
    \]
    for a conic $C \subset \P^3$.
\end{enumerate}
\end{prop}

\begin{proof}
Assume that $e \geq \tfrac{11}{3}$. The all $W(E, \OO(-1))$ has radius $1$ and center $s(E, \OO(-1)) = -2$. Assume that $E$ is destabilized by a wall $W$ above $W(E, \OO(-1))$ induced by a short exact sequence
\[
0 \to F \to E \to G \to 0
\]
with $\ch(F) = (s, x, y, z)$. By Proposition \ref{prop:properties_dest_sequences} on destabilizing sequences we may assume that $s > 0$ and $\mu(F) < \mu(E) = -\tfrac{1}{2}$. Since $W$ is larger than $W(E, \OO(-1))$, we get $\mu(F) \in (-1, -\tfrac{1}{2})$. Since $\Delta(E) = 20$, we can use Proposition \ref{prop:properties_dest_sequences} on destabilizing sequences to see that $s \leq 4$. We are left with $(s, x) = (4, -3)$ and $(s, x) = (3, -2)$.

Assume that $(s, x) = (4, -3)$. By Proposition \ref{prop:rank_four_special_cases} we get $y \leq -\tfrac{1}{2}$. It turns out that the wall is either empty or on the wrong side of the vertical wall. Therefore, we must have $(s, x) = (3, -2)$. If $y \leq -1$, then the wall is either empty or on the wrong side of the vertical wall. Therefore, we are left with $y = 0$. Then Proposition \ref{prop:rank_three_special_cases} says that $z \leq \tfrac{2}{3}$ and equality implies $F = T(-2)$. We can apply Theorem \ref{thm:rank_one} to the quotient $G$ to get $e - \tfrac{2}{3} \leq e - z \leq 3$, i.e., $e = \tfrac{11}{3}$ and thus, $G = \II_C$, where $C$ is a conic.

If $E$ is stable below $W(E, T(-2))$, then it is stable all the way up to $W(E, \OO(-1))$. In particular, it is stable along $W(E, \OO(-5)[1])$ and this implies $\Ext^2(\OO(-1), E) = \Hom(E, \OO(-5)[1]))^{\vee} = 0$. Thus, $\hom(\OO(-1), E) \geq \chi(\OO(-1), E) = e + \tfrac{4}{3} \geq 5$. This implies the remaining wall and that there are no stable objects below.
\end{proof}

\begin{prop}
\label{prop:classification_(4,-1,-5/2)}
Let $E$ be $\nu_{\alpha, \beta}$-semistable with $\ch(E) = (4, -1, -\tfrac{5}{2}, e)$. Then the inequality $e \leq E(4, -1, -\tfrac{5}{2}) = \tfrac{17}{6}$ holds. Assume that $e = E(4, -1, -\tfrac{5}{2})$. Then there are two walls in tilt stability for such objects $E$.
\begin{enumerate}
    \item There are no semistable objects below the smallest wall $\alpha^2 + (\beta + \tfrac{3}{2})^2 = \tfrac{1}{4}$ which is induced by short exact sequences of the form
    \[
    0 \to \OO(-1)^{\oplus 7} \to E \to \OO(-2)^{\oplus 3}[1] \to 0.
    \]
    \item The second and largest wall is given by $\alpha^2 + (\beta + \tfrac{11}{2})^2 = \tfrac{105}{4}$. Let $E$ be strictly $\nu_{\alpha, \beta}$-semistable along this wall. If $E$ is stable below the wall, then it fits into a non-trivial short exact sequence
    \[
    0 \to \II_C \to E \to \Omega(1) \to 0
    \]
    for a conic $C \subset \P^3$. If $E$ is stable above the wall, then it fits into a non-trivial short-exact sequence
    \[
    0 \to \Omega(1) \to E \to \II_C \to 0
    \]
    for a conic $C \subset \P^3$.
\end{enumerate}
\end{prop}

\begin{proof}
Assume that $e \geq \tfrac{17}{6}$. Let $E$ be destabilized by a wall $W$ above $W(E, \OO(-1))$ induced by a short exact sequence
\[
0 \to F \to E \to G \to 0
\]
where $\ch(F) = (s, x, y, z)$. By Proposition \ref{prop:properties_dest_sequences} on destabilizing sequences we may assume that $s > 0$ and $\mu(F) < \mu(E) = -\tfrac{1}{4}$. Since $W$ is larger than $W(E, \OO(-1))$, it has to intersect the rays $\beta = -1$ and $\beta = -2$. Therefore, $\ch^{-1}(F) > 0$ and $\ch^{-2}(G) > 0$. Overall, we get
\[
\mu(F) \in \left(-1, \min \left\{-\frac{1}{4}, \frac{7}{s} - 2\right\}\right).
\]
Since $\Delta(E) = 21$ and the wall is larger than $W(E, \OO(-1))$, we can use Proposition \ref{prop:properties_dest_sequences} on destabilizing sequences to see that $s \leq 6$. We are left with
\[
(s, x) \in \{ (2, -1), (3, -1), (3, -2), (4, -2), (4, -3), (5, -4) \}.
\]
We have to rule out all cases except $(s, x) = (3, -1)$. If $(s, x) = (2, -1)$, then Bogomolov's inequality implies $y \leq -\tfrac{1}{2}$ and the wall is smaller than or equal to $W(E, \OO(-1))$. If $(s, x) = (3, -2)$, then Bogomolov's inequality implies $y \leq 0$ and the wall is smaller than or equal to $W(E, \OO(-1))$. If $(s, x) = (4, -2)$, then Bogomolov's inequality and Lemma \ref{lem:no_(4, -2, 0)} imply $y \leq -1$ and the wall is smaller than or equal to $W(E, \OO(-1))$. If $(s, x) = (4, -3)$, then Bogomolov's inequality implies $y \leq \tfrac{1}{2}$ and the wall is smaller than or equal to $W(E, \OO(-1))$. If $(s, x) = (5, -4)$, then Bogomolov's inequality implies $y \leq 1$ and the wall is smaller than or equal to $W(E, \OO(-1))$.

We have shown $(s, x) = (3, -1)$. We have $y \leq -\tfrac{1}{2}$. If the inequality is strict, than the wall is again smaller than or equal to $W(E, \OO(-1))$. Thus, $y = -\tfrac{1}{2}$ and we already know $z \leq -\tfrac{1}{6}$. We can apply Theorem \ref{thm:rank_one} to the quotient $G$ to obtain $e + \tfrac{1}{6} \leq e - z \leq 3$, i.e., $e = \tfrac{17}{6}$, and $z = -\tfrac{1}{6}$. This implies $F = \Omega(1)$ and $G = \II_C$ for a conic $C \subset \P^3$. 

Next, assume that $E$ is stable below $W(E, \Omega(1))$. We have already shown that there is no wall until $W(E, \OO(-1))$. In particular, the object $E$ is stable below $W(E, \OO(-5)[1])$ and this implies $\Ext^2(\OO(-1), E) = \Hom(E, \OO(-5)[1])^{\vee} = 0$, i.e., $\hom(\OO(-1), E) \geq \chi(\OO(-1), E) = 7$. The induced map $\OO(-1)^{\oplus 7} \to E$ leads to the final wall.
\end{proof}

\begin{prop}
\label{prop:classification_(4,0,-4)}
Let $E$ be $\nu_{\alpha, \beta}$-semistable with $\ch(E) = (4, 0, -4, e)$. Then $e \leq E(4, 0, -4) = 4$. Assume that $e = E(4, 0, -4)$. Then there are two walls in tilt stability for such objects $E$.
\begin{enumerate}
    \item There are no semistable objects below the smallest wall $\alpha^2 + (\beta + \tfrac{3}{2})^2 = \tfrac{1}{4}$ which is induced by short exact sequences of the form
    \[
    0 \to \OO(-1)^{\oplus 8} \to E \to \OO(-2)^{\oplus 4}[1] \to 0.
    \]
    \item The second and largest wall is given by $\alpha^2 + (\beta + \tfrac{5}{2})^2 = \tfrac{17}{4}$. Let $E$ be strictly $\nu_{\alpha, \beta}$-semistable along this wall. If $E$ is stable below the wall, then it fits into a non-trivial short exact sequence
    \[
    0 \to G(1) \to E \to \Omega(1) \to 0
    \]
    where $G \in M^{\alpha, \beta - 1}(1, 0, -4, 8)$ for $(\alpha, \beta) \in W(E, \Omega(1))$ or into
    \[
    0 \to \OO_V(-2) \to E \to F \to 0
    \]
    where $F \in M(4, -1, -\tfrac{3}{2}, \tfrac{5}{6})$ and $V \subset \P^3$ is a plane. If $E$ is stable above the wall, then it fits into a non-trivial short-exact sequence
    \[
    0 \to \Omega(1) \to E \to G(1) \to 0
    \]
    where $G \in M^{\alpha, \beta - 1}(1, 0, -4, 8)$ for $(\alpha, \beta) \in W(E, \Omega(1))$.
\end{enumerate}
\end{prop}

\begin{proof}
Assume that $e \geq 4$. Let $E$ be destabilized by a wall $W$ above $W(E, \OO(-1))$ induced by a short exact sequence
\[
0 \to F \to E \to G' \to 0
\]
where $\ch(F) = (s, x, y, z)$. By Proposition \ref{prop:properties_dest_sequences} on destabilizing sequences we may assume that $s > 0$ and $\mu(F) < \mu(E) = 0$. Note that it is possible that the quotient $G'$ satisfies these inequalities instead of $E$ and we will come back to that possibility. Since $W$ is larger than $W(E, \OO(-1))$, it has intersect the ray $\beta = -1$ and $\beta = -2$. Therefore, $\ch^{-1}(F) > 0$ and $\ch^{-2}(G) > 0$. Overall, we get
\[
\mu(F) \in \left(-1, \min \left\{0, \frac{8}{s} - 2\right\}\right).
\]
Since $\Delta(E) = 32$ and the wall is larger than $W(E, \OO(-1))$, we can use Proposition \ref{prop:properties_dest_sequences} on destabilizing sequences to see that $s \leq 7$. We are left with
\[
(s, x) \in \{ (2, -1), (3, -1), (3, -2), (4, -1), (4, -2), (4, -3), (5, -3), (5, -4), (6, -5) \}.
\]
Next will rule out all cases except $(s, x) = (3, -1)$ and $(s, x) = (4, -1)$. If $(s, x) = (2, -1)$, then Bogomolov's inequality implies $y \leq -\tfrac{1}{2}$ and the wall is smaller than or equal to $W(E, \OO(-1))$. If $(s, x) = (3, -2)$, then Bogomolov's inequality implies $y \leq 0$ and the wall is smaller than or equal to $W(E, \OO(-1))$. If $(s, x) = (4, -2)$, then Bogomolov's inequality and Lemma \ref{lem:no_(4, -2, 0)} imply $y \leq -1$ and the wall is smaller than or equal to $W(E, \OO(-1))$. If $(s, x) = (4, -3)$, then Bogomolov's inequality implies $y \leq \tfrac{1}{2}$ and the wall is smaller than or equal to $W(E, \OO(-1))$. If $(s, x) = (5, -3)$, then Bogomolov's inequality together with Theorem \ref{thm:li_bound} implies $y \leq -\tfrac{1}{2}$ and the wall is smaller than or equal to $W(E, \OO(-1))$. If $(s, x) = (5, -4)$, then Bogomolov's inequality implies $y \leq 1$ and the wall is smaller than or equal to $W(E, \OO(-1))$. If $(s, x) = (6, -5)$, then Bogomolov's inequality implies $y \leq \tfrac{3}{2}$ and the wall is smaller than or equal to $W(E, \OO(-1))$.

Assume that $(s, x) = (4, -1)$. Then by Lemma \ref{lem:no_(4, -1, -1/2)} we know $y \leq -\tfrac{3}{2}$. We must have $y = -\tfrac{3}{2}$, since otherwise the wall is smaller than or equal to $W(E, \OO(-1))$. By Proposition \ref{prop:rank_four_special_cases} we know $z \leq \tfrac{5}{6}$. We can use Theorem \ref{thm:rank_zero} on $G'$ to get $e - \tfrac{5}{6} \leq e - z \leq \tfrac{19}{6}$, i.e., $e = 4$ and $z = \tfrac{5}{6}$. Again by Proposition \ref{prop:rank_four_special_cases} we know that $F$ fits into a short exact sequence
\[
0 \to \Omega(1) \to F \to \II_L \to 0.
\]
It turns out that $W(F, \Omega(1)) = W(E, \Omega(1))$, i.e., $F$ is strictly-semistable along the wall, and the object $E$ is also destabilized by the stable subobject $\Omega(1)$. Therefore, this case is subsumed by the case $(s, x) = (3, -1)$ as we are dealing really with the subobject. If $F$ is instead the quotient, then we get the sequence
\[
0 \to \OO_V(-2) \to E \to F \to 0
\]
where $V \subset \P^3$ is a plane.

Now let $(s, x) = (3, -1)$. Then $y \leq -\tfrac{1}{2}$ and the fact that our wall is larger than $W(E, \OO(-1))$ implies $y = -\tfrac{1}{2}$. By Proposition \ref{prop:rank_three_special_cases} we know $z \leq -\tfrac{1}{6}$. The quotient satisfies 
\[
	\ch(G' \otimes \OO_X(-1)) = (1, 0, -4, e - z + \tfrac{23}{6}).
\]
Proposition \ref{prop:rank_one_between_-1_-2} implies $e + 4 \leq e - z + \tfrac{23}{6} \leq 8$, i.e., $e = 4$ and $z = -\tfrac{1}{6}$. That also leads to $F = \Omega(1)$.

Next, assume that $E$ is stable below $W(E, \Omega(1))$. We have already shown that there is no wall until potentially $W(E, \OO(-1))$. In particular, $E$ is stable along $W(E, \OO(-5)[1])$ and this implies $\Ext^2(\OO(-1), E) = \Hom(E, \OO(-5)[1])^{\vee} = 0$, i.e., $\hom(\OO(-1), E) \geq \chi(\OO(-1), E) = e + 4 \geq 8$. The induced map $\OO(-1)^{\oplus 8} \to E$ leads to the final wall.
\end{proof}


\section{Rank three and the third Chern character}
\label{sec:rank_three}

\begin{thm}
\label{thm:rank_three}
We have equalities $E(3, -2, d) = \tfrac{1}{2}d^2 - \tfrac{3}{2}d + \tfrac{2}{3}$ for $d \leq 0$, $E(3, -1, -\tfrac{1}{2}) = -\tfrac{1}{6}$, $E(3, -1, d) = \tfrac{1}{2}d^2 + \tfrac{17}{24}$ for $d \leq -\tfrac{3}{2}$, $E(3, 0, 0) = 0$, $E(3, 0, -1) = -1$, and $E(3, 0, d) = \tfrac{1}{2}d^2 + \tfrac{1}{2}d$ for $d \leq -2$. The same bounds hold for tilt-semistable objects.
\end{thm}

The bounds on $\ch_3$ in the above theorem do not hold for sheaves $E$ that are slope-semistable, but not $2$-Gieseker-semistable. For example, let $C$ be a plane degree $-d$ curve. Then $\OO \oplus \II_C$ is slope-semistable, but not $2$-Gieseker-semistable with Chern character $(3, 0, d, \tfrac{1}{2}d^2 - \tfrac{1}{2}d)$.

\begin{lem}
\label{lem:rank_three_classification_general}
Let $E$ be $\nu_{\alpha, \beta}$-semistable for some $(\alpha, \beta) \in \R_{\geq 0} \times \R$ with $\ch(E) = (3, c, d, e)$.
\begin{enumerate}
    \item Let $c = -2$ and $d \leq -1$. Then $e \leq \tfrac{1}{2} d^2 - \tfrac{3}{2} d + \tfrac{2}{3}$. In case of equality $E$ is destabilized by a short exact sequence
    \[
    0 \to \OO(-1)^{\oplus 3} \to E \to \OO_V(d - 1) \to 0.
    \]
    \item Let $c = -1$ and $d \leq -\tfrac{7}{2}$. Then $e \leq \tfrac{1}{2} d^2 + \tfrac{17}{24}$. In case of equality $E$ is destabilized by one of the following short exact sequences
    \begin{align*}
    0 \to F \to &E \to \II_C \to 0, \\
    0 \to \II_C \to &E \to F \to 0
    \end{align*}
    where $F \in M(2, -1, -\tfrac{1}{2}, \tfrac{5}{6})$ and $C$ is a plane curve of degree $- d - \tfrac{1}{2}$, or
    \[
    0 \to T(-2) \to E \to \OO_V\left(d + \frac{1}{2}\right) \to 0
    \]
    where $V \subset \P^3$ is a plane.
    \item Let $c = 0$ and $d \leq - 4$. Then $e \leq \tfrac{1}{2}d^2 + \tfrac{1}{2}d$. In case of equality $E$ is destabilized by a short exact sequence
    \[
    0 \to \Omega(1) \to E \to \OO_V(d + 1) \to 0.
    \]
\end{enumerate}
\end{lem}

\begin{proof}
The proof is by induction on $\Delta(E)$. The start of the induction are not the cases in this statement, but the special cases established in Section \ref{sec:special_cases}. Assume that $e$ is larger than or equal to the claimed bound. We will deal with the situation case by case.
\begin{enumerate}
    \item If $c = -2$ and $d \leq -1$, then $e \geq \tfrac{1}{2} d^2 - \tfrac{3}{2} d + \tfrac{2}{3}$.
    A straightforward computation implies
    \begin{align*}
    \rho^2_Q(E) - \frac{\Delta(E)}{16} &= \frac{24d^3 - 12d^2 + 108de + 81e^2 - 48e}{(6d - 4)^2} + \frac{3d}{8} - \frac{1}{4} \\
    &\geq \frac{(81d^2 - 120d + 45)d^2}{(12d - 8)^2} > 0.
    \end{align*}
    This implies two things. Firstly, $\rho^2_Q(E) > 0$ means that $E$ has to be destabilized along a semicircular wall. Assume this wall is induced by a short exact sequence
    \[
    0 \to F \to E \to G \to 0.
    \]
    By Proposition \ref{prop:properties_dest_sequences} on destabilizing sequences we can assume that both $\ch_0(F) \geq 1$ and $\mu(F) < \mu(E)$. Secondly, Proposition \ref{prop:properties_dest_sequences} also implies $\ch_0(F) \in [1, 3]$. We can compute
    \[
    Q_{0, -\tfrac{4}{3}}(E) = 4d^2 - 16d - 12e + \frac{64}{9} \leq -2d^2 + 2d - \frac{8}{9} < 0.
    \]
    By definition of $\Coh^{\beta}(\P^3)$ this implies $\ch^{-4/3}(F) > 0$ and together with $\mu(F) < \mu(E)$ we obtain
    \begin{equation} \label{eq:(3, -2, d)_x_bound}
        -\frac{4}{3} \ch_0(F) < \ch_1(F) < -\frac{2}{3} \ch_0(F).
    \end{equation}
    \begin{enumerate}
        \item If $\ch_0(F) = 1$, then \eqref{eq:(3, -2, d)_x_bound} implies $\ch_1(F) = -1$.
        Assume that 
        \[
        	\ch(F) = (1, 0, y, z) \cdot \ch(\OO(-1)).
        \]
        Then
        \[
        0 \geq s(E, F) - s_Q(E) \geq \left(d - 3y - \frac{3}{2}\right) - \frac{9d^2 - 23d + 12}{12d - 8}
        \]
        which implies
        \[
        \frac{d^2 - d}{12d - 8} \leq y \leq 0.
        \]
        We can use the bound from Theorem \ref{thm:rank_one} on $F$ and the bound from \ref{thm:rank_two} on $G$ to obtain
        \[
        e \leq \frac{1}{2}d^2 - dy + y^2 - \frac{3}{2}d + \frac{2}{3}.
        \]
        This is a parabola in $y$ with minimum at $y = \tfrac{1}{2}d$ which is smaller than our range of allowed $y$. Therefore, the maximum occurs at $y = 0$, i.e., 
        \[
        e = \frac{1}{2} d^2 - \frac{3}{2} d + \frac{2}{3}.
        \]
        In this case, $F = \OO(-1)$, $\ch_0(G) = (2, -1, d - \tfrac{1}{2}, E(2, -1, d - \tfrac{1}{2}))$, and by Theorem \ref{thm:rank_two} this object $G$ has further maps from $\OO(-1)$ that factor through $E$. This means $G$ is strictly-semistable and this case will be subsumed by the following cases.
        
        \item If $\ch_0(F) = 2$, then \eqref{eq:(3, -2, d)_x_bound} implies $\ch_1(F) = -2$.
        Assume that 
        \[
        	\ch(F) = (2, 0, y, z) \cdot \ch(\OO(-1)).
        \]        
        Then
        \[
        0 \geq s(E, F) - s_Q(E) \geq \left(d - \frac{3}{2}y - \frac{3}{2}\right) - \frac{9d^2 - 23d + 12}{12d - 8}
        \]
        which implies
        \[
        \frac{d^2 - d}{6d - 4} \leq y \leq 0.
        \]
        If $y \leq -1$, then we use the bound from Theorem \ref{thm:rank_two} one $F$ and the bound from \ref{thm:rank_one} on $G$ to obtain
        \[
        e \leq \frac{1}{2}d^2 - dy + y^2 - \frac{3}{2}d + \frac{5}{3}.
        \]
        As in the previous case, this parabola in $y$ is increasing in our range of $y$ and thus, the maximum occurs at $y = -1$ where we get
        \[
        e = \frac{1}{2} d^2 - \frac{1}{2} d + \frac{5}{3} \leq \frac{1}{2} d^2 - \frac{3}{2} d + \frac{2}{3}.
        \]
        If $y = 0$, then we also use Theorem \ref{thm:rank_two} and Theorem \ref{thm:rank_one} to get
        \[
        e = \frac{1}{2} d^2 - \frac{3}{2} d + \frac{2}{3}.
        \]
        In this case, $F = \OO(-1)^{\oplus 2}$ and the quotient satisfies $\ch(G) = (1, 0, d - 1, E(2, -1, d - 1))$ and by Theorem \ref{thm:rank_one} this object $G$ has a further map from $\OO(-1)$ that factors through $E$. Thus, $G$ is strictly-semistable and this case will be subsumed by the final case.
        
        \item If $\ch_0(F) = 3$, then \eqref{eq:(3, -2, d)_x_bound} implies $\ch_1(F) = -3$.
        Assume that 
        \[
        	\ch(F) = (3, 0, y, z) \cdot \ch(\OO(-1)).
        \]
        By induction, we have $z \leq \tfrac{1}{2}y^2 + \tfrac{1}{2}y$. Note that this bound is not sharp for $y = -1$, but nonetheless true. Then
        \[
        0 \geq s(E, F) - s_Q(E) \geq \left(d - y - \frac{3}{2}\right) - \frac{9d^2 - 23d + 12}{12d - 8}
        \]
        which implies
        \[
        \frac{3d^2 - 3d}{12d - 8} \leq y \leq 0.
        \]
        Using the bound from Theorem \ref{thm:rank_zero} on $G$ implies
        \[
        e \leq \frac{1}{2}d^2 - dy + y^2 - \frac{3}{2}d + \frac{2}{3}.
        \]
        As in the previous two cases the maximum occurs at $y = 0$ and
        \[
        e = \frac{1}{2} d^2 - \frac{3}{2} d + \frac{2}{3}.
        \]
        In particular, $F = \OO(-1)^{\oplus 3}$ and we obtain the sequence
        \[
        0 \to \OO(-1)^{\oplus 3} \to E \to \OO_V(d - 1) \to 0.
        \]
    \end{enumerate}
    
    \item If $c = -1$ and $d \leq -\tfrac{7}{2}$, then $e \geq \tfrac{1}{2} d^2 + \tfrac{17}{24}$. A straightforward computation implies
    \begin{align*}
    \rho^2_Q(E) - \frac{\Delta(E)}{16} &= \frac{24d^3 - 3d^2 + 54de + 81e^2 - 6e}{(6d - 1)^2} + \frac{3d}{8} - \frac{1}{16} \\
    &\geq \frac{(216d^3 + 148d^2 + 106d + 155)(6d + 15)}{(48d - 8)^2} > 0.
    \end{align*}
    This implies two things. Firstly, $\rho^2_Q(E) > 0$ means that $E$ has to be destabilized along a semicircular wall. Assume this wall is induced by a short exact sequence
    \[
    0 \to F \to E \to G \to 0.
    \]
    By Proposition \ref{prop:properties_dest_sequences} on destabilizing sequences we can assume that both $\ch_0(F) \geq 1$ and $\mu(F) < \mu(E)$. Secondly, Proposition \ref{prop:properties_dest_sequences} implies $\ch_0(F) \in [1, 3]$. We can compute
    \[
    Q_{0, -1}(E) = 4d^2 - 8d - 12e + 1 \leq -\frac{1}{2} (2d + 5)(2d + 3) < 0.
    \]
    Therefore, $\ch^{-1}(F) > 0$ and together with $\mu(F) < \mu(E)$ we obtain
    \begin{equation} \label{eq:(3, -1, d)_x_bound}
        - \ch_0(F) < \ch_1(F) < -\tfrac{1}{3} \ch_0(F).
    \end{equation}
    In particular, this is impossible if $\ch_0(F) = 1$.
    \begin{enumerate}
        \item If $\ch_0(F) = 2$, then \eqref{eq:(3, -1, d)_x_bound} implies $\ch_1(F) = -1$. Let $\ch(F) = (2, -1, y, z)$. Then
        \[
        0 \geq s(E, F) - s_Q(E) \geq (2d - 3y) - \frac{36d^2 + 8d + 51}{48d - 8}
        \]
        which implies
        \[
        \frac{20d^2 - 8d - 17}{48d - 8} \leq y \leq -\frac{1}{2}.
        \]
        We can use the bound from Theorem \ref{thm:rank_two} on $F$ and the bound from Theorem \ref{thm:rank_one} on $G$ to obtain
        \[
        e \leq \frac{1}{2}d^2 - dy + y^2 - \frac{1}{2}d - \frac{1}{2}y + \frac{5}{24}.
        \]
        This is a parabola in $y$ with minimum at $y = \tfrac{1}{2}d + \tfrac{1}{4}$ which is smaller than our range $y$. Therefore, the maximum occurs at $y = -\tfrac{1}{2}$, i.e., 
        \[
        e = \frac{1}{2} d^2 + \frac{17}{24}.
        \]
        In this case, we have $F \in M(2, -1, -\tfrac{1}{2}, \tfrac{5}{6})$ and the quotient has Chern character $\ch(G) = (1, 0, d + \tfrac{1}{2}, E(1, 0, d + \tfrac{1}{2}))$. By Theorem \ref{thm:rank_one} we get $G = \II_C$ for a plane curve of degree $-d - \tfrac{1}{2}$.
        
        \item If $\ch_0(F) = 3$, then \eqref{eq:(3, -1, d)_x_bound} implies $\ch_1(F) = -2$. Let $\ch(F) = (3, -2, y, z)$. Then
        \[
        0 \geq s(E, F) - s_Q(E) \geq (d - y) - \frac{36d^2 + 8d + 51}{48d - 8}
        \]
        which implies
        \[
        \frac{12d^2 - 16d - 51}{48d - 8} \leq y \leq 0.
        \]
        By induction we have $z \leq \tfrac{1}{2} y^2 - \tfrac{3}{2} y + \tfrac{2}{3}$. We can use the bound from Theorem \ref{thm:rank_zero} on $G$ to obtain
        \[
        e \leq \frac{1}{2}d^2 - dy + y^2 - \frac{3}{2}y + \frac{17}{24}.
        \]
        This is a parabola in $y$ with minimum at $y = \tfrac{1}{2}d + \tfrac{3}{4}$ which is smaller than our range $y$. Therefore, the maximum occurs at $y = 0$, i.e., 
        \[
        e = \frac{1}{2} d^2 + \frac{17}{24}.
        \]
        In this case, $F = T(-2)$ and the quotient is $G = \OO_V(d + \tfrac{1}{2})$.
    \end{enumerate}
    
    \item If $c = 0$ and $d \leq -4$, then $e \geq \tfrac{1}{2}d^2 + \tfrac{1}{2}d$. A straightforward computation implies
    \begin{align*}
    \rho^2_Q(E) - \frac{\Delta(E)}{16} &= \frac{8d^3 + 27e^2}{12d^2} + \frac{3}{8}d \\
    &\geq \frac{9}{16}d^2 + \frac{13}{6}d + \frac{9}{16} > 0.
    \end{align*}
    This implies two things. Firstly, $\rho^2_Q(E) > 0$ means that $E$ has to be destabilized along a semicircular wall. Assume this wall is induced by a short exact sequence
    \[
    0 \to F \to E \to G \to 0.
    \]
    By Proposition \ref{prop:properties_dest_sequences} on destabilizing sequences we can assume that both $\ch_0(F) \geq 1$ and $\mu(F) < \mu(E)$. Secondly, Proposition \ref{prop:properties_dest_sequences} also implies $\ch_0(F) \in [1, 3]$. We can compute
    \[
    Q_{0, -1}(E) = 4d^2 - 6d - 18e \leq -5d(d + 3) < 0.
    \]
    The definition of $\Coh^{\beta}(\P^3)$ implies $\ch^{-1}(F) > 0$ and together with $\mu(F) < \mu(E)$ we obtain
    \begin{equation} \label{eq:(3, 0, d)_x_bound}
        -\ch_0(F) < \ch_1(F) < 0.
    \end{equation}
    This implies that $\ch_0(F) \neq 1$.
    \begin{enumerate}
        \item If $\ch_0(F) = 2$, then \eqref{eq:(3, -1, d)_x_bound} implies $\ch_1(F) = -1$. Let $\ch(F) = (2, -1, y, z)$. Then
        \[
        0 \geq s(E, F) - s_Q(E) \geq \left(\frac{2}{3} d - y\right) - \left(\frac{3}{4}d + \frac{3}{4}\right)
        \]
        which implies
        \[
        -\frac{1}{12}d - \frac{3}{4} \leq y \leq -\frac{1}{2}.
        \]
        This is a contradiction to $d \leq -4$.
        
        \item If $\ch_0(F) = 2$, then \eqref{eq:(3, -1, d)_x_bound} implies $\ch_1(F) \in \{ -1, -2 \}$. Let $\ch(F) = (3, x, y, z)$. If $x = -2$, then 
        \[
        0 \geq s(E, F) - s_Q(E) \geq \left(\frac{1}{2} d - \frac{1}{2} y\right) - \left(\frac{3}{4}d + \frac{3}{4}\right)
        \]
        which implies
        \[
        -\frac{1}{2}d - \frac{3}{2} \leq y \leq 0.
        \]
        This is again a contradiction to $d \leq -4$. Therefore, we must have $x = -1$. Assume first that $y \leq -\tfrac{3}{2}$. Then 
        \[
        0 \geq s(E, F) - s_Q(E) \geq \left(d - y\right) - \left(\frac{3}{4}d + \frac{3}{4}\right)
        \]
        which implies
        \[
        \frac{1}{4}d - \frac{3}{4} \leq y \leq -\frac{1}{2}.
        \]
        We know by induction that $z \leq \tfrac{1}{2} y^2 + \tfrac{17}{24}$. We can use the bound from Theorem \ref{thm:rank_zero} on $G$ to obtain
        \[
        e \leq \frac{1}{2}d^2 - dy + y^2 + \frac{3}{4}.
        \]  
        This is a parabola in $y$ with minimum at $y = \tfrac{1}{2}d$ which is smaller than our range $y$. Therefore, the maximum occurs at $y = -\tfrac{3}{2}$, i.e., 
        \[
        e \leq \frac{1}{2} d^2 + \frac{3}{2} d + 3 < \frac{1}{2}d^2 + \frac{1}{2}d.
        \]
        We are left with the case $y = -\tfrac{1}{2}$. Here we know $z \leq -\tfrac{1}{6}$ and another application of Theorem \ref{thm:rank_zero} on $G$ leads to
        \[
        e = \frac{1}{2}d^2 + \frac{1}{2}d.
        \]
        In this case, $F = \Omega(1)$ and the quotient is $G = \OO_V(d + 1)$. \qedhere
    \end{enumerate}
\end{enumerate}
\end{proof}

\begin{cor}
\label{cor:conjecture_rank_three}
Conjecture \ref{conj:irreducible_smooth} holds for rank three sheaves.
\end{cor}

\begin{proof}
The proof will proceed on a case by case basis. Let $E$ be a Gieseker-semistable sheaf with $\ch(E) = (3, c, d, E(3, c, d))$.
\begin{enumerate}
    \item Assume that $c = -2$.
    \begin{enumerate}
        \item If $d = 0$, then by Proposition \ref{prop:rank_three_special_cases} we have $E = T(-2)$ and $M(3, -2, 0, \tfrac{2}{3})$ is a single point.
        
        \item If $d \leq -1$, then by Lemma \ref{lem:rank_three_classification_general} any such $E$ fits into a short exact sequence
        \[
        0 \to \OO(-1)^{\oplus 3} \to E \to \OO_V(d-1) \to 0.
        \]
        Therefore, the moduli space is a Grassmann bundle, i.e., the space is irreducible and smooth. The base of this bundle is the space of planes $V \subset \P^3$ and the fibers are the Grassmannians of three-dimensional subspaces of $\Ext^1(\OO_V(d-1), \OO(-1))$.
    \end{enumerate}
    
    \item Assume that $c = -1$.
    \begin{enumerate}
        \item If $d = -\tfrac{1}{2}$, then Proposition \ref{prop:rank_three_special_cases} says $E = \Omega(1)$ and the moduli space $M(3, -1, -\tfrac{1}{2}, -\tfrac{1}{6})$ is a single point.
        \item If $d = -\tfrac{3}{2}$, then by Proposition \ref{prop:rank_three_special_cases} $E$ fits into a short exact sequence
        \[
        0 \to \OO(-2)^{\oplus 2} \to \OO(-1)^{\oplus 5} \to E \to 0.
        \]
        As shown in the proof of \cite[Theorem 7.1]{Sch20:stability_threefolds} the space of these $E$ is a moduli space of representations over a generalized Kronecker quiver. In particular, it is smooth and irreducible.
        \item If $d = -\tfrac{5}{2}$, then we use Proposition \ref{prop:classification_(3,-1,-5/2)}. Here we have to deal with multiple walls. The smallest wall is given by extensions of the form
        \[
        0 \to \OO(-1)^{\oplus 4} \to E \to \OO(-3)[1] \to 0.
        \]
        This is again an instance of the proof of \cite[Theorem 7.1]{Sch20:stability_threefolds} and we get a smooth and irreducible moduli space over a generalized Kronecker quiver. More precisely, we are dealing with $\Gr(4, 10)$ which has dimension $24$.
        
        The next wall deals with extensions between $T(-2)$ and $\OO_V(-2)$ for planes $V \subset \P^3$. Some homological algebra shows $\ext^1(T(-2), T(-2)) = 0$, $\ext^1(\OO_V(-2), \OO_V(-2)) = 3$, $\ext^1(T(-2), \OO_V(-2)) = 1$, $\ext^1(\OO_V(-2), T(-2)) = 21$. This has two consequence. On the one hand, this means that the destabilized locus in our first moduli space $\Gr(4, 10)$ is isomorphic to $\P^3$, and the new locus is a $\P^{20}$-bundle over this $\P^3$. In particular, both loci are irreducible. If we can show that the new points are smooth, then we know that the second space is smooth and irreducible as well. Indeed, for a tilt-stable $E$ that fits into
        \[
        0 \to T(-2) \to E \to \OO_V(-2) \to 0
        \]
        we can do some homological algebra to get
        \[
        \ext^1(E, E) \leq 3 + 1 + 21 - 1 = 24.
        \]
        
        The largest wall deals with extensions between $F \in M(2, -1, -\tfrac{1}{2}, \tfrac{5}{6})$ and $\II_C$ for a conic $C \subset \P^3$. Recall that $F$ fits into a short exact sequence
        \[
        0 \to \OO(-2) \to \OO(-1)^{\oplus 3} \to F \to 0.
        \]
        Slope stability implies $\hom(\II_C, F) = 0$ and $\ext^3(\II_C, F) = \hom(F, \II_C(-4)) = 0$. Thus, $\ext^1(\II_C, F) \leq -\chi(\II_C, F) = 12$. By tilt stability along the wall, we also get $\hom(F, \II_C) = 0$. Moreover, slope stability implies $\ext^3(F, \II_C) = \hom(\II_C, F(-4)) = 0$ and thus, $\ext^1(F, \II_C) \leq - \chi(\II_C, F) = 2$. The Hilbert scheme of conics in $\P^3$ is smooth of dimension $8$, i.e., $\ext^1(\II_C, \II_C) = 8$. The moduli space $M(2, -1, -\tfrac{1}{2}, \tfrac{5}{6})$ is smooth of dimension $3$, i.e., $\ext^1(F, F) = 3$.
        
        Again this has two consequences. On one hand, the destabilized loci from both sides are irreducible, since they are projective bundles over the product of $M(2, -1, -\tfrac{1}{2}, \tfrac{5}{6})$ and the space of conics. On the other hand, we do some homological algebra to get
        \[
        \ext^1(E, E) \leq 3 + 8 + 12 + 2 - 1 = 24.
        \]
        
        \item Next, we assume that $d \leq -\tfrac{7}{2}$. Again there is more than one wall to analyze. The smallest wall has no stable objects below and destabilizes $E$ that fit into short exact sequences
        \[
        0 \to T(-2) \to E \to \OO_V\left( d + \frac{1}{2} \right) \to 0
        \]
        where $V \subset \P^3$ is a plane. Some homological algebra yields 
        \begin{align*}
        \ext^1(T(-2), T(-2)) &= 0, \\
        \ext^1(\OO_V(d + \tfrac{1}{2}), \OO_V(d + \tfrac{1}{2})) &= 3, \\
        \ext^1(T(-2), \OO_V(d + \tfrac{1}{2})) &= 0, \\
        \ext^1(\OO_V(d + \tfrac{1}{2}), T(-2)) &= \tfrac{3}{2} d^2 - 4d + \tfrac{13}{8}.
        \end{align*}        
        As before, we obtain
        \[
        \dim M(3, -1, d, E(3, -1, d)) = 3 + \tfrac{3}{2} d^2 - 4d + \tfrac{13}{8} - 1 = \tfrac{3}{2} d^2 - 4d + \tfrac{29}{8}.
        \]
        Moreover, the moduli space of such extensions is a $\P^{\tfrac{3}{2} d^2 - 4d + \tfrac{5}{8}}$-bundle over the space of planes in $\P^3$, i.e., over $\P^3$ itself. This first moduli is smooth and irreducible.
        
        The second wall deals with extension between $F \in M(2, -1, -\tfrac{1}{2}, \tfrac{5}{6})$ and $\II_C$ for a plane curve $C \subset \P^3$ of degree $-d - \tfrac{1}{2}$. We have short exact sequences
        \[
        0 \to \OO(-2) \to \OO(-1)^{\oplus 3} \to F \to 0
        \]
        and
        \[
        0 \to \OO(-1) \to \II_C \to \OO_V\left(d + \tfrac{1}{2}\right) \to 0
        \]
        where $V$ is the plane that $C$ is contained in. Some homological algebra leads to $\ext^1(F, F) = 3$, $\ext^1(\II_C, \II_C) = \tfrac{1}{2} d^2 - d + \tfrac{19}{8} > 0$, $\ext^1(F, \II_C) = 1$, and finally $\ext^1(\II_C, F) = d^2 - 3d - \tfrac{7}{4} > 0$. Again this has two consequences. On the one hand, the destabilized loci from both sides are irreducible, since they are projective bundles over the product of $M(2, -1, -\tfrac{1}{2}, \tfrac{5}{6})$ and the space of plane curves of degree $-d - \tfrac{1}{2}$. On the other hand, some homological algebra implies
        \[
        \ext^1(E, E) \leq 3 + (\tfrac{1}{2} d^2 - d + \tfrac{19}{8}) + 1 + (d^2 - 3d - \tfrac{7}{4}) - 1 = \tfrac{3}{2} d^2 - 4d + \tfrac{29}{8}.
        \]
    \end{enumerate}
    
    \item Assume that $c = 0$.
    \begin{enumerate}
        \item If $d = 0$, then $E = \OO^{\oplus 3}$ and the space is a single point.
        
        \item If $d = -1$, then $E$ fits into a short exact sequence
        \[
        0 \to E \to \OO^{\oplus 3} \to \OO_L(2) \to 0
        \]
        for a line $L \subset \P^3$. Since $H^0(\OO_L(2)) = \C^3$, there is a unique $E$ for every $L$ and the moduli space is $\Gr(2, 4)$.
        
        \item If $d = -2$, then $E$ fits into a short exact sequence
        \[
        0 \to \Omega(1) \to E \to \OO_V(-1) \to 0
        \]
        for a plane $V \subset \P^3$. Some homological algebra shows $\Ext^1(\OO_V(-1), \Omega(1)) = \C^{14}$. This means the moduli space is a $\P^{13}$-bundle over $\P^3$.
        
        \item If $d = -3$, then we have to deal with two walls. The first wall deals with extensions
        \[
        0 \to \OO(-1)^{\oplus 6} \to E \to \OO(-2)^{\oplus 3}[1] \to 0.
        \]
        Above this wall, we get the moduli space of quiver representations with dimension vector $(3, 6)$ on the generalized Kronecker with four arrows. Clearly, this is irreducible, and smooth along the stable points. It has dimension $28$.
        
        The second wall is about extensions between $\Omega(1)$ and $\OO_V(-2)$ for planes $V \subset \P^3$. Some homological algebra leads to $\ext^1(\Omega(1), \Omega(1)) = 0$, $\ext^1(\OO_V(-2), \OO_V(-2)) = 3$, $\ext^1(\OO_V(-2), \Omega(1)) = 25$, and $\ext^1(\Omega(1), \OO_V(-2)) = 1$. Again this has two consequences. On the one hand, the destabilized loci from both sides are irreducible. The locus below the wall is isomorphic to $\P^3$, and the locus above the wall is a $\P^{24}$-bundle over this $\P^3$. Some homological algebra implies for any such extensions $E$ that
        \[
        \ext^1(E, E) \leq 3 + 25 + 1 - 1 = 28.
        \]
        \item Let $d \leq -4$. Then $E$ fits into a short exact sequence
        \[
        0 \to \Omega(1) \to E \to \OO_V(d+1) \to 0.
        \]
        We do some homological algebra again to get $\ext^1(\OO_V(d+1), \Omega(1)) = \tfrac{3}{2} d^2 - \tfrac{7}{2} d + 1$. Therefore, the moduli space is a projective bundle with base $\P^3$ and fibers of dimension $\tfrac{3}{2} d^2 - \tfrac{7}{2} d$. \qedhere
    \end{enumerate}
\end{enumerate}
\end{proof}

\begin{rmk}
\label{rmk:moduli_spaces_rank_three}
In the proof of Corollary \ref{cor:conjecture_rank_three} we actually obtained more geometric information about the moduli spaces. Let us summarize these results:
\begin{enumerate}
    \item $M(3, -2, 0, \tfrac{2}{2})$ is a point.
    \item For $d \leq -1$, the moduli space $M(3, -2, -d, \tfrac{1}{2} d^2 - \tfrac{3}{2} d + \tfrac{2}{3})$ is a $\Gr(3, \tfrac{1}{2}d^2 - \tfrac{5}{2}d + 3)$-bundle over $\P^3$.
    \item $M(3, -1, -\tfrac{1}{2}, -\tfrac{1}{6})$ is a point.
    \item $M(3, -1, -\tfrac{3}{2}, \tfrac{11}{6})$ is the moduli space of quiver representations of the generalized Kronecker quiver with four arrows and dimension vector $(2, 5)$.
    \item $M(3, -1, -\tfrac{5}{2}, \tfrac{23}{6})$ is birational to the Grassmannian $\Gr(4, 10)$. The first wall in tilt stability corresponds to a blow-up of $\Gr(4, 10)$ in a locus isomorphic to $\P^3$, while the second wall corresponds to a flip.
    \item For $d \leq -\tfrac{7}{2}$, the moduli space $M(3, -1, d, \tfrac{1}{2} d^2 + \tfrac{17}{24})$ is the blow-up of a $\P^{\tfrac{3}{2}d^2 - 4d + \tfrac{5}{8}}$-bundle over $\P^3$ in a sublocus isomorphic to $M(2, -1, -\tfrac{1}{2}, \tfrac{5}{6}) \times M(1, 0, -2, 3)$.
    \item $M(3, 0, 0, 0)$ is a point.
    \item $M(3, 0, -1, -1)$ is $\Gr(2, 4)$.
    \item $M(3, 0, -2, 1)$ is a $\P^{13}$-bundle over $\P^3$.
    \item $M(3, 0, -3, 3)$ is the blow-up of the moduli space of quiver representations of the generalized Kronecker quiver with four arrows and dimension vector $(3, 6)$ in a sublocus isomorphic to $\P^3$.
    \item For $d \leq -4$, the moduli space $M(3, 0, d, \tfrac{1}{2}d^2 + \tfrac{1}{2}d)$ is a $\P^{\tfrac{3}{2}d^2 - \tfrac{7}{2}d}$-bundle over $\P^3$.
\end{enumerate}
\end{rmk}


\section{Rank four and the third Chern character}
\label{sec:rank_four}

\begin{thm}
\label{thm:rank_four}
We have $E(4, -3, d) = \tfrac{1}{2}d^2 - 2d + \tfrac{11}{8}$ for $d \leq -\tfrac{1}{2}$, $E(4, -2, d) = \tfrac{1}{2}d^2 - \tfrac{1}{2}d + \tfrac{2}{3}$ for $d \leq -1$, $E(4, -1, d) = \tfrac{1}{2}d^2 - \tfrac{7}{24}d$ for $d \leq -\tfrac{3}{2}$, $E(4, 0, 0) = E(4, 0, -2) = 0$, $E(4, 0, -1) = -2$, and $E(4, 0, d) = \tfrac{1}{2}d^2 + \tfrac{3}{2}d + 2$ for $d \leq -3$. The same bounds hold for tilt-semistable objects.
\end{thm}

The bounds on $\ch_3$ in the above theorem do not hold for sheaves $E$ that are slope-semistable, but not $2$-Gieseker-semistable. For example, let $C$ be a plane degree $-d$ curve. Then $\OO^{\oplus 3} \oplus \II_C$ is slope-semistable, but not $2$-Gieseker-semistable with Chern character $(4, 0, d, \tfrac{1}{2}d^2 - \tfrac{1}{2}d)$.

\begin{lem}
\label{lem:rank_four_classification_general}
Let $E$ be $\nu_{\alpha, \beta}$-semistable for some $(\alpha, \beta) \in \R_{\geq 0} \times \R$ with $\ch(E) = (4, c, d, e)$.
\begin{enumerate}
    \item Let $c = -3$ and $d \leq -\tfrac{1}{2}$. Then $e \leq \tfrac{1}{2} d^2 - 2d + \tfrac{11}{8}$. In case of equality $E$ is destabilized by a short exact sequence
    \[
    0 \to \OO(-1)^{\oplus 4} \to E \to \OO_V\left(d - \tfrac{3}{2}\right) \to 0
    \]
    where $V \subset \P^3$ is a plane.
     
    \item Let $c = -2$ and $d \leq -3$. Then $e \leq \tfrac{1}{2} d^2 - \tfrac{1}{2}d + \tfrac{2}{3}$. In case of equality $E$ is destabilized by a short exact sequence
    \[
    0 \to T(-2) \to E \to \II_C \to 0
    \]
    where $C \subset \P^3$ is a plane curve of degree $-d$.
     
    \item Let $c = -1$ and $d \leq -\tfrac{7}{2}$. Then $e \leq \tfrac{1}{2} d^2 - \tfrac{7}{24}$. In case of equality $E$ is destabilized by a short exact sequence
    \[
    0 \to \Omega(1) \to E \to \II_C \to 0
    \]
    where $C$ is a plane curve of degree $-d - \tfrac{1}{2}$.

    \item Let $c = 0$ and $d \leq -5$. Then $e \leq \tfrac{1}{2}d^2 + \tfrac{3}{2}d + 2$. In case of equality $E$ is destabilized by a short exact sequence
    \[
    0 \to F \to E \to \OO_V(d + 2) \to 0
    \]
    where $F \in M(4, -1, -\tfrac{3}{2}, \tfrac{5}{6})$ and $V \subset \P^3$ is a plane.
\end{enumerate}
\end{lem}

\begin{proof}
The proof is by induction on $\Delta(E)$. The start of the induction are not the cases in this statement, but the special cases established in Section \ref{sec:special_cases}. Assume that $e$ is larger than or equal to the claimed bound. We will deal with the situation case by case.
\begin{enumerate}
    \item Assume that $c = -3$ and $d \leq -\tfrac{1}{2}$. Then $e \geq \tfrac{1}{2} d^2 - 2d + \tfrac{11}{8}$. A straightforward computation implies
    \begin{align*}
    \rho^2_Q(E) - \frac{\Delta(E)}{20} &= \frac{32d^3 - 27d^2 + 216de + 144e^2 - 162e}{(8d - 9)^2} + \frac{2d}{5} - \frac{9}{20} \\
    &\geq \frac{9(20d^2 - 48d + 29)(2d - 1)^2}{20(8d - 9)^2} > 0.
    \end{align*}
    This implies two things. Firstly, $\rho^2_Q(E) > 0$ means that $E$ has to be destabilized along a semicircular wall. Assume this wall is induced by a short exact sequence
    \[
    0 \to F \to E \to G \to 0.
    \]
    By Proposition \ref{prop:properties_dest_sequences} on destabilizing sequences we can assume that both $\ch_0(F) \geq 1$ and $\mu(F) < \mu(E)$. Secondly, Proposition \ref{prop:properties_dest_sequences} also implies $\ch_0(F) \in [1, 4]$. We can compute
    \[
    Q_{0, -\tfrac{5}{4}}(E) = 4d^2 - 20d - 12e + \frac{225}{16} \leq -2d^2 + 4d - \frac{39}{16} < 0.
    \]
    By definition of $\Coh^{\beta}(\P^3)$ this implies $\ch^{-5/4}(F) > 0$ and together with $\mu(F) < \mu(E)$ we obtain
    \begin{equation}
    \label{eq:(4, -3, d)_x_bound}
    -\frac{5}{4} \ch_0(F) < \ch_1(F) < -\frac{3}{4} \ch_0(F).
    \end{equation}
    \begin{enumerate}
        \item If $\ch_0(F) = 1$, then \eqref{eq:(4, -3, d)_x_bound} implies $\ch_1(F) = -1$.
        Assume that 
        \[
        \ch(F) = (1, 0, y, z) \cdot \ch(\OO(-1)).
        \]
        Then
        \[
        0 \geq s(E, F) - s_Q(E) \geq (d - 4y - 2) - \frac{3(4d^2 - 14d + 11)}{2(8d - 9)}
        \]
        which implies
        \[
        \frac{4d^2 - 8d + 3}{8(8d - 9)} \leq y \leq 0.
        \]
        We can use the bound from Theorem \ref{thm:rank_one} on $F$ and the bound from Lemma \ref{lem:rank_three_classification_general} on $G$ to obtain
        \[
        e \leq \frac{1}{2}d^2 - dy + y^2 - 2d + \frac{1}{2}y + \frac{11}{8}.
        \]
        This is a parabola in $y$ with minimum at $y = \tfrac{1}{2}d - \tfrac{1}{4}$ which is smaller than our range of allowed $y$. Therefore, the maximum occurs at $y = 0$, i.e., 
        \[
        e = \frac{1}{2} d^2 - 2d + \frac{11}{8}.
        \]
        In this case, $F = \OO(-1)$, $\ch(G) = (3, -2, d - \tfrac{1}{2}, E(3, -2, d - \tfrac{1}{2}))$, and by Theorem \ref{thm:rank_three} this object $G$ has further maps from $\OO(-1)$ that factor through $E$. This means $G$ is strictly-semistable and this case will be subsumed by the following cases.
        
        \item If $\ch_0(F) = 2$, then \eqref{eq:(4, -3, d)_x_bound} implies $\ch_1(F) = -2$.
        Assume that 
        \[
        \ch(F) = (2, 0, y, z) \cdot \ch(\OO(-1)).
        \]
        Then
        \[
        0 \geq s(E, F) - s_Q(E) \geq (d - 2y - 2) - \frac{3(4d^2 - 14d + 11)}{2(8d - 9)}
        \]
        which implies
        \[
        \frac{4d^2 - 8d + 3}{4(8d - 9)} \leq y \leq 0.
        \]
        If $y \leq -1$, then we use the bound from Theorem \ref{thm:rank_two} on $F$ and $G$ to obtain
        \[
        e \leq \frac{1}{2}d^2 - dy + y^2 - 2d + \frac{3}{2}y + \frac{19}{8}.
        \]
        As in the previous case, this parabola in $y$ is increasing in our range of $y$ and thus, the maximum occurs at $y = -1$ where we get
        \[
        e \leq \frac{1}{2} d^2 - d + \frac{15}{8} \leq \frac{1}{2} d^2 - 2d + \frac{11}{8}.
        \]
        If $y = 0$, then we also use Theorem \ref{thm:rank_two} on both $F$ and $G$ to obtain
        \[
        e = \frac{1}{2} d^2 - 2d + \frac{11}{8}.
        \]
        In this case, $F = \OO(-1)^{\oplus 2}$, $\ch(G) = (2, -1, d - 1, E(2, -1, d - 1))$, and by Theorem \ref{thm:rank_two} this object $G$ has further maps from $\OO(-1)$ that factor through $E$. This means $G$ is strictly-semistable and this case will be subsumed by the remaining cases.
        
        \item If $\ch_0(F) = 3$, then \eqref{eq:(4, -3, d)_x_bound} implies $\ch_1(F) = -3$. Assume that
        \[
        \ch(F) = (3, 0, y, z) \cdot \ch(\OO(-1)).
        \]
        Then
        \[
        0 \geq s(E, F) - s_Q(E) \geq \left(d - \frac{4}{3}y - 2\right) - \frac{3(4d^2 - 14d + 11)}{2(8d - 9)}
        \]
        which implies
        \[
        \frac{3(4d^2 - 8d + 3)}{8(8d - 9)} \leq y \leq 0.
        \]
        We can use the bound from Lemma \ref{lem:rank_three_classification_general} on $F$ and the bound from Theorem \ref{thm:rank_one} on $G$ to obtain
        \[
        e \leq \frac{1}{2}d^2 - dy + y^2 - 2d + \frac{3}{2}y + \frac{11}{8}.
        \]
        As in the previous case, this parabola in $y$ is increasing in our range of $y$ and thus, the maximum occurs at $y = 0$ where we get
        \[
        e = \frac{1}{2} d^2 - 2d + \frac{11}{8}.
        \]
        In this case, $F = \OO(-1)^{\oplus 3}$ and the quotient satisfies $\ch(G) = (1, 0, d - \tfrac{3}{2}, E(1, 0, d - \tfrac{3}{2}))$ and by Theorem \ref{thm:rank_two} this object $G$ has a further map from $\OO(-1)$ that factors through $E$. This means $G$ is strictly-semistable and this case will be subsumed by the remaining case.
        
        \item If $\ch_0(F) = 4$, then \eqref{eq:(4, -3, d)_x_bound} implies $\ch_1(F) = -4$.
        Assume that 
        \[
        \ch(F) = (4, 0, y, z) \cdot \ch(\OO(-1)).
        \]
        By induction, we have $z \leq \tfrac{1}{2}y^2 + \tfrac{3}{2}y + 2$. Note that this bound is not sharp for $y \geq -2$, but nonetheless true. We assume for the moment that $y \leq -1$. Then
        \[
        0 \geq s(E, F) - s_Q(E) \geq (d - y - 2) - \frac{3(4d^2 - 14d + 11)}{2(8d - 9)}
        \]
        which implies
        \[
        \frac{4d^2 - 8d + 3}{2(8d - 9)} \leq y \leq -1.
        \]
        Using the bound from Theorem \ref{thm:rank_zero} on $G$ implies
        \[
        e \leq \frac{1}{2}d^2 - dy + y^2 - 2d + \frac{5}{2}y + \frac{27}{8}.
        \]
        As in the previous cases the maximum occurs at $y = -1$ and
        \[
        e \leq \frac{1}{2} d^2 - d + \frac{15}{8} \leq \frac{1}{2} d^2 - 2d + \frac{11}{8}.
        \]
        We are left to deal with $y = 0$. By induction we have $z \leq 0$, and together with Theorem \ref{thm:rank_zero} we get
        \[
        e = \frac{1}{2} d^2 - 2d + \frac{11}{8}
        \]
        and also $z = 0$. In particular, $F = \OO(-1)^{\oplus 4}$ and we obtain the sequence
        \[
        0 \to \OO(-1)^{\oplus 4} \to E \to \OO_V\left(d - \tfrac{3}{2}\right) \to 0.
        \]
    \end{enumerate}
    
    \item Assume that $c = -2$ and $d \leq -3$. Then $e \geq \tfrac{1}{2} d^2 - \tfrac{1}{2}d + \tfrac{2}{3}$. A straightforward computation implies
    \begin{align*}
    \rho^2_Q(E) - \frac{\Delta(E)}{20} &= \frac{8d^3 - 3d^2 + 36de + 36e^2 - 12e}{4(2d - 1)^2} + \frac{2d}{5} - \frac{1}{5} \\
    &\geq \frac{9(5d^3 - 2d^2 + 2d + 2)(d + 2)}{20(2d - 1)^2} > 0.
    \end{align*}
    This implies two things. Firstly, $\rho^2_Q(E) > 0$ means that $E$ has to be destabilized along a semicircular wall. Assume this wall is induced by a short exact sequence
    \[
    0 \to F \to E \to G \to 0
    \]
    with $\ch(F) = (s, x, y, z)$. By Proposition \ref{prop:properties_dest_sequences} on destabilizing sequences we can assume that $s \geq 1$ and $\mu(F) < \mu(E)$. Secondly, Proposition \ref{prop:properties_dest_sequences} also implies $s \in [1, 4]$. We can compute
    \[
    Q_{0, -1}(E) = 4d^2 - 12d - 12e + 4 \leq -2(d + 2)(d + 1) < 0.
    \]
    By definition of $\Coh^{\beta}(\P^3)$ this implies $\ch^{-1}(F) > 0$ and together with $\mu(F) < \mu(E)$ we obtain
    \[
    - s < x < -\frac{1}{2} s.
    \]
    This implies $(s, x) \in \{ (3, -2), (4, -3) \}$.
    \begin{enumerate}
        \item If $(s, x) = (3, -2)$, then
        \[
        0 \geq s(E, F) - s_Q(E) \geq \left(\frac{3}{2}d - 2y\right) - \frac{3d^2 - 2d + 4}{2(2d - 1)}
        \]
        which implies
        \[
        \frac{3d^2 - d - 4}{4(2d - 1)} \leq y \leq 0.
        \]
        We can use the bound from Lemma \ref{lem:rank_three_classification_general} on $F$ and the bound from \ref{thm:rank_one} on $G$ to obtain
        \[
        e \leq \frac{1}{2}d^2 - dy + y^2 - \frac{1}{2}d - y + \frac{2}{3}.
        \]
        As in the previous case, this parabola in $y$ is increasing in our range of $y$ and thus, the maximum occurs at $y = 0$ where we get
        \[
        e = \frac{1}{2} d^2 - \frac{1}{2}d + \frac{2}{3}.
        \]
        Therefore, Lemma \ref{lem:rank_three_classification_general} and Theorem \ref{thm:rank_one} also imply $F = T(-2)$ and $G = \II_C$ for a plane curve $C$ of degree $-d$.
        
        \item If $(s, x) = (4, -3)$, then Proposition \ref{prop:rank_four_special_cases} implies that $y \leq -\tfrac{1}{2}$. We compute
        \[
        0 \geq s(E, F) - s_Q(E) \geq (d - y) - \frac{3d^2 - 2d + 4}{2(2d - 1)}
        \]
        which implies
        \[
        \frac{d^2 - 4}{2(2d - 1)} \leq y \leq -\tfrac{1}{2}.
        \]
        By induction we have $z \leq \tfrac{1}{2} y^2 - 2y + \tfrac{11}{8}$. An application of Theorem \ref{thm:rank_zero} to $G$ leads to 
        \[
        e \leq \frac{1}{2}d^2 - dy + y^2 - 2y + \frac{17}{12}.
        \] 
        As in the previous case, this parabola in $y$ is increasing in our range of $y$ and thus, the maximum occurs at $y = -\tfrac{1}{2}$ where we get
        \[
        e = \frac{1}{2} d^2 + \frac{1}{2}d + \frac{8}{3} \leq \frac{1}{2} d^2 - \frac{1}{2}d + \frac{2}{3}.
        \]
    \end{enumerate}
    
    \item Assume that $c = -1$ and $d \leq -\tfrac{7}{2}$. Then $e \geq \tfrac{1}{2} d^2 - \tfrac{7}{24}$. A straightforward computation implies
    \begin{align*}
    \rho^2_Q(E) - \frac{\Delta(E)}{20} &= \frac{32d^3 - 3d^2 + 72de + 144e^2 - 6e}{(8d - 1)^2} + \frac{2d}{5} - \frac{1}{20} \\
    &\geq \frac{9(40d^3 + 84d^2 - 106d + 31)(2d + 1)}{20(8d - 1)^2} > 0.
    \end{align*}
    This implies two things. Firstly, $\rho^2_Q(E) > 0$ means that $E$ has to be destabilized along a semicircular wall. Assume this wall is induced by a short exact sequence
    \[
    0 \to F \to E \to G \to 0
    \]
    with $\ch(F) = (s, x, y, z)$. By Proposition \ref{prop:properties_dest_sequences} on destabilizing sequences we can assume that $s \geq 1$ and $\mu(F) < \mu(E)$. Secondly, Proposition \ref{prop:properties_dest_sequences} also implies $s \in [1, 4]$. We can compute
    \[
    Q_{0, -1}(E) = 4d^2 - 10d - 18e + 1 \leq -\frac{5}{4} (2d + 5)(2d - 1) < 0.
    \]
    By definition of $\Coh^{\beta}(\P^3)$ this implies $\ch^{-1}(F) > 0$ and together with $\mu(F) < \mu(E)$ we obtain
    \[
    - s < x < -\frac{1}{4} s.
    \]
    This implies $(s, x) \in \{ (2, -1), (3, -2), (3, -1), (4, -3), (4, -2) \}$.
    \begin{enumerate}
        \item If $(s, x) = (2, -1)$, then $y \leq -\tfrac{1}{2}$ and we can compute
        \[
        0 \geq s(E, F) - s_Q(E) \geq (d - 2y) - \frac{12d^2 + 2d - 7}{2(8d - 1)}
        \]
        which implies
        \[
        \frac{4d^2 - 4d + 7}{4(8d - 1)} \leq y \leq -\tfrac{1}{2}.
        \]
        We can use the bound from Theorem \ref{thm:rank_two} on both $F$ and $G$ to obtain
        \[
        e \leq \frac{1}{2}d^2 - dy + y^2 + \frac{1}{2}d - \frac{3}{2} y + \frac{29}{24}.
        \]
        As in the previous case, this parabola in $y$ is increasing in our range of $y$ and thus, the maximum occurs at $y = -\tfrac{1}{2}$ where we get
        \[
        e \leq \frac{1}{2} d^2 + d + \frac{53}{24} \leq \frac{1}{2} d^2 - \frac{7}{24}.
        \]
        
        \item If $(s, x) = (3, -2)$, then $y \leq 0$ and we can compute
        \[
        0 \geq s(E, F) - s_Q(E) \geq \left(\frac{3}{5}d - \frac{4}{5}y\right) - \frac{12d^2 + 2d - 7}{2(8d - 1)}
        \]
        which implies
        \[
        -\frac{12d^2 + 16d - 35}{8(8d - 1)} \leq y \leq -\frac{1}{2}.
        \]
        This is a contradiction to $d \leq -\tfrac{7}{2}$.
        
        \item If $(s, x) = (3, -1)$, then $y \leq -\tfrac{1}{2}$ and we can compute
        \[
        0 \geq s(E, F) - s_Q(E) \geq (3d - 4y) - \frac{12d^2 + 2d - 7}{2(8d - 1)}
        \]
        which implies
        \[
        \frac{36d^2 - 8d + 7}{8(8d - 1)} \leq y \leq -\tfrac{1}{2}.
        \]
        Assume for the moment that $y \leq -\tfrac{3}{2}$. Then we can use the bound from Lemma \ref{lem:rank_three_classification_general} on $F$ and the bound from \ref{thm:rank_one} on $G$ to obtain
        \[
        e \leq \frac{1}{2}d^2 - dy + y^2 - \frac{1}{2}d + \frac{1}{2} y + \frac{17}{24}.
        \]
        As in the previous case, this parabola in $y$ is increasing in our range of $y$ and thus, the maximum occurs at $y = -\tfrac{3}{2}$ where we get
        \[
        e \leq \frac{1}{2} d^2 + d + \frac{53}{24} < \frac{1}{2} d^2 - \frac{7}{24}.
        \]
        The only possibility left is $y = -\tfrac{1}{2}$. In that case, we can use the bound from Lemma \ref{lem:rank_three_classification_general} on $F$ and the bound from \ref{thm:rank_one} on $G$ again to get $z = -\tfrac{1}{6}$ and
        \[
        e = \frac{1}{2} d^2 - \frac{7}{24}.
        \]
        Moreover, in that case $F = \Omega(1)$ and $G = \II_C$ for a plane curve $C \subset \P^3$ of degree $-d - \tfrac{1}{2}$.
        
        \item If $(s, x) = (4, -3)$, then $y \leq -\tfrac{1}{2}$ and we can compute
        \[
        0 \geq s(E, F) - s_Q(E) \geq \left(\frac{1}{2}d - \frac{1}{2}y\right) - \frac{12d^2 + 2d - 7}{2(8d - 1)}
        \]
        which implies
        \[
        -\frac{4d^2 + 3d - 7}{8d - 1} \leq y \leq -\frac{1}{2}.
        \]
        This is a contradiction to $d \leq -\tfrac{7}{2}$.
        
        \item If $(s, x) = (4, -2)$, then $y \leq -1$ and by induction $z \leq \tfrac{1}{2}y^2 - \tfrac{1}{2}y + \tfrac{2}{3}$. We can compute
        \[
        0 \geq s(E, F) - s_Q(E) \geq (d - y) - \frac{12d^2 + 2d - 7}{2(8d - 1)}
        \]
        which implies
        \[
        \frac{4d^2 - 4d + 7}{2(8d - 1)} \leq y \leq -1.
        \]
        Using the bound from Theorem \ref{thm:rank_zero} on $G$ implies
        \[
        e \leq \frac{1}{2}d^2 - dy + y^2 - \frac{1}{2}y + \frac{17}{24}.
        \]
        As previously the maximum occurs at $y = -1$ and
        \[
        e \leq \frac{1}{2} d^2 + d + \frac{53}{24} \leq \frac{1}{2} d^2 - \frac{7}{24}.
        \]
    \end{enumerate}
    
    \item Assume that $c = 0$ and $d \leq -5$. Then $e \geq \tfrac{1}{2} d^2 + \tfrac{3}{2} d + 2$. We can compute
    \begin{align*}
        Q_{0, -\frac{3}{4}} &= 4d^2 - \frac{9}{2}d - 18e \leq -\frac{1}{2}(5d + 24)(2d + 3) < 0, \\
        Q_{0, \frac{2}{3}d} &= -\frac{32}{9}d^3 + 4d^2 + 16de \leq \frac{4}{9}d(5d + 24)(2d + 3) < 0.
    \end{align*}
    Firstly, this means that $E$ has be destabilized along a semicircular wall induced by a short exact sequence
    \[
    0 \to F \to E \to G \to 0.
    \]
    Let $\ch(F) = (s, x, y, z)$. By Proposition \ref{prop:properties_dest_sequences} on destabilizing sequences we may assume $s > 0$ and $\mu(F) < \mu(E) = 0$. Secondly, we get $\ch^{-\tfrac{3}{4}}(F) > 0$ and $\ch^{\tfrac{2}{3}d}(G) > 0$ which yields
    \[
    \mu(F) \in \left( -\frac{3}{4}, \min \left\{0, -\frac{8d}{3s} + \frac{2d}{3} \right\} \right).
    \]
    If $d \leq -6$, then this immediately implies $s \leq 4$. If $d = -5$, then the interval is empty for $s \geq 6$. If $d = -5$ and $s = 5$, then we have $\mu(F) \in (-\tfrac{3}{4}, \tfrac{2}{3})$, but this interval contains no value in $\tfrac{1}{5} \Z$. In either case, we get $s \leq 4$. We are left with the following cases:
    \[
    (s, x) \in \{ (2, -1), (3, -1), (3, -2), (4, -1), (4, -2) \}.
    \]
    \begin{enumerate}
        \item If $(s, x) = (2, -1)$, then $y \leq -\tfrac{1}{2}$ and we can compute
        \[
        0 \geq s(E, F) - s_Q(E) \geq \left(\frac{1}{2}d - y\right) - \frac{3d^2 + 9d + 12}{4d}
        \]
        which implies
        \[
        -\frac{d^2 + 9d + 12}{4d} \leq y \leq -\frac{1}{2}.
        \]
        This is a contradiction to $d \leq -5$.
        
        \item If $(s, x) = (3, -1)$, then $y \leq -\tfrac{1}{2}$ and by Lemma \ref{lem:rank_three_classification_general} $z \leq \tfrac{1}{2} y^2 + \tfrac{17}{24}$. This bound on $z$ is only sharp for $y \leq -\tfrac{3}{2}$. We can compute
        \[
        0 \geq s(E, F) - s_Q(E) \geq \left(\frac{3}{4}d - y\right) - \frac{3d^2 + 9d + 12}{4d}
        \]
        which implies
        \[
        -\frac{9d + 12}{4d} \leq y \leq -\frac{1}{2}.
        \]
        We can compute $\ch(G \otimes \OO(-1)) = (1, 0, d - y - \tfrac{1}{2}, -d + e + y - z + \tfrac{1}{3})$. We can apply Proposition \ref{prop:rank_one_between_-1_-2} with $s = \tfrac{3}{4}d - y - 1$ to obtain
        \[
        e \leq \frac{5}{16}d^2 - \frac{3}{4}dy + \frac{1}{2}y^2 - \frac{1}{8}d + z + \frac{1}{24} \leq \frac{5}{16}d^2 - \frac{3}{4}dy + y^2 - \frac{1}{8}d + \frac{3}{4}.
        \]
        This is a parabola in $y$ with minimum before our range for $y$, i.e., it is increasing in $y$. For $y = -\tfrac{3}{2}$, we get
        \[
        e \leq \frac{5}{16}d^2 + d + 3 \leq \frac{1}{2} d^2 + \frac{3}{2} d + 2.
        \]
        If $y = -\tfrac{1}{2}$, we can instead use the stronger bound $z \leq -\tfrac{1}{6}$ to get
        \[
        e \leq \frac{5}{16}d^2 + \frac{1}{4}d \leq \frac{1}{2} d^2 + \frac{3}{2} d + 2.
        \]
        
        \item If $(s, x) = (3, -2)$, then $y \leq 0$ and we can compute
        \[
        0 \geq s(E, F) - s_Q(E) \geq \left(\frac{3}{8}d - \frac{1}{2}y\right) - \frac{3d^2 + 9d + 12}{4d}
        \]
        which implies
        \[
        -\frac{3d^2 + 18d + 24}{4d} \leq y \leq 0.
        \]
        This is a contradiction to $d \leq -5$.
        
        \item If $(s, x) = (4, -2)$, then Proposition \ref{prop:rank_four_special_cases} says $y \leq -1$. We can compute 
        \[
        0 \geq s(E, F) - s_Q(E) \geq \left(\frac{1}{2}d - \frac{1}{2}y\right) - \frac{3d^2 + 9d + 12}{4d}
        \]
        which implies
        \[
        -\frac{d^2 + 9d + 12}{2d} \leq y \leq -1.
        \]
        This is a contradiction to $d \leq -5$.
        
        \item Lastly, if $(s, x) = (4, -1)$, then Proposition \ref{prop:rank_four_special_cases} says $y \leq -\tfrac{3}{2}$ and by induction we have $z \leq \tfrac{1}{2}y^2 - \tfrac{7}{24}$. We compute
        \[
        0 \geq s(E, F) - s_Q(E) \geq (d - y) - \frac{3d^2 + 9d + 12}{4d}
        \]
        which implies
        \[
        \frac{d^2 - 9d - 12}{4d} \leq y \leq -\frac{3}{2}.
        \]
        We apply Theorem \ref{thm:rank_zero} to $G$ together with our upper bound on $z$ to obtain
        \[
        e \leq \frac{1}{2}d^2 - dy + y^2 - \frac{1}{4}.
        \]
        This is a parabola in $y$ and its maximum in our range of $y$ occurs at $y = -\tfrac{3}{2}$ where we get
        \[
        e = \frac{1}{2} d^2 + \frac{3}{2}d + 2.
        \]
        Therefore, $F \in M(4, -1, -\tfrac{3}{2}, \tfrac{5}{6})$ and by Theorem \ref{thm:rank_zero} we have $G = \OO_V(d + 2)$ for a plane $V \subset \P^3$. \qedhere
    \end{enumerate}
\end{enumerate}
\end{proof}

\begin{cor}
\label{cor:conjecture_rank_four}
Conjecture \ref{conj:irreducible_smooth} holds for rank four sheaves.
\end{cor}

\begin{proof}
The proof will proceed on a case by case basis. Let $E$ be a Gieseker-semistable sheaf with $\ch(E) = (4, c, d, E(4, c, d))$.
\begin{enumerate}
    \item Assume $c = -3$. Then by Lemma \ref{lem:rank_four_classification_general} any such $E$ fits into a short exact sequence
    \[
    0 \to \OO(-1)^{\oplus 4} \to E \to \OO_V\left(d - \tfrac{3}{2}\right) \to 0
    \]
    where $V \subset \P^3$ is a plane. This means the moduli space is a Grassmann-bundle of four-dimensional subspaces of $\Ext^1(\OO_V(d - \tfrac{3}{2}), \OO(-1))$ over the space of planes $\P^3$, i.e., it is irreducible and smooth.
    
    \item Assume $c = -2$. 
    \begin{enumerate}
        \item If $d = -1$, then by Proposition \ref{prop:rank_four_special_cases} the sheaf $E$ fits into a short exact sequence
        \[
        0 \to \OO(-2)^{\oplus 2} \to \OO(-1)^{\oplus 6} \to E.
        \]
        This means the moduli space is the moduli space of quiver representations over the generalized Kronecker quiver with two vertices and four arrows with dimension vector $(2, 6)$. This is irreducible and smooth along the stable points, since this is true for all moduli spaces of quiver representations.
        
        \item If $d = -2$, then by Proposition \ref{prop:classification_(4,-2,-2)} we have to deal with two walls. The first wall destabilizes objects $E$ that fit into a short exact sequence
        \[
        0 \to \OO(-3) \to \OO(-1)^{\oplus 5} \to E \to 0.
        \]
        Since $H^0(\OO(2)) = \C^{10}$, the first moduli space is $\Gr(5, 10)$. Clearly, this Grassmannian is irreducible and smooth of dimension $25$. The second wall deals with extensions $E$ between $T(-2)$ and $\II_C$ for a conic $C \subset \P^3$. Some homological algebra shows $\ext^1(T(-2), T(-2)) = 0$, $\ext^1(\II_C, T(-2)) = 17$, $\ext^1(T(-2), \II_C) = 1$, $\ext^1(\II_C, \II_C) = 8$. This has two consequence. On the one hand, this means that the destabilized locus in our first moduli space $\Gr(5, 10)$ is isomorphic to the space of conics, and the new locus is a $\P^{16}$-bundle over this space of conics. In particular, both loci are irreducible. If we can show that the new points are smooth, then we know that the second space is smooth and irreducible as well. Indeed, for a tilt-stable $E$ that fits into
        \[
        0 \to T(-2) \to E \to \II_C \to 0
        \]
        we can do some homological algebra to get
        \[
        \ext^1(E, E) \leq 8 + 1 + 17 - 1 = 25.
        \]
        
        \item If $d \leq -3$, then $E$ fits into a short exact sequence
        \[
        0 \to T(-2) \to E \to \II_C \to 0
        \]
        where $C$ is a plane curve of degree $-d$. Some standard homological algebra shows $\ext^1(\II_C, T(-2)) = \tfrac{3}{2} d^2 - \tfrac{11}{2}d$. This shows that the moduli space is a $\P^{\tfrac{3}{2} d^2 - \tfrac{11}{2}d - 1}$-bundle over the moduli space of plane degree $-d$ curves. In particular, this space is smooth and irreducible.
    \end{enumerate}
    
    \item Assume $c = -1$.
    \begin{enumerate}
        \item If $d = -\tfrac{3}{2}$, then Proposition \ref{prop:rank_four_special_cases} says that $E$ fits into a short exact sequence
        \[
        0 \to \Omega(1) \to E \to \II_L \to 0
        \]
        for a line $L \subset \P^3$. Some homological algebra shows $\ext^1(\II_L, \Omega(1)) = 8$. This means that the moduli space is a $\P^7$-bundle over the Grassmannian $\Gr(2, 4)$ of lines in $\P^3$.
        
        \item If $d = -\tfrac{5}{2}$, then according to Proposition \ref{prop:classification_(4,-1,-5/2)} we have tow deal with two walls. The smallest wall deals with objects $E$ fitting into a short exact sequence of the form
        \[
        0 \to \OO(-2)^{\oplus 3} \to \OO(-1)^{\oplus 7} \to E \to 0.
        \]
        Therefore, the moduli space right above this wall is given by the moduli space of quiver representation of the generalized Kronecker quiver with four arrows and dimension vector $(3, 7)$. This space is smooth and irreducible of dimension $27$. 
        
        The second walls deals with extensions between $\Omega(1)$ and $\II_C$ for a conic $C \subset \P^3$. Some homological algebra yields $\ext^1(\Omega(1), \Omega(1)) = 0$, $\ext^1(\II_C, \II_C) = 8$, $\ext^1(\Omega(1), \II_C) = 1$, and $\ext^1(\II_C, \Omega(1)) = 19$. This has two consequence. On the one hand, this means that the destabilized locus in our first moduli space is isomorphic to the space of conics, and the new locus is a $\P^{18}$-bundle over this space of conics. In particular, both loci are irreducible. If we can show that the new points are smooth, then we know that the second space is smooth and irreducible as well. Indeed, for a tilt-stable $E$ that fits into
        \[
        0 \to \Omega(1) \to E \to \II_C \to 0
        \]
        we can do some homological algebra to get
        \[
        \ext^1(E, E) \leq 8 + 1 + 19 - 1 = 27.
        \]

        \item If $d \leq -\tfrac{7}{2}$, then according to Lemma \ref{lem:rank_four_classification_general} any such $E$ fits into a short exact sequence
        \[
        0 \to \Omega(1) \to E \to \II_C \to 0
        \]
        where $C$ is a plane curve of degree $-d - \tfrac{1}{2}$. Some homological algebra leads to $\ext^1(\II_C, \Omega(1)) = \tfrac{3}{2} d^2 - 5d - \tfrac{23}{8}$. This shows that the moduli space is a $\P^{\tfrac{3}{2} d^2 - 5d - \tfrac{31}{8}}$-bundle over the moduli space of plane curves of degree $-d-\tfrac{1}{2}$. In particular, this space is smooth and irreducible.
    \end{enumerate}
    
    \item Assume $c = 0$.
    \begin{enumerate}
        \item If $d = 0$, then $E = \OO_X^{\oplus 4}$ and the moduli space is a single point.
        
        \item If $d = -1$, then by Proposition \ref{prop:rank_four_special_cases} any such $E$ fits into a short exact sequence
        \[
        0 \to E \to \OO^{\oplus 4} \to \OO_L(3) \to 0
        \]
        for a line $L \subset \P^3$. Since $H^0(\OO_L(3)) = \C^4$ there is a unique $E$ for each line and the moduli space is simply $\Gr(2, 4)$.
        
        \item If $d = -2$, then by Proposition \ref{prop:rank_four_special_cases} any such $E$ fits into a short exact sequence
        \[
        0 \to \OO(-1)^{\oplus 2} \to \Omega(1)^{\oplus 2} \to E \to 0. 
        \]
        Since $\Hom(\OO(-1), \Omega(1)) = H^0(\Omega(2)) = \C^6$, our moduli space is given by quiver representations of the generalized Kronecker quiver with six arrows and dimensions vector $(2, 2)$. Therefore, it is irreducible and smooth along its locus of stable sheaves.
        
        \item If $d = -3$, then Proposition \ref{prop:rank_four_special_cases} says that $E$ fits into a short exact sequence
        \[
        0 \to \Omega(1) \to E \to G(1) \to 0
        \]
        where $G \in M^{\alpha, \beta - 1}(1, 0, -3, 5)$ for $(\alpha, \beta) \in W(E, \Omega(1))$. By Proposition \ref{prop:final_model_(1, 0, -3, 5)} we know that $G$ fits into a short exact sequence
        \[
        0 \to \OO(-2)^{\oplus 3} \to G \to \OO(-3)^{\oplus 2}[1] \to 0.
        \]
        Some homological algebra shows $\ext^1(G(1), \Omega(1)) = \C^{22}$, i.e., our moduli space is a $\P^{21}$-bundle over $M^{\alpha, \beta - 1}(1, 0, -3, 5)$ for $(\alpha, \beta) \in W(E, \Omega(1))$.
        
        \item If $d = -4$, then according to Proposition \ref{prop:classification_(4,0,-4)} we have to deal with two walls. The smallest wall deals with objects $E$ that fit into a short exact sequence
        \[
        0 \to \OO(-1)^{\oplus 8} \to E \to \OO(-2)^{\oplus 4}[1] \to 0.
        \]
        This means the moduli space of tilt-semistable objects right above this wall is the moduli space of quiver representations over the generalized Kronecker quiver with four arrows with dimension vector $(4, 8)$. This is irreducible and smooth along the locus of stable objects. Its dimension is $49$.
        
        If $E$ is tilt-semistable above the second wall, and destabilized at the second wall, then it fits into a short exact sequence
        \[
        0 \to \Omega(1) \to E \to G(1) \to 0
        \]
        where $G \in M^{\alpha, \beta - 1}(1, 0, -4, 8)$ for $(\alpha, \beta) \in W(E, \Omega(1))$. A straightforward computation shows that $W(E, \Omega(1)) = W(G(1), \II_L)$. By Proposition \ref{prop:final_model_(1, 0, -4, 8)} there some possibilities for $G$.
        
        Firstly, assume that $G(1)$ is extension between $\OO_V(-2)$ and $\II_L$ for a line $L \subset \P^3$ and a plane $V \subset \P^3$.  Some homological algebra computation implies $\ext^1(\Omega(1), \Omega(1)) = 0$, $\ext^1(G(1), G(1)) = 16$, $\ext^1(\Omega(1), G(1)) = 1$, $\ext^1(G(1), \Omega(1)) = 33$ and thus,
        \[
        \ext^1(E, E) \leq 0 + 16 + 1 + 33 - 1 = 49.
        \]
        This means these new points are all smooth.
        
        Secondly, $G(1)$ could be tilt-stable, i.e., fits into a short exact sequence
        \[
        0 \to \OO(-1)^{\oplus 2} \to G(1) \to \OO(-3)[1] \to 0.
        \]
        As before, we compute the cohomology $\ext^1(\Omega(1), \Omega(1)) = 0$, $\ext^1(G(1), G(1)) = 16$, $\ext^1(\Omega(1), G(1)) = 1$, $\ext^1(G(1), \Omega(1)) = 33$ and thus,
        \[
        \ext^1(E, E) \leq 0 + 16 + 1 + 33 - 1 = 49.
        \]
        
        Since the locus of new points is connected and only consists of smooth points, the moduli space above this wall has to stay irreducible.

        \item If $d \leq -5$, then according to Lemma \ref{lem:rank_four_classification_general} any such $E$ fits into a short exact sequence
        \[
        0 \to F \to E \to \OO_V(d + 2) \to 0
        \]
        where $F \in M(4, -1, -\tfrac{3}{2}, \tfrac{5}{6})$ and $V \subset \P^3$ is a plane. We have
        \[
        \ext^1(\OO_V(d+2), F) = 2d^2 - d - 2
        \]
        and thus, our moduli space is a projective bundle over $M(4, -1, -\tfrac{3}{2}, \tfrac{5}{6}) \times \P^3$. \qedhere
    \end{enumerate}
\end{enumerate}
\end{proof}

\begin{rmk}
\label{rmk:moduli_spaces_rank_four}
In the proof of Corollary \ref{cor:conjecture_rank_four} we actually obtained more geometric information about the moduli spaces. Let us summarize these results:
\begin{enumerate}
    \item For $d \leq -\tfrac{1}{2}$ the moduli space $M(4, -3, d, \tfrac{1}{2} d^2 - 2d + \tfrac{11}{8})$ is $\Gr(4, n)$-bundle over $\P^3$ for $n = \binom{-d + 7/2}{2}$.
    
    \item $M(4, -2, -1, \tfrac{5}{3})$ is the moduli space of quiver representations over the generalized Kronecker quiver with four arrows and dimension vector $(2, 6)$.
    
    \item $M(4, -2, -2, \tfrac{11}{3})$ is the blow up of $\Gr(5, 10)$ in a sublocus isomorphic the Hilbert scheme of conics in $\P^3$. 
    
    \item If $d \leq -3$, then $M(4, -2, d, \tfrac{1}{2}d^2 - \tfrac{1}{2}d + \tfrac{2}{3})$ is a $\P^{\tfrac{3}{2}d^2 - \tfrac{11}{2}d - 1}$-bundle over the Hilbert scheme of plane degree $-d$ curves in $\P^3$.
    
    \item $M(4, -1, -\tfrac{3}{2}, \tfrac{5}{6})$ is a $\P^7$-bundle over $\Gr(2, 4)$.
    
    \item $M(4, -1, -\tfrac{5}{2}, \tfrac{17}{6})$ is birational to the moduli space of quiver representations of the generalized Kronecker quiver with four arrows and dimension vector $(3, 7)$. More precisely, $M(4, -1, -\tfrac{5}{2}, \tfrac{17}{6})$ is the blow-up of the aforementioned moduli space of quiver representations in a sublocus isomorphic to the Hilbert scheme of conics in $\P^3$.
    
    \item For $d \leq -\tfrac{7}{2}$ the moduli space $M(4, -1, d, \tfrac{1}{2}d^2  - \tfrac{7}{24})$ is a $\P^{\tfrac{3}{2}d^2 - 5d - \tfrac{31}{8}}$-bundle over the Hilbert scheme of plane degree $-d - \tfrac{1}{2}$ curves in $\P^3$.
    
    \item $M(4, 0, 0, 0)$ is a point.
    
    \item $M(4, 0, -1, -2)$ is $\Gr(2, 4)$.
    
    \item $M(4, 0, -2, 0)$ is the moduli space of quiver representation over the generalized Kronecker quiver with six arrows with dimension vector $(2, 2)$.
    
    \item $M(4, 0, -3, 2)$ is a $\P^{21}$-bundle over the moduli space of quiver representations over the generalized Kronecker quiver with four arrows and dimension vector $(2, 3)$.
    
    \item $M(4, 0, -4, 4)$ is birational to the moduli space of quiver representation over the generalized Kronecker quiver with four arrows and dimension vector $(4, 8)$. The wall crossing is more complicated, and we only showed that $M(4, 0, -4, 4)$ is smooth and irreducible.
    
    \item $M(4, 0, d, \tfrac{1}{2}d^2 + \tfrac{3}{2}d + 2)$ is a $\P^{2d^2 - d -3}$-bundle over $M(4, -1, -\tfrac{3}{2}, \tfrac{5}{6}) \times \P^3$. \qedhere
\end{enumerate}
\end{rmk}

\def\cprime{$'$} \def\cprime{$'$}


\begin{thebibliography}{BMSZ17}

\bibitem[AB13]{AB13:k_trivial}
D.~Arcara and A.~Bertram.
\newblock {B}ridgeland-stable moduli spaces for {$K$}-trivial surfaces.
\newblock {\em J. Eur. Math. Soc. (JEMS)}, 15(1):1--38, 2013.
\newblock With an appendix by Max Lieblich.

\bibitem[BMS16]{BMS16:abelian_threefolds}
A.~Bayer, E.~Macr\`i, and P.~Stellari.
\newblock The space of stability conditions on abelian threefolds, and on some
  {C}alabi-{Y}au threefolds.
\newblock {\em Invent. Math.}, 206(3):869--933, 2016.

\bibitem[BMSZ17]{BMSZ17:stability_fano}
M.~Bernardara, E.~Macr\`\i, B.~Schmidt, and X.~Zhao.
\newblock Bridgeland stability conditions on {F}ano threefolds.
\newblock {\em \'Epijournal Geom. Alg\'ebrique}, 1:Art. 2, 24, 2017.

\bibitem[BMT14]{BMT14:stability_threefolds}
A.~Bayer, E.~Macr\`i, and Y.~Toda.
\newblock Bridgeland stability conditions on threefolds {I}:
  {B}ogomolov-{G}ieseker type inequalities.
\newblock {\em J. Algebraic Geom.}, 23(1):117--163, 2014.

\bibitem[Bog78]{Bog78:inequality}
F.~A. Bogomolov.
\newblock Holomorphic tensors and vector bundles on projective manifolds.
\newblock {\em Izv. Akad. Nauk SSSR Ser. Mat.}, 42(6):1227--1287, 1439, 1978.

\bibitem[Bri07]{Bri07:stability_conditions}
T.~Bridgeland.
\newblock Stability conditions on triangulated categories.
\newblock {\em Ann. of Math. (2)}, 166(2):317--345, 2007.

\bibitem[Bri08]{Bri08:stability_k3}
T.~Bridgeland.
\newblock Stability conditions on {$K3$} surfaces.
\newblock {\em Duke Math. J.}, 141(2):241--291, 2008.

\bibitem[CH16]{CH16:ample_cone_plane}
I.~Coskun and J.~Huizenga.
\newblock The ample cone of moduli spaces of sheaves on the plane.
\newblock {\em Algebr. Geom.}, 3(1):106--136, 2016.

\bibitem[DLP85]{DP85:stable_p2}
J.-M. Drezet and J.~Le~Potier.
\newblock Fibr\'es stables et fibr\'es exceptionnels sur {${\bf P}\sb 2$}.
\newblock {\em Ann. Sci. \'Ecole Norm. Sup. (4)}, 18(2):193--243, 1985.

\bibitem[GHS18]{GHS16:elliptic_quartics}
P.~Gallardo, {C. Lozano} Huerta, and B.~Schmidt.
\newblock Families of elliptic curves in $\mathbb{P}^3$ and {B}ridgeland
  stability.
\newblock {\em Michigan Math. J.}, 67(4):787--813, 2018.

\bibitem[Gie77]{Gie77:vector_bundles}
D.~Gieseker.
\newblock On the moduli of vector bundles on an algebraic surface.
\newblock {\em Ann. of Math. (2)}, 106(1):45--60, 1977.

\bibitem[GP78]{GP78:genre_courbesI}
L.~Gruson and C.~Peskine.
\newblock Genre des courbes de l'espace projectif.
\newblock In {\em Algebraic geometry ({P}roc. {S}ympos., {U}niv. {T}roms\o,
  {T}roms\o, 1977)}, volume 687 of {\em Lecture Notes in Math.}, pages 31--59.
  Springer, Berlin, 1978.

\bibitem[Hal82]{Hal82:genus_space_curves}
G.~Halphen.
\newblock Memoire sur la classification des courbes gauches algebriques.
\newblock {\em J. \'Ecole Polyt.}, 52:1--200, 1882.

\bibitem[HL10]{HL10:moduli_sheaves}
D.~Huybrechts and M.~Lehn.
\newblock {\em The geometry of moduli spaces of sheaves}.
\newblock Cambridge Mathematical Library. Cambridge University Press,
  Cambridge, second edition, 2010.

\bibitem[Ina02]{Ina02:moduli_complexes}
M.~Inaba.
\newblock Toward a definition of moduli of complexes of coherent sheaves on a
  projective scheme.
\newblock {\em J. Math. Kyoto Univ.}, 42(2):317--329, 2002.

\bibitem[JM22]{JM22:walls_bridgeland_3fold}
M.~Jardim and A.~Maciocia.
\newblock Walls and asymptotics for {B}ridgeland stability conditions on 3-folds.
\newblock {\em \'Epijournal Geom. Alg\'ebrique}, 6:Art. 22, 61 p., 2022.

\bibitem[Kin94]{Kin94:moduli_quiver_reps}
A.~D. King.
\newblock Moduli of representations of finite-dimensional algebras.
\newblock {\em Quart. J. Math. Oxford Ser. (2)}, 45(180):515--530, 1994.

\bibitem[Li19]{Li19:conjecture_fano_threefold}
C.~Li.
\newblock Stability conditions on {F}ano threefolds of {P}icard number 1.
\newblock {\em J. Eur. Math. Soc. (JEMS)}, 21(3):709--726, 2019.

\bibitem[Lie06]{Lie06:moduli_complexes}
M.~Lieblich.
\newblock Moduli of complexes on a proper morphism.
\newblock {\em J. Algebraic Geom.}, 15(1):175--206, 2006.

\bibitem[Mac14a]{Mac14:nested_wall_theorem}
A.~Maciocia.
\newblock Computing the walls associated to {B}ridgeland stability conditions
  on projective surfaces.
\newblock {\em Asian J. Math.}, 18(2):263--279, 2014.

\bibitem[Mac14b]{Mac14:conjecture_p3}
E.~Macr\`i.
\newblock A generalized {B}ogomolov-{G}ieseker inequality for the
  three-dimensional projective space.
\newblock {\em Algebra Number Theory}, 8(1):173--190, 2014.

\bibitem[Mar77]{Mar77:stable_sheavesI}
M.~Maruyama.
\newblock Moduli of stable sheaves. {I}.
\newblock {\em J. Math. Kyoto Univ.}, 17(1):91--126, 1977.

\bibitem[MR85]{MR85:P3_gaps_rank_two}
R.~M. Mir\'{o}-Roig.
\newblock Gaps in {C}hern classes of rank {$2$} stable reflexive sheaves.
\newblock {\em Math. Ann.}, 270(3):317--323, 1985.

\bibitem[MR87a]{MR87:P3_rank_three_chern_classes}
R.~M. Mir\'{o}-Roig.
\newblock Chern classes of rank {$3$} stable reflexive sheaves.
\newblock {\em Math. Ann.}, 276(2):291--302, 1987.

\bibitem[MR87b]{MR87:P3_rank_three_extremal_moduli}
R.~M. Mir\'{o}-Roig.
\newblock Stable rank {$3$} reflexive sheaves on {${\bf P}^3$} with extremal
  {$c_3$}.
\newblock {\em Math. Z.}, 196(4):537--546, 1987.

\bibitem[MS17]{MS17:lectures_notes}
E.~Macr\`i and B.~Schmidt.
\newblock Lectures on {B}ridgeland stability.
\newblock In {\em Moduli of curves}, volume~21 of {\em Lect. Notes Unione Mat.
  Ital.}, pages 139--211. Springer, Cham, 2017.

\bibitem[MS20]{MS20:space_curves}
E.~Macr\`i and B.~Schmidt.
\newblock Derived categories and the genus of space curves.
\newblock {\em Algebr. Geom.}, 7(2):153--191, 2020.

\bibitem[Mum63]{Mum63:quotients}
D.~Mumford.
\newblock Projective invariants of projective structures and applications.
\newblock In {\em Proc. {I}nternat. {C}ongr. {M}athematicians ({S}tockholm,
  1962)}, pages 526--530. Inst. Mittag-Leffler, Djursholm, 1963.

\bibitem[OS85]{OS85:spectrum_torsion_free_sheavesII}
C.~Okonek and H.~Spindler.
\newblock Das {S}pektrum torsionsfreier {G}arben. {II}.
\newblock In {\em Seminar on deformations (\L \'od\'z/{W}arsaw, 1982/84)},
  volume 1165 of {\em Lecture Notes in Math.}, pages 211--234. Springer,
  Berlin, 1985.

\bibitem[Sch20a]{Sch20:stability_threefolds}
B.~Schmidt.
\newblock Bridgeland stability on threefolds: some wall crossings.
\newblock {\em J. Algebraic Geom.}, 29(2):247--283, 2020.

\bibitem[Sch20b]{Sch20:rank_two_p3}
B.~Schmidt.
\newblock Rank two sheaves with maximal third {C}hern character in
  three-dimensional projective space.
\newblock {\em Mat. Contemp.}, 47:228--270, 2020.

\bibitem[Sim94]{Sim94:moduli_representations}
C.~T. Simpson.
\newblock Moduli of representations of the fundamental group of a smooth
  projective variety. {I}.
\newblock {\em Inst. Hautes \'Etudes Sci. Publ. Math.}, 79:47--129, 1994.

\bibitem[Tak72]{Tak72:Stability1}
F.~Takemoto.
\newblock Stable vector bundles on algebraic surfaces.
\newblock {\em Nagoya Math. J.}, 47:29--48, 1972.

\end{thebibliography}
\end{document}